\documentclass[11pt,a4paper]{amsart}
\usepackage[foot]{amsaddr}

\usepackage{caption}
 \usepackage{geometry}
\usepackage{amsmath,amsfonts,amssymb,amsthm,mathrsfs}
\usepackage{graphicx}
\usepackage{stmaryrd}
\usepackage{hyperref}
\usepackage{tikz}
\usetikzlibrary{decorations.pathreplacing,calligraphy}
\usepackage{dsfont}

\newtheorem{theorem}{Theorem}[section]
\newtheorem{lemma}[theorem]{Lemma}
\newtheorem{claim}[theorem]{Claim}
\newtheorem{proposition}[theorem]{Proposition}
\newtheorem{corollary}[theorem]{Corollary}
\newtheorem{assumption}[theorem]{Assumption}

\newtheorem{question}[theorem]{Question}

\newtheorem{definition}[theorem]{Definition}
\theoremstyle{remark}
\newtheorem{remark}[theorem]{Remark}

\numberwithin{equation}{section}

\newcommand{\R}{\mathbb{R}}
\newcommand{\Z}{\mathbb{Z}}
\def\P{\mathbb{P}} 
\newcommand{\E}{\mathbb{E}}

\newcommand{\id}{\mathds{1}}

\newcommand{\eps}{\varepsilon}
\newcommand{\cross}{\text{\textup{Cross}}}
\newcommand{\ann}{\text{\textup{Ann}}}
\newcommand{\arm}{\text{\textup{Arm}}}
\newcommand{\capa}{\text{\textup{Cap}}}
\newcommand{\tube}{\text{\textup{Tube}}}

\newcommand{\new}[1]{\textcolor{red}{#1}}

\begin{document}

\title[Percolation of strongly correlated Gaussian fields I. Subcritical decay]{Percolation of strongly correlated Gaussian fields I. Decay of subcritical connection probabilities}
\author{Stephen Muirhead$^1$}
\email{smui@unimelb.edu.au}
\address{$^1$School of Mathematics and Statistics, University of Melbourne}
\author{Franco Severo$^2$}
\address{$^2$ETH Z\"{u}rich}
\email{franco.severo@math.ethz.ch}
\begin{abstract}
We study the decay of connectivity of the subcritical excursion sets of a class of strongly correlated Gaussian fields. Our main result shows that, for smooth isotropic Gaussian fields whose covariance kernel $K(x)$ is regularly varying at infinity with index $\alpha \in [0, 1)$, the probability that $\{f \le \ell\}$, $\ell < \ell_c$, connects the origin to distance $R$ decays sub-exponentially in~$R$ at log-asymptotic rate $c_\alpha (\ell_c-\ell)^2 / K(R)$ for an explicit $c_\alpha > 0$. If $\alpha = 1$ and $\int_0^\infty K(x) dx = \infty$ then the log-asymptotic rate is $c_1 (\ell_c-\ell)^2 R (\int_0^R K(x) dx)^{-1}$, and if $\alpha > 1$ the decay is exponential.

Our findings extend recent results on the Gaussian free field (GFF) on $\Z^d$, $d \ge 3$, and can be interpreted as showing that the subcritical behaviour of the GFF is universal among fields with covariance $K(x) \sim c|x|^{d-2}$. Our result is also evidence in support of physicists' predictions that the correlation length exponent is $\nu = 2/\alpha$ if $\alpha \le 1$, and in $d=2$ we establish rigorously that $\nu \ge 2/\alpha$. More generally, our approach opens the door to the large deviation analysis of a wide variety of percolation events for smooth Gaussian fields.

This is the first in a series of two papers studying level-set percolation of strongly correlated Gaussian fields, which can be read independently.
\end{abstract}

\thanks{S.M.\ was supported by the Australian Research Council (ARC) Discovery Early Career Researcher Award DE200101467, and also acknowledges the hospitality of the Statistical Laboratory, University of Cambridge, where part of this work was carried out. F.S.\ was supported by the Swiss National Science Foundation, the NCCR SwissMAP and the European Research Council (ERC) under the European Union’s Horizon 2020 research and innovation program (grant agreement No 851565). The authors thank Alejandro Rivera and Hugo Vanneuville for helpful discussions on bootstrapping arguments and the physics literature on the correlation length exponent, Andriy Olenko for help with references on strongly correlated Gaussian fields, and two anonymous referees for detailed comments and corrections on a previous version.}
\date{\today}
\maketitle

\section{Introduction}
Let $f$ be a smooth centred stationary Gaussian field on $\mathbb{R}^d$, $d \ge 2$. We consider the probability that the excursion sets
\[ \{f \le \ell\} := \{ x \in \R^d : f(x) \le \ell \} \]
contain connected components of large diameter. For $R > 0$, define the centred Euclidean ball $B(R) = \{ x : |x| \le R\} \subset \R^d$, and denote by $\arm_\ell(R)$ the event that $\{f \le \ell\}$ contains a path between the origin and $\partial B(R)$. By monotonicity of $\{f \le \ell\}$ with respect to~$\ell$, there exists a critical level $\ell_c \in [-\infty, \infty]$ such that, as $R \to \infty$,
\[   \P[ \arm_\ell(R) ]  \to \begin{cases} \theta(\ell) =  0  & \text{if } \ell < \ell_c , \\  \theta(\ell) > 0 & \text{if } \ell > \ell_c , \end{cases} \]
and under certain conditions it is known that $\ell_c$ is finite \cite{ms83b}. It is further known that $\ell_c = 0$ for a very wide class of planar fields as a consequence of self-duality \cite{rv20, mv20, mrv20}, and also that $\ell_c < 0$ for certain fields if $d \ge 3$ \cite{drrv21}.

\smallskip
Our primary interest is the rate of decay of  $\P[ \arm_\ell(R) ] \to 0$ as $R \to \infty$ in the subcritical phase $\ell < \ell_c$. For fields with `weak' (i.e.~rapidly decaying) correlations it is known that $\P[ \arm_\ell(R) ]$ decays exponentially -- in particular, for positive and integrable covariances (with mild additional assumptions) \cite{mv20, s21}. This is analogous to the `sharp' phase transition that occurs in many percolation models with short-range dependence, including the archetype Bernoulli percolation \cite{men86,ab87}.

\smallskip
By contrast, much less is known about the decay of  $\P[ \arm_\ell(R) ]$ for fields with `strong' (i.e.~slowly decaying) correlations. In the physics literature it is predicted that such fields can lie outside the Bernoulli percolation universality class \cite{w84, ik91, jgrs20} (see Section~\ref{s:univ} for further discussion), and that $\P[ \arm_\ell(R)]$ may decay slower than exponentially.

\smallskip
 A motivating example is the Gaussian free field (GFF), which is the centred stationary Gaussian field on $\Z^d$, $d \ge 3$, with covariance given by the Green's function $G(x) \sim c_{G,d} |x|^{2-d}$, where $c_{G,d} > 0$ is a constant. For $d = 3$, it was recently shown \cite{grs21} that, for $\ell < \ell_c$,
\begin{equation}
\label{e:gff}
- \log  \P[ \arm_\ell(R) ]   \sim  \frac{(\ell_c-\ell)^2 R  }{4 c_{G,3} ( \log R)}   , 
  \end{equation}
where $\ell_c  = \ell_c(d) < 0$ is the critical level of the GFF, whereas in higher dimensions $d \ge 4$ the decay is exponential \cite{pr15}. Although the GFF is defined on $\Z^d$ rather than $\R^d$, heuristically this should be irrelevant when considering large-scale connectivity properties.

\smallskip
In this paper we address the decay of $\P[ \arm_\ell(R) ]$ for smooth isotropic Gaussian fields whose covariance kernel $K(x) = \E[f(0) f(x)]$ is regularly varying at infinity with index $\alpha \in [0, d)$, i.e.\ $K(x) = |x|^{-\alpha} L(|x|)$ where $L(x)$ is a slowly varying function (see~\eqref{e:sv} for a precise definition); the GFF corresponds to $\alpha = d-2$ and $L \sim c_{G,d}$. Our main result extends \eqref{e:gff} to a wide class of fields, and shows in particular that the behaviour of the GFF is universal among fields whose correlations are asymptotic to the Green's function (see Section \ref{r:gff}).

\smallskip
While the analysis in \cite{grs21} relied on specific properties of the GFF (e.g.\ harmonicity, the Markov property, the connection to random walks etc.), here we develop a robust approach that, in principle, relies only on isotropy and the regular variation of the covariance (although we use extra assumptions for technical reasons). In this we are guided by two main ideas:
\begin{itemize}
\item The notion of \textit{(harmonic) capacity} of a set $D \subset \Z^d$, central to the analysis of the GFF, can be replaced by the \textit{(generalised) capacity} of a set $D \subset \R^d$ (with respect to the kernel $K$), defined equivalently as either
   \begin{equation}
\label{e:cap2}
 \capa_K(D)  = \Big( \min_{ \mu \in \mathcal{P}(D) } \int_D \int_D K(x-y) d\mu(x) d\mu(y) \Big)^{-1} ,
 \end{equation}
 where $\mathcal{P}(D)$ is the set of (Borel) probability measures on $D \subset \R^d$, or as 
\begin{equation}
\label{e:cap1}
   \capa_K(D) = \min \big\{ \|h\|^2_H : h \in H, h \ge 1 \text{ on } D \big\} , 
   \end{equation}
where $H$ is the \textit{reproducing kernel Hilbert space (RKHS)} associated to the kernel $K$; see Section \ref{s:cap1} for a detailed introduction. In the case of regularly varying covariance, the asymptotic capacity of scaled sets can be computed explicitly (see Section \ref{s:cap2}).
\item The decomposition of the GFF into a harmonic average and a `local' field, which allows for renormalisation arguments, can be replaced by a generalised notion of \textit{local-global decomposition} (see Section \ref{s:locglo}).
\end{itemize}
Although we restrict our attention to smooth isotropic fields, we believe this approach could be adapted to Gaussian fields on $\Z^d$ whose covariance is \textit{asymptotically} isotropic (e.g.\ the~GFF).

\smallskip
To explain the relevance of the capacity, standard heuristics suggest that the asymptotics of $\P[ \arm_\ell(R) ] $ are carried, in a large deviation sense, by the event that the field has excess mean $\ell_c - \ell$ in a neighbourhood of a path connecting $0$ to $\partial B(R)$, so that the excursion set $\{f \le \ell\}$ `behaves supercritically' in this neighbourhood. Since the large deviation rate function for a Gaussian field to have excess mean $h \in H$ is $\frac{1}{2}\|h\|_H^2$, and assuming that the excess mean is easiest to achieve on a line segment, this leads one to suppose that
\begin{equation}
\label{e:result}   - \log  \P[ \arm_\ell(R) ]  \approx  \frac{ \min \big\{ \|h\|^2_H : h \in H, h \ge \ell_c - \ell \text{ on } [0, R] \big\}}{2}   = \frac{(\ell_c - \ell)^2 \capa_K([0,R])}{2}  . 
\end{equation}
It is not hard to obtain the lower bound in \eqref{e:result} using an entropic bound; the main difficulty lies in obtaining a matching upper bound, for which we use renormalisation arguments. In the case $\alpha \le 1$, we obtain \eqref{e:result} as an asymptotic equivalence. By contrast, if $\alpha > 1$ we obtain \eqref{e:result} only up to multiplicative error; this is quite natural because in this regime one expects many paths from $0$ to $\partial B(R)$ to contribute to $ - \log  \P[ \arm_\ell(R) ]$.

\smallskip
Beyond the study of $\arm_\ell(R)$, our approach opens the door to the analysis of large deviation and entropic repulsion phenomena for a wide variety of percolation-type events for strongly correlated Gaussian fields in both subcritical and supercritical phases (e.g.\ disconnection, excess density, large finite clusters etc.), which have received a lot of attention in the context of the GFF \cite{Szn15,Nit18,CN20a,Szn19b,PS21a} (or the related model of random interlacements \cite{Szn17,NS20,Szn19a,Szn20,Szn21}). We stress that, unlike the aforementioned works, our approach does not require any kind of Markov property or connection to random walks.

\smallskip
We conclude the introduction with some notation. Henceforth all Gaussian fields and vectors are centred without further mention. A function $g : \R^d \to \R$ is \textit{isotropic} if $g(x)$ depends only on $|x|$, where $|\cdot|$ is the Euclidean norm; we often abuse notation by identifying isotropic functions with their restriction to $\R^+$. A Gaussian field is \textit{isotropic} if its covariance kernel is isotropic. For functions $g,h$ we write $g(x) \sim h(x)$, $g(x) = o(h(x))$, $g(x) = O(h(x))$, and $g(x) \asymp h(x)$ respectively to indicate that, as $x \to \infty$, $g(x)/h(x) \to 1$, $|g(x)| / |h(x)| \to 0$, $|g(x)|/|h(x)|$ is bounded away from $\infty$, and $g(x)/h(x)$ is bounded away from $0$ and $\infty$. When the context is clear, we sometimes equate a set $D \subseteq \R^k$, $k < d$, with the set $D \times \{0\}^{d-k} \subseteq \R^d$.

\subsection{A first description of our results}
We begin by stating our results for an explicit one-parameter family of Gaussian fields in $d=2$; this illustrates our results while avoiding the need for extra assumptions. We state general results in the next section. 

\smallskip
For $\alpha > 0$, let $F_\alpha$ be the smooth isotropic Gaussian field on $\R^2$ with covariance
\begin{equation}
\label{e:cauchy}
 K(x) = (1+x^2)^{-\alpha/2} \sim x^{-\alpha},
 \end{equation}
 known as the \textit{Cauchy kernel}; the critical level of $F_\alpha$ is $\ell_c = 0$ \cite{mv20, mrv20}.

\smallskip
Our main result characterises how the decay of $\P[\arm_\ell(R)]$, $\ell <  0$, depends on $\alpha$:

\begin{theorem}[Decay of subcritical connection probabilities]
\label{t:spec}
Let $f = F_\alpha$, $\alpha \in (0,2)$, on $\R^2$. Then for every $\ell < \ell_c=0$, as $R \to \infty$:
\begin{enumerate}
\item If $\alpha \in (1,2)$,
\begin{equation}
\label{e:spec1}
-  \log  \P[ \arm_\ell(R) ]   \asymp  R   . 
  \end{equation}
\item If $\alpha = 1$,
\begin{equation}
\label{e:spec2}
 - \log  \P[\arm_\ell(R)]   \sim \frac{ \ell^2 R}{4 (\log R) }  .  
  \end{equation}
\item If $\alpha \in (0, 1)$,
\begin{equation}
\label{e:spec3}
  - \log  \P[\arm_\ell(R)]   \sim   \frac{c_\alpha \ell^2 R^\alpha}{2}  
   \end{equation}
where 
\begin{equation}
\label{e:calpha}
 c_\alpha = \frac{1}{\pi}  \textrm{B} \Big( \frac{1+\alpha}{2}, \frac{1+\alpha}{2} \Big) \cos \Big( \frac{\pi \alpha}{2} \Big)\in (0, 1)  ,
 \end{equation}
 and $\textrm{B}(x,y)$ is the Beta function.
 \end{enumerate}
\end{theorem}

\begin{remark}
We have omitted the case $\alpha \ge 2$ from Theorem \ref{t:spec} for technical reasons, although previous work has shown that \eqref{e:spec1} also holds in this case \cite{mv20}.
\end{remark}

As a corollary, we deduce the order of the largest cluster inside~$B(R)$. For $R > 0$ and $\ell \in \R$, let $D_{R, \ell} $ denote the largest diameter among the connected components of $\{ f \le \ell \} \cap B(R)$.

\begin{theorem}[Diameter of the largest cluster]
\label{t:diamspec}
Let $f = F_\alpha$, $\alpha \in (0,2)$, on $\R^2$. Then for every $\ell < \ell_c = 0$, as $R \to \infty$:
\begin{enumerate}
\item If $\alpha \in (1, 2)$,
\[ \frac{  D_{R, \ell} }{   \log R  } \quad  \text{ is bounded above and away from zero in probability.} \]
\item If $\alpha = 1$,
 \[  \frac{  D_{R, \ell} }{   \log R \log \log R }   \to    \frac{8}{\ell^2}   \quad \text{ in probability.}  \]
\item If $\alpha \in (0, 1)$, 
\[    \frac{  D_{R, \ell} }{  (  \log R )^{1/\alpha} }   \to  \Big( \frac{4}{c_\alpha \ell^2} \Big)^{1/\alpha}  \quad \text{ in probability,}  \]
where $c_\alpha$ is defined as in \eqref{e:calpha}. 
\end{enumerate}
\end{theorem}

Our results also cover fields with slower-than-polynomial correlation decay (i.e.\ the $\alpha = 0$ case). Although we lack explicit examples, they are straightforward to construct indirectly via convolution. Fix $\gamma > 0$ and let $q$ be a smooth isotropic unimodal function on~$\R^2$ such that 
\[ q(x) \propto  x^{-1} (\log x)^{-(\gamma+1)/2}   , \]
for all $x \ge 2$. Then let $F_0^\gamma$ be the smooth isotropic Gaussian field on $\R^2$ with covariance $K = q \star q$, where $\star$ denotes convolution. As we show in Appendix \ref{a:tau}, we can and will normalise $q$ so that $K(x) \sim  (\log x)^{-\gamma}$ as $x \to \infty$. The critical level of $F_0^\gamma$ is also $\ell_c = 0$ \cite{mrv20}.

\begin{theorem}[Decay of subcritical connection probabilities; $\alpha = 0$ case]
\label{t:spec0}
Let $f = F_0^\gamma$, $\gamma > 0$, on $\R^2$. Then for every $\ell < \ell_c = 0$, as $R \to \infty$,
\begin{equation}
\label{e:spec4}
 - \log  \P[ \arm_\ell(R) ]   \sim   \frac{ \ell^2 ( \log R)^{\gamma}}{2}  .
  \end{equation}
\end{theorem}

\begin{remark}
For technical reasons we do not prove the analogue of Theorem \ref{t:diamspec} for $F_0^\gamma$.
\end{remark}

\subsection{Discussion, extensions, and conjectures}

\subsubsection{The role of the capacity and the constant $c_\alpha$}
\label{sss:cap}
We could unify \eqref{e:spec2}, \eqref{e:spec3} and \eqref{e:spec4} by stating them collectively as
\begin{equation}
\label{e:capresult}
- \log \P[\arm_\ell(R)]  \sim   \frac{ \ell^2 \capa_K([0, R]) }{2} , 
\end{equation}
where $\capa_K$ is defined in \eqref{e:cap2}--\eqref{e:cap1}. Indeed it can be shown (see Proposition~\ref{p:capline}) that
\begin{equation}
\label{e:capasym}
\capa_K([0, R])  \sim  \begin{cases}  R /( 2 \int_0^R K(x) dx ) \sim  R/ (2 \log R)    & \text{if } \alpha = 1 , \\   c_\alpha / K(R)  &   \text{if } \alpha \in (0, 1) , \\ 1/K(R) & \text{if } \alpha = 0.  \end{cases} 
\end{equation}

\smallskip
In the case $\alpha \in (0, 1)$, one arrives at \eqref{e:capasym} by applying the spatial rescaling $x \mapsto x/R$ in \eqref{e:cap2}, which by regular variation leads to
\begin{equation}
\label{e:calpha2}
 \capa_K([0, R])  K(R) \sim  \Big(  \min_{ \mu \in \mathcal{P}([0,1]) } \int_0^1 \int_0^1 |x-y|^{-\alpha} d\mu(x) d\mu(y) \Big)^{-1}  =:    c_\alpha .
  \end{equation}
By the classical theory of the $\alpha$-Riesz kernel $K_\alpha(x, y) = |x-y|^{-\alpha}$ (see \cite[Section II.3.13 p.163--164 and Appendix p.399--400]{lan72}) the minimum in \eqref{e:calpha2} is achieved for a beta-distributed $\mu$ with density proportional to $(x(1-x))^{(\alpha-1)/2}$, which explains the form of \eqref{e:calpha}. 

\smallskip
In the case $\alpha = 0$ the kernel homogenises under this rescaling, leading to 
\[ \capa_K([0, R])  K(R) \sim   \Big( \min_{ \mu \in \mathcal{P}([0,1]) } \int_0^1 \int_0^1 d\mu(x) d\mu(y) \Big)^{-1}  =  1  . \]
In the case $\alpha = 1$ the optimal probability measure $\mu \in \mathcal{P}([0, R])$ in~\eqref{e:cap2} homogenises as $R \to \infty$, i.e.\ one can approximate the capacity by setting $\mu$ to be the normalised Lebesgue measure, which leads to the estimate
\begin{equation}
\label{e:capa1}
 \capa_K([0, R])  \sim  R^2 \Big( \int_0^R \int_0^R K(x-y) dx dy \Big)^{-1}   \sim R \Big( 2 \int_0^R K(x) dx \Big)^{-1} . 
 \end{equation}
This homogenisation also occurs if $\alpha > 1$, or more generally if $\int_0^\infty K(x) dx < \infty$, where one similarly has $\capa_K([0, R]) \sim R ( 2 \int_0^\infty K(x) dx )^{-1}$, so in this case \eqref{e:spec1} can be written as
\[ - \log \P[\arm_\ell(R)]  \asymp  \capa_K([0, R]) .  \]

\subsubsection{Comparison with the GFF}
\label{r:gff}
The decay \eqref{e:spec2} in the $\alpha = 1$ case matches precisely the known decay for the GFF on $\Z^3$ in \eqref{e:gff} (up to an $o(1)$ term in the exponent); this can be interpreted as showing that sub-exponential decay for the GFF on $\Z^3$ is universal for fields whose covariance is asymptotic to the Green's function $G(x) \sim c_{G,3} |x|^{-1}$. Moreover, \eqref{e:spec1} shows that the exponential decay for the GFF in dimensions $d \ge 4$ is universal for fields whose covariance has regularly varying decay with index $\alpha > 1$.\

\smallskip
As mentioned above, we believe this universality extends to discrete stationary Gaussian fields on $\Z^d$ whose covariance is asymptotically isotropic and regularly varying.

\subsubsection{The logarithmically-correlated case} 
Theorem \ref{t:spec0} demonstrates that, under suitable assumptions, if $K(x) \sim (\log x)^{-1}$ then
 \[    \P[ \arm_\ell(R) ]   = R^{ -\ell^2 / 2  + o(1)} \ , \quad \ell <  0 , \]
i.e.\ there is power law decay of connectivity throughout the subcritical phase with varying exponent. This contrasts markedly with the behaviour of fields with weak correlations \cite{bg17, rv20, mv20, s21}. We are not aware of any other percolation model that exhibits this behaviour (although long-range percolation models typically have power-law decay in the subcritical phase \cite{dcrt20}, usually the exponent does not vary).

\subsubsection{Other connection events}
\label{s:other}
Aside from $\arm_\ell(R)$ one could consider other connection events, for instance the event $\cross_\ell(R)$ that the square $[0, R]^2$ is crossed by $\{f \le \ell\}$ from left-to-right, i.e.\ $\{ f \le \ell \} \cap [0, R]^2$ contains a path that intersects $\{0\} \times [0, R]$ and $\{R\} \times [0, R]$. For this event the conclusion of Theorem \ref{t:spec} is true with essentially no change to the proof.

\smallskip
More generally, one can consider quad-crossing events $\cross^Q_\ell(R)$ for a quad $Q = (D; \gamma_1, \gamma_2)$ where $D \subset \R^2$ is a compact domain and $\gamma_1, \gamma_2$ are disjoint boundary arcs on $\partial D$; this is the event that $\{f \le \ell\} \cap RD$ contains a path that intersects $R \gamma_1$ and $R \gamma_2$. While it does not follow immediately from the proof of Theorem \ref{t:spec}, in this case we expect that
\[ - \log \P[\cross^Q_\ell(R)] \sim   \begin{cases}   \frac{ \tilde{c}_Q \ell^2 R}{4 ( \log R)}   & \text{if } \alpha = 1, \\  \frac{ \tilde{c}_{Q,\alpha} \ell^2 R^\alpha}{2}    &   \text{if } \alpha \in (0, 1) , \end{cases} \]
where $\tilde{c}_Q$ is the length of the shortest path in $D$ between $\gamma_1$ and $\gamma_2$, and $\tilde{c}_{Q,\alpha}$ is the minimal capacity, with respect to the $\alpha$-Riesz kernel, among all paths in $D$ between $\gamma_1$ and $\gamma_2$.

\smallskip
As discussed above, in addition to subcritical connection events one can also ask about large deviations for other percolation events, for instance the event that the connected component of $\{f \le \ell\}$ that contains the origin is large and bounded in the supercritical phsae (as was analysed in \cite{grs21} for the GFF).

\subsubsection{Universality classes and critical exponents}
\label{s:univ}
In the physics literature it is predicted that isotropic planar Gaussian fields with covariance $K(x) \sim x^{-\alpha}$, $\alpha \in (0, \infty)$, lie inside the Bernoulli percolation universality class if $\alpha > 3/2$, whereas if $\alpha < 3/2$ the model lies inside a distinct, $\alpha$-dependent, universality class (see \cite{w84, ik91, jgrs20}).

\smallskip
One manifestation of this is the conjectured $\alpha$-dependency of the \textit{critical exponents} that govern the behaviour of the connectivity at criticality $\ell = \ell_c$ and near criticality $\ell \approx \ell_c$. For example, recalling the square-crossing event $\cross_\ell(R)$ from Section \ref{s:other}, for $\ell < \ell_c = 0$ one can define the \textit{correlation length} to be
\begin{equation}
\label{e:xi}
  \xi(\ell) = \inf \{ R \ge 0 :  \P[ \cross_\ell(R)] < \eps  \} , 
  \end{equation}
where $\eps \in (0,1/2)$ is chosen arbitrarily. The \textit{correlation length exponent} $\nu$ is defined as the constant satisfying $\xi(\ell) = \ell^{-\nu + o(1)}$ as $\ell \uparrow \ell_c = 0$, if such a constant exists.
 
 \smallskip
For planar models in the Bernoulli percolation universality class it is believed, and in a few cases known rigorously \cite{sw01}, that $\nu = 4/3$. For smooth Gaussian fields with covariance $K(x) \sim x^{-\alpha}$, $\alpha \in (0, \infty)$, it has been conjectured \cite{w84, ik91, jgrs20} that
\[ \nu = \nu(\alpha) = \begin{cases} 4/3 & \text{if } \alpha \ge 3/2 , \\  2/\alpha  & \text{if } \alpha \le 3/2. \\ \end{cases} \]
This has been recently proved \cite{DPR21} for the special case of the GFF on the cable system of certain transient graphs. We believe that our findings are evidence in support of this conjecture, at least in the case $\alpha \le 1$. Indeed, if one considers the analogue of \eqref{e:spec3} for the square-crossing event $\cross_\ell(R)$, in the case $\alpha \le 1$ it states that, for $\ell < \ell_c$, as $R \to \infty$,
\begin{equation}
\label{e:cross}
  - \log \P[ \cross_\ell(R) ]   =   \frac{c' \ell^2 R^\alpha (1 + o(1)) }{(\log R)^{\id_{\alpha = 1}}}    .
  \end{equation}
for some constant $c' > 0$ depending only on $\alpha$. Although formally \eqref{e:cross} is valid only as $R \to \infty$ and for fixed $\ell < \ell_c$ (whereas the correlation length exponent $\nu$ captures the behaviour of $ \P[ \cross_\ell(R) ] $ simultaneously as $R \to \infty$ and $\ell \uparrow 0$), it strongly suggests that $\nu = 2/\alpha$ if $\alpha \le 1$; indeed if the $o(1)$ term is neglected it implies precisely this. In fact, adapting our proof of the lower bound in \eqref{e:cross} we can rigorously establish the lower bound $\nu \ge 2/\alpha$:

\begin{proposition}[Lower bound on the correlation length exponent]
\label{p:cl}
Let $f = F_\alpha$, $\alpha \in (0,1]$, on $\R^2$, as in Theorem~\ref{t:spec}. Fix $\eps \in (0, 1/(2e))$, and define $\xi(\ell)$ as in \eqref{e:xi}. Then
\[  \liminf_{\ell \uparrow \ell_c} \frac{ - \log \xi(\ell) }{\log \ell}  \ge \frac{2}{\alpha} . \]
In particular, $\nu \ge 2/\alpha$ if it exists.
\end{proposition}

\vspace{0cm}
\subsubsection{Beyond the regularly-varying case}
While our proof depends strongly on properties of regularly varying kernels, we believe that the decay of $\P[\arm_\ell(R)]$  is  governed by the asymptotics of the capacity $\capa_K([0, R])$ in general.

\begin{question}
\label{c}
Is the following true for every smooth ergodic isotropic Gaussian field $f$ on~$\mathbb{R}^2$? For every $\ell < 0$, as $R \to \infty$,
\[  - \log \P[\arm_\ell(R)]   \asymp \capa_K([0, R])  . \] 
More precisely: 
\begin{enumerate}
\item If $ \lim_{R \to \infty} \int_0^R K(x) dx = \infty$, then \eqref{e:capresult} holds.
\item If $\int_0^\infty K(x) dx < \infty$, then $\P[\arm_\ell(R)]$ has exponential decay.
\end{enumerate}
\end{question}

Let us briefly consider the important case of the random plane wave (RPW) for which $K(x) = J_0(|x|)$, where $J_0$ is the zeroth Bessel function. Although $J_0(|x|)$ is not absolutely integrable on $\R^d$, due to oscillations one has $\int_0^\infty J_0(x) dx < \infty$. Hence we expect $\P[\arm_\ell(R)]$ to decay exponentially, which is consistent with physicists' predictions that the RPW lies in the Bernoulli percolation universality class \cite{bs07}. Note that the oscillations of $J_0$ mean that RPW does not fall within our regularly-varying setting.

\subsection{The general result}
We now present a generalisation of Theorems \ref{t:spec} and \ref{t:diamspec} to a wider class of fields in all dimensions $d \ge 2$. 

\smallskip
Recall that a continuous strictly positive function $L : \R^+ \to \R^+$ is said to be \textit{slowly varying (at infinity)} if, for all $a > 0$,
\begin{equation}
\label{e:sv}
\lim_{x \to \infty} \frac{L(ax)}{L(x)} = 1 . 
\end{equation}
We say that a continuous isotropic function $h: \mathbb{R}^d \to \mathbb{R}$ is \textit{regularly varying (at infinity) with index} $\alpha \in \R$ if $h(x) \sim x^{-\alpha} L(x)$, for $L$ slowly varying. While $h$ may not be positive, necessarily it is \textit{eventually} positive. We collect basic properties of regularly varying functions in Appendix \ref{a:reg}, and refer to \cite{bgt87} for a detailed treatment.

\smallskip
The \textit{spectral measure} of a continuous stationary Gaussian field $f$ on $\R^d$ is the finite measure $\nu$ defined via 
\[K(x) =  \mathcal{F}[\nu](x) = \int_{\R^d}  e^{2\pi i (x, \lambda) } d\nu(\lambda) , \] 
where $\mathcal{F}[\cdot]$ denotes the Fourier transform. The field $f$ will be called \textit{regular} if it is a.s.\ $C^2$-smooth and its spectral measure contains an open set in its support. In particular this implies that $\{f = \ell\}$ a.s.\ consists of a collection of hypersurfaces (see \cite[Corollary A.3]{m22}).

\smallskip
Our main assumption is the following:
\begin{assumption}
\label{a:rv}
$f$ is a continuous regular isotropic Gaussian field on $\mathbb{R}^d$, $d \ge 2$, whose covariance kernel $K$ is regularly varying with index $\alpha \in [0, d)$, i.e.\ there exists a slowly varying function $L$ such that, as $x \to \infty$,
\begin{equation}
\label{a:k}
K(x) \sim x^{-\alpha} L(x)  .
\end{equation}
If $\alpha = 0$ or $\alpha = 1$ we further assume there is a $\gamma \in \R$ such that, as $x \to \infty$,
 \[      L(x) \sim  (\log x)^{-\gamma} ,  \]
 where $\gamma < 1$ if $\alpha =1$ and $\gamma > 0$ if $\alpha = 0$.
\end{assumption}

In addition, our main results assume the existence of an appropriate local-global decomposition. To keep the discussion simple we defer the precise definition to Section \ref{s:locglo} (see Definition \ref{d:locglo}), but roughly speaking this requires that, for any $L \ge 1$, $f$ may be decomposed as $f = f_L + g_L$, where $f_L$ is $L$-range dependent, and $g_L$ `carries' the covariance structure on scales $\gg L$. We refer to $f_L$ and $g_L$ as the \textit{local field} and the \textit{global field} respectively, and we emphasise that $f_L$ and $g_L$ need not be independent (and indeed will not be in our setting). 

\smallskip
In Section \ref{s:locglo} we present quite general sufficient conditions for this decomposition to exist. These cover, among other examples, the Gaussian fields with Cauchy kernel \eqref{e:cauchy} in all dimensions. It is plausible that a local-global decomposition \textit{always} exists under Assumption~\ref{a:rv}.

\smallskip
To state our results we also need to introduce certain variants of the critical level, $\ell^\ast_c$ and $\ell^{\ast \ast}_c$, which provide a priori control in the subcritical and supercritical phases respectively. More precisely, $\ell < \ell^\ast_c$ guarantees that the probability of `annuli crossings' decays on large scales, whereas $\ell > \ell^{\ast \ast}_c$ ensures the existence of a crossing in a thin `tube'. 

\smallskip
For $R > 0$, let $A(R) \subset \R^d$ denote the annulus $B(2R) \setminus B(R)$, and let $\ann_\ell(R)$ denote the event that $\{f \le \ell\} \cap A(R)$ contains a path that intersects $\partial B(R)$ and $\partial B(2R)$. For $R > 0$ and $s > 0$, let $T(R; \rho)$ denote the `tube' $[0, R] \times [0, R^\rho]^{d-1}$, and let $\tube_\ell(R; \rho)$ denote the event that $T(R; \rho)$ is crossed by $\{f \le \ell\}$ from left-to-right, i.e.\ $\{ f \le \ell \} \cap T(R;\rho)$ contains a path that intersects $\{0\} \times [0, R^\rho]^{d-1}$ and $\{R\} \times [0, R^\rho]^{d-1}$. 
Then define
\[ \ell^\ast_c = \sup \{ \ell :  \lim_{R \to \infty} \P[ \ann_\ell(R)] = 0 \} \quad \text{and} \quad \ell^{\ast \ast}_c = \inf \{ \ell :  \forall \rho > 0 ,  \lim_{R \to \infty} \P[ \tube_\ell(R; \rho)] = 1   \} . \]

As a preliminary we establish basic properties of these levels including their finiteness; these follow from combining sprinkled bootstrapping arguments, similar to those in \cite{rs13,pt15,pr15}, with results in \cite{mrv20}:

\begin{proposition}
\label{p:nontriv}
Let $f$ satisfy Assumption \ref{a:rv} and suppose that $f$ has a local-global decomposition in the sense of Definition \ref{d:locglo}. Then
\[ \ell^\ast_c \le \ell_c \le 0  \quad \text{and} \quad \ell^\ast_c \le \ell^{\ast \ast}_c. \]
 Moreover if either $\alpha > 0$ or $\alpha = 0$ and $\gamma > 1$, then
 \[ \ell^\ast_c  > -\infty .\]
 Finally, if $d = 2$ then
  \[  \ell^\ast_c = \ell_c = 0 . \]
\end{proposition}
\begin{remark}\label{r:ell^**}
As a consequence of Theorem \ref{t:gen} below we will later deduce (see Remark \ref{r:lastast}) that, if either $\alpha > 0$ or $\alpha = 0$ and $\gamma > 1$, then $\ell^{\ast \ast}_c \le 0$, and hence
\[  \ell_c, \ell^\ast_c, \ell^{\ast \ast}_c \in (-\infty, 0]  .  \]
\end{remark}

We believe that the critical levels $\ell_c, \ell^\ast_c, \ell^{\ast \ast}_c $ coincide for all fields considered in this paper, as they do for Bernoulli percolation \cite{gm90, gr99}, however so far this is only known if $d=2$ and $\alpha > 0$ (where they are zero). In a companion paper \cite{m22}, the equality $\ell_c = \ell^\ast_c$ is established for a wide class of strongly correlated Gaussian fields in $d \ge 2$, which partially overlaps with those considered in this paper. So far, the full equality $\ell_c = \ell^\ast_c = \ell^{\ast \ast}_c$ in $d \ge 3$ is only known for fields with positive and rapidly decaying correlations (i.e.\ the case $\alpha > d)$ \cite{s21}, and for the GFF on $\Z^d$ \cite{dgrs20}.

\smallskip
We are now ready to state our main result in full generality. For $0 \le r \le R$, let $\arm_\ell(r,R)$ denote the event that $\{f \le \ell\}$ contains a path between $\partial B(r)$ and $\partial B(R)$, recalling that $\arm_\ell(R) = \arm_\ell(0, R)$. Recall also the constant $c_\alpha$ defined in \eqref{e:calpha}.

\begin{theorem}
\label{t:gen}
Let $f$ satisfy Assumption \ref{a:rv} and suppose that $f$ has a local-global decomposition in the sense of Definition \ref{d:locglo}. Then for every $\ell < \ell^\ast_c$, as $R \to \infty$:
\begin{enumerate}
\item If  $\alpha \in (1, d)$ and $K \ge 0$,
\begin{equation}
\label{e:gen1}
  \log \P[ \arm_\ell(1,R) ]   \asymp  -R .  
  \end{equation}
\item If $\alpha = 1$ (and recall that $\gamma < 1$),
\begin{equation}
\label{e:gen2}
 -\frac{(\ell^{\ast \ast}_c-\ell)^2  (1 + o(1)) }{(4/(1-\gamma)) K(R) (\log R)}   \le  \log \P[ \arm_\ell(1,R) ]   \le  -\frac{(\ell^{\ast}_c-\ell)^2  (1 + o(1)) }{ (4/(1-\gamma)) K(R) (\log R) }     .  
  \end{equation}
\item If $\alpha \in (0,1)$,
\begin{equation}
\label{e:gen3}
 - \frac{c_\alpha  (\ell^{\ast \ast}_c - \ell)^2  (1 + o(1))}{2 K(R)}    \le   \log \P[ \arm_\ell(1,R) ]   \le -  \frac{c_\alpha (\ell^{\ast}_c - \ell)^2 (1 + o(1))}{2 K(R)}    . 
   \end{equation}
\item If $\alpha = 0$, then \eqref{e:gen3} holds with $c_\alpha = 1$ and $\ell_c$ in place of $\ell_c^{\ast \ast}$.
\end{enumerate}
If in addition $K \ge 0$, then these results hold for $\arm_\ell(R)$ in place of $\arm_\ell(1, R)$.
\end{theorem}

We can also deduce the analogue of Theorem \ref{t:diamspec}, but for simplicity we state this only in the case that $K(x) \sim x^{-\alpha} (\log x)^{-\gamma}$. Recall that $D_{R,\ell}$ denotes the largest diameter among the connected components of $\{f \le \ell\} \cap B(R)$. A sequence of events $(A_R)_{R > 0}$ is said to hold \textit{with high probability (w.h.p.)} if $\P[A_R] \to 1$ as $R \to \infty$.

\begin{theorem}
\label{t:diamgen}
Let $f$ satisfy Assumption \ref{a:rv} with $L(x) = (\log x)^{-\gamma}$, $\gamma \in \R$, and suppose that $f$ has a local-global decomposition in the sense of Definition \ref{d:locglo}. Then for every $\ell < \ell^\ast_c$, as $R \to \infty$:
\begin{enumerate}
\item If $\alpha \in (1, d)$ and $K \ge 0$,
\[ \frac{  D_{R, \ell} }{   \log R  } \quad  \text{is bounded above and away from zero in probability.} \]
\item If $\alpha = 1$ (and recall that $\gamma < 1$), for every $\eps > 0$,
 \[  \frac{  D_{R, \ell} }{   \log R (\log \log R)^{1-\gamma} }   \in  \Big(  \frac{4d(1-\eps)}{(1-\gamma)(\ell^{\ast \ast}_c - \ell)^2}  , \frac{4d(1+\eps)}{(1-\gamma)(\ell^{\ast}_c - \ell)^2}  \Big)   \quad \text{w.h.p.}  \]
\item If $\alpha \in (0, 1)$ (and recall that $\gamma > 0$), for every $\eps > 0$,
\[    \frac{  D_{R, \ell} }{  (  \log R )^{1/\alpha} (\log \log R)^{-\gamma/\alpha} }   \in  \Big( \Big( \frac{2d(1-\eps)}{c_\alpha (\ell^{\ast \ast}_c - \ell)^2} \Big)^{1/\alpha} ,   \Big( \frac{2d(1+\eps)}{c_\alpha (\ell^{\ast}_c - \ell)^2} \Big)^{1/\alpha} \Big) \quad \text{w.h.p.}   \]
\end{enumerate}
\end{theorem}

\begin{remark}
\label{r:lastast}
A corollary of Theorem \ref{t:gen} is that, if either $\alpha > 0$ or $\alpha = 0$ and $\gamma > 1$, then $\ell^{\ast \ast}_c \le 0$. To see this, fix $\ell > 0$ and $\rho > 0$ and observe that  $\tube_\ell(R; \rho)^c$ is implied by the event that $\{f|_{\R^2 } \ge \ell \}$ contains a path connecting $[0, R] \times \{0\} $ and $[0, R] \times \{R^\rho\} $. Hence by the union bound and the symmetry in law of $f$ and $-f$, for any $R \ge 1$,
\[ \P[\tube_\ell(R;\rho) ] \ge  1 -   R \times \P[ f|_{\R^2 } \in \arm_{-\ell}(1,R^\rho) ] .\]
 By the upper bound of Theorem~\ref{t:gen} applied in the case $d=2$ (recall that $\ell_c^\ast = 0$ in that case), $\P[ f|_{\R^2 } \in \arm_{-\ell}(1,R^\rho) ] $ decays faster than any polynomial, which gives the claim.

\smallskip
A consequence is that, if $d=2$, the conclusions of Theorem \ref{t:gen} and \ref{t:diamgen} can be simplified (as in Theorems \ref{t:spec} and \ref{t:diamspec}) since $ \ell^\ast_c = \ell_c = 0 $ in these cases, and moreover $\ell^{\ast \ast}_c = 0$ if $\alpha > 0$ by the previous remark.
\end{remark}

\begin{remark}
In the case $\alpha \in (0, 1)$ we could relax the definition of $\ell_c^{\ast \ast}$ by only demanding that the `tubes' be of the form $[0, R] \times [0, \rho R]^{d-1}$ rather than  $[0, R] \times [0, R^\rho]^{d-1}$ (see Remark~\ref{r:tube}), although this would be insufficient in the case~$\alpha = 1$. The fact that if $\alpha = 0$ we can replace $\ell_c^{\ast \ast}$ with $\ell_c$ is a consequence of the capacities of $[0, R]$ and~$B(R)$ being asymptotically equal in this case.
 \end{remark}

\begin{remark}
In the cases $\alpha = 0$ and $\alpha = 1$ our proofs can be adapted to other classes of slowly varying $L$ that are not poly-logarithmic, for instance $L(x) \sim (\log \log x)^{-\gamma}$. We have not formally considered these cases for simplicity.
\end{remark}

\begin{remark}[On positive associations]
\label{r:pa}
Except for the lower bound in the case $\alpha >1$, Theorem \ref{t:gen} does \textit{not} require that $K \ge 0$, so positive associations (i.e.\ the FKG inequality) does not necessarily hold \cite{pit82}.

\smallskip
In the case $\alpha > 1$ we assume $K \ge 0$ only so that, for $R \ge 1$ and $\delta > 0$,
\[ \P[ \arm_\ell(R) ] \ge \P\Big[ \sup_{x \in [0, R]} f(x) \le \ell \Big]   \ge \P\Big[ \sup_{x \in [0, \delta]} f(x) \le \ell \Big]^{R/\delta+1} \ge e^{-cR} ,\]
for a $c = c(f,\ell,\delta)  > 0$, where the second step used positive associations, and the third step the fact that, by continuity, $\P[ \sup_{x \in [0, \delta]} f(x) \le \ell ] > 0$ for sufficiently small $\delta = \delta(f,\ell) > 0$.

\smallskip
Observe also that Theorem \ref{t:gen} is stated, in general, for $\arm_\ell(1, R)$ rather than $\arm_\ell(R)$. If $K \ge 0$ these events are comparable, since by positive associations
\[\P[ \arm_\ell(R)] \le \P[ \arm_\ell(1, R) ] \le \frac{  \P[\arm_\ell(R)] }{ \P[ \sup_{x \in B(1)} f(x) \le \ell  ]} = c  \,   \P[\arm_\ell(R)]  \]
for some $c = c(f,\ell) > 0$, but they may not be comparable in general. Positive associations will play no further role in the remainder of the paper.
\end{remark}
 
Let us finish this section by observing that Theorems \ref{t:spec}, \ref{t:diamspec} and \ref{t:spec0} are consequences of Proposition \ref{p:nontriv} and Theorems \ref{t:gen} and \ref{t:diamgen}:

\begin{proof}[Proof of Theorems \ref{t:spec}, \ref{t:diamspec} and \ref{t:spec0}]
The fields $F_\alpha$, $\alpha \in (0,2)$, satisfy Assumption \ref{a:rv} and have a local-global decomposition (see Remark~\ref{r:examples1}), as do the fields $F_0^\gamma$, $\gamma > 0$ (see Remark~\ref{r:examples2}). Thus the results follow from a combination of Proposition~\ref{p:nontriv}, Remark~\ref{r:ell^**} and Theorems~\ref{t:gen}--\ref{t:diamgen}.
\end{proof}

\subsection{Sketch of the proof}
The core of the paper is the proof of Theorem~\ref{t:gen} in the case that $\arm_\ell(R)$ has sub-exponential decay (i.e.\ $\alpha \le 1$). As discussed above, the basic strategy is to show that $\arm_\ell(R)$ is asymptotically carried, in a large deviation sense, by the event in which the field has `excess' mean of $\ell^{\ast \ast}_c - \ell$ on the line segment $[0, R]$. If this occurs, the field `looks supercritical' in neighbourhood of the line segment, and one can show, using the definition of $\ell^{\ast \ast}_c$, that $0$ is connected to $\partial B(R)$ with non-negligible probability. 

\smallskip
An application of the Cameron-Martin theorem (see Proposition \ref{p:eb}) shows that this excess mean has probabilistic cost at most $e^{-\|h\|_H^2  / 2 }$, where $H$ is the reproducing kernel Hilbert space associated to $f$ (see Section \ref{s:cap1}), and $h \in H$ is any function such that $h \ge \ell^{\ast \ast}_c - \ell $ on the line segment $[0, R]$. This leads immediately to the lower bound
 \begin{align*}
  - \log \P[ \arm_\ell(R) ]  & \le    \frac{1}{2}    \min \{ \|h\|^2_H :  h \ge \ell^{\ast \ast}_c - \ell  \text{ on } [0, R] \}  + O(1)  \\
  & =   \frac{1}{2} ( \ell^{\ast \ast}_c - \ell  )^2     \capa_K([0, R]) + O(1)    .
  \end{align*}
 
It remains to obtain a matching upper bound. The strategy we use is a `one-step renormalisation' based on the fact that the event $\arm_\ell(R)$ implies the existence of a path in $\{ f \le \ell\}$ that crosses many well-separated annuli on some mesoscopic scale $L$. Setting $\delta > 0$ small and recalling the local-global decomposition $f = f_L + g_L$, this implies that either:
\begin{itemize}
\item The set $\{f_L \le \ell^{\ast}_c - \delta \}$ crosses a fraction $\delta$ of these annuli; or
\item The set $\{g_L  \ge \ell^{\ast}_c - \ell - \delta\}$ intersects a fraction $1-\delta$ of these annuli.
\end{itemize}
 Given sufficient a priori control on the probability that $\{f_L \le \ell^{\ast}_c - \delta \}$ crosses an annulus, then since $f_L$ is $L$-range dependent, by independence the first event is very unlikely. On the other hand, one can show (see Proposition~\ref{p:ld2}) that the probability that  $\{g_L  \ge \ell^{\ast}_c - \ell - \delta\}$ intersects a fraction of $1-\delta$ of these annuli is approximately
 \[   \exp \Big( - \frac{(\ell^\ast_c - \ell - \delta)^2}{2} \capa_K(  \text{`union of annuli'} )  \Big) . \]
By a `projection' argument, among all possible configurations of annuli the capacity is minimised by those which aligns along a line segment connecting $0$ to $\partial B(R)$, and by a `condensation' argument this capacity is asymptotically equivalent to $\capa_K([0, (1-\delta)R])$. Putting this together, and taking first $R \to \infty$ and then $\delta \to 0$, gives the upper bound
  \[ - \log \P[ \arm_\ell(R) ] \gtrsim    \frac{1}{2} (\ell^\ast_c - \ell)^2 \capa_K([0, R] )  . \]
 Note that the mesoscopic scale $L$ must be well-chosen: it needs to be large enough that the combinatorial complexity of the set of annuli does not overwhelm the bounds, but small enough for the renormalisation and the `condensation' to be effective. 

 \smallskip
The strategy sketched above is inspired by \cite{Szn15,grs21}, which used similar methods to analyse the GFF. However, as mentioned above, we replace the GFF-specific tools used in these works by robust tools relying only on the regularly-varying setting.

\smallskip
The a priori input is drawn from two sources. In the case $\alpha \in (0, 1]$, a sprinkled bootstrapping argument yields the a priori estimate  $- \log \P[ \arm_\ell(R) ] \ge c R^{\psi}$ for any $\psi < \alpha$. For $\alpha = 0$ the sprinkled bootstrapping arguments fail in general, and instead we use a recent estimate from \cite{mrv20} on the sharpness of the phase transition for strongly correlated fields.
 
\smallskip
The proof of Theorem \ref{t:gen} for $\alpha > 1$ is simpler, and only uses the sprinkled bootstrapping argument mentioned above. Theorem \ref{t:diamgen} is a consequence of Theorem \ref{t:gen} and properties of the local-global decomposition.
 
\subsection{Outline of the remainder of the paper}
In Section \ref{s:prelim} we collect preliminary results on Gaussian fields, including basic properties of the capacity. In Section \ref{s:lgd} we introduce the local-global decomposition and give sufficient conditions for its existence. In Section \ref{s:eup} we use a sprinkled bootstrapping argument to establish Proposition \ref{p:nontriv}; as a byproduct we obtain a priori upper bounds in the case $\alpha \in (0, 1]$ and complete the proof of Theorem \ref{t:gen} for $\alpha > 1$. In Section \ref{s:seup} we give the main renormalisation procedure which completes the proof of the upper bounds in Theorem \ref{t:gen} in the case $\alpha \in [0, 1]$. In Section \ref{s:lb} we give the proof of the lower bounds in Theorem \ref{t:gen} and the remaining results (Theorem \ref{t:diamgen} and Proposition~\ref{p:cl}). Finally the Appendix collects technical results on Gaussian fields and regularly varying functions.

\smallskip
\section{Gaussian field preliminaries}
\label{s:prelim}

In this section we collect preliminaries results on continuous Gaussian fields. We will state explicitly any additional assumptions that we use, but mostly this is limited to assuming the field is isotropic with regularly varying covariance kernel.

\subsection{Gaussian fields: The RKHS, energy, capacity, and convex duality}
\label{s:cap1}
We begin by developing some aspects of the general theory of continuous Gaussian fields, including the notion of capacity and its connection to the RKHS. Our references are \cite{lan72,jan97,ams14}.

\smallskip
 Let $f$ be a continuous Gaussian field on $\R^d$, not necessarily stationary, with covariance kernel $K(x, y)$. Let $\Xi$ be the Gaussian Hilbert space associated to $f$, i.e.\ the closed linear span of $\{f(x)\}$ in $L^2$ (see \cite{jan97}). The \textit{reproducing kernel Hilbert space} (RKHS) (also known as the \textit{Cameron-Martin space}) associated to $f$ is the linear space
  \[ H =  \{h_\xi: \xi \in \Xi \} , \quad h_\xi(x) =  \E[f(x) \xi ]  , \]
 equipped with the inner product
\[ \langle  h_\xi, h_{\xi'} \rangle_H = \E[\xi \xi'] .\]
The name `RKHS' refers to the \textit{reproducing property} of $K$: for all $h_\xi \in H, x \in \R^d$,
\[ \langle h_\xi(\cdot), K(x, \cdot) \rangle_H =  \langle h_\xi(\cdot), h_{f(x)}(\cdot)  \rangle_H  = \E[\xi f(x)] = h_\xi(x) . \]
The RKHS contains all deterministic functions $h$ such that the laws of $f + h$ and $f$ are mutually absolutely continuous; indeed, by the Cameron-Martin formula (see \cite[Theorems 14.1 and 3.33]{jan97}), for every $h_\xi \in H$,
\[\frac{dP}{dQ}(\omega)  = \exp \Big(  \xi(\omega)  - \frac{1}{2} \new{\E_Q[ \xi^2 ] }  \Big) ,   \]
where $P$ and $Q$ denote the laws of $f + h_\xi$ and $f$ respectively, and $dP/dQ$ is the Radon-Nikodym derivative. In particular, since $\E_Q[  \xi ] = 0$ and $\E_Q[ \xi^2 ] = \|h_\xi\|_H^2$, the \textit{relative entropy} from $Q$ to~$P$ is
\begin{equation}
\label{e:cmt}
     \int  - \Big( \! \log  \frac{dP}{dQ} \Big) dQ  =  \frac{1}{2} \|h_\xi \|_H^2 .
\end{equation}

\smallskip
If $f$ is stationary, the RKHS has a convenient representation as the Fourier transform of a (weighted, complex) $L^2$-space. Recall the spectral measure $\nu$ of $f$, and let $L^2_{\text{sym}}(\nu)$ denote the Hilbert space of complex Hermitian (i.e.\ $g(-x) = \bar{g}(x)$) functions on $\R^d$ equipped with inner product 
\[ \langle g_1, g_2 \rangle_{L^2_{\text{sym}}(\nu)} = \int  g_1 \bar{g}_2   \, d \nu .  \] 
Then the RKHS can be represented as
\[ H  =  \{ \mathcal{F}[g d\nu] :  g   \in L^2_{\text{sym}}(\nu)  \}  \]
with inner product inherited from $L^2_{\text{sym}}(\nu)$.

\smallskip
We next introduce the capacity (see \cite{lan72} for a general reference). Let $D \subset \R^d$ be a compact subset. Recall that $\mathcal{P}(D)$ denotes the set of (Borel) probability measures on $D \subset \R^d$, and define $\mathcal{M}(D)$ to be the set of signed finite (Borel) measures on $D \subset \R^d$. For a measure $\mu \in \mathcal{M}(D)$, the \textit{energy of $\mu$ (with respect to the kernel $K$)} is
\begin{equation}
\label{e:energy}
 E(\mu)  =  \int_D \int_D K(x,y) d\mu(x) d\mu(y)  .
 \end{equation}
Since $K$ is positive definite, $E(\mu) \ge 0$, and the fact that
\begin{equation}
\label{e:convex}
E(\lambda \mu_1 + (1-\lambda) \mu_2) = \lambda E(\mu_1) + (1-\lambda) E(\mu_2) - \lambda(1-\lambda) E(\mu_1 - \mu_2) ,
\end{equation}
for every $\mu_1, \mu_2 \in \mathcal{M}(D)$ and $\lambda \in \R$, shows that $\mu \mapsto E(\mu)$ is a convex function on $\mathcal{M}(D)$. In terms of the RKHS, the energy can be expressed as
\begin{equation}
\label{e:energyh}
  E(\mu) =  \| h_\mu  \|_H^2  
  \end{equation}
where $h_\mu(\cdot) = \int  K(\cdot, x) d \mu(x)  \in H $ is the \textit{potential} of the measure $\mu$; in the case of a stationary field with spectral measure $\nu$ this is equal to
   \begin{equation}
   \label{e:energystat}
   E(\mu) = \int |\mathcal{F}[\mu]|^2 d\nu  .
   \end{equation}
  One can also view the energy as capturing the fluctuations of $f$ averaged over $\mu$
\begin{equation}
\label{e:fluct}
   E(\mu)  = \textup{Var}\Big[ \int f(x) d\mu \Big]  .
   \end{equation}

\smallskip
The \textit{measures of minimal energy (for a compact $D$)} is the collection $\mathcal{S}(D)$ of probability measures $\mu \in \mathcal{P}(D)$ which minimise $E(\mu)$; by continuity and the compactness of $\mathcal{M}(D)$, there is at least one such measure. Although in general there may be multiple measures of minimal energy, uniqueness is guaranteed if $\mu \mapsto E(\mu)$ is \textit{strictly} convex on $\mathcal{M}(D)$, and in the stationary case there is a simple sufficient condition for this:

\begin{proposition}
\label{p:unique}
If $f$ is stationary and the support of the spectral measure contains an open set, then $\mu \mapsto E(\mu)$ is strictly convex on $\mathcal{M}(D)$. In particular $|\mathcal{S}(D)| = 1$.
\end{proposition}
\begin{proof}
Recall that $\mu \in \mathcal{M}(D)$ is compactly supported. Hence by the Paley-Weiner theorem, $\mathcal{F}[\mu]$ is an entire function, and so cannot vanish on an open set. By \eqref{e:energystat} and the assumption on the support of the spectral measure, this implies that $E(\mu) > 0$ for every $\mu \in \mathcal{M}(D)$, which gives the required strict convexity  by \eqref{e:convex}. See also \cite[Proposition 3.1]{cs18} for a similar statement in the one-dimensional case.
\end{proof}

The \textit{capacity of $D$ (with respect to the kernel $K$)} is defined as
\begin{equation}
\label{e:ca2}
 \capa_K(D)  = \Big( \min_{ \mu \in \mathcal{P}(D) } E(\mu) \Big)^{-1}  \in (0, \infty].
 \end{equation}
While in general the capacity can be infinite, it is finite if $\mu \mapsto E(\mu)$ is strictly convex, and hence under the sufficient conditions in Proposition \ref{p:unique}. By definition the capacity is increasing in the domain and decreasing under the addition of a positive definite kernel: for every $D \subset D'$, and continuous covariance kernels $K_1$ and $K_2$, 
\begin{equation}
\label{e:mono}
 \capa_{K_1+K_2}(D) \le \capa_{K_1}(D') .
\end{equation}
It is also clear that $\capa_K(D)$ depends only on the restriction of $K$ to $D$.

\smallskip
The capacity has a dual formulation in terms of the RKHS:

\begin{proposition}[Convex duality; see {\cite[Theorems 3.2, 4.1 and 4.3]{ams14}}]
\label{p:duality}
For every compact $D \subset \R^d$,
\begin{equation}
\label{e:dual}
 \capa_K(D) = \inf \{ \|h\|^2_H  :  h \in H, h|_D \ge 1 \}    .
\end{equation}
Moreover, if $\capa_K(D) < \infty $ then there is a unique minimiser in \eqref{e:dual} which is equal to $h = \capa_K(D) h_\mu$ for any $\mu \in \mathcal{S}(D)$. 
\end{proposition}
\begin{remark}
Although the setting in \cite{ams14} is one-dimensional, the proofs extend immediately to all compact $D \subset \R^d$. See also \cite[Theorem 2.5]{fug60}.
\end{remark}

While it is not true in general that the minimiser in \eqref{e:dual} satisfies $h = 1$ on $D$ (e.g.\ it is not true for smooth fields \cite[Section 5]{ams14}, nor if the covariance is regularly varying with index $\alpha < d-2$, although it \textit{is} true for the GFF), it is always the case that $h = 1$ on the support of energy minimising measures $\mu \in \mathcal{S}(D)$ (see \cite[Theorem 4.1(iii)]{ams14}).

\smallskip
A useful consequence of convex duality is the following bound on $\capa_K(D)$: for every $\mu \in \mathcal{P}(D)$ and $h \in H$ such that $h|_D \ge 1$,
\begin{equation}
\label{e:capbounds}
  E(\mu)^{-1}   \le \capa_K(D) \le \|h\|^2_H .
\end{equation}

\subsection{Capacity of paths, lines and boxes: Asymptotics, condensation, projection}
\label{s:cap2}

We next study the capacity of paths, lines, and boxes for isotropic fields with regularly varying covariance. In particular we (i) establish asymptotics for the capacity of large line segments and boxes, (ii) show that the capacity of a line segment is well-approximated by coarse grained approximations (`condensation'), and (iii) verify that the capacity of a collection of balls is minimised when they are positioned along a line (`projection'). 

\smallskip
In the remainder of the subsection we consider continuous isotropic Gaussian fields on~$\R^d$, $d \ge 1$, whose covariance kernel $K$ is regularly varying with index $\alpha \in [0, \infty)$. In the case $\alpha = 1$ we further assume that $K(x) \sim x^{-1} (\log x)^{-\gamma}$ for $\gamma <  1$.

\smallskip
We first study the asymptotic growth of the capacity of line segments:

\begin{proposition}[Capacity of line segments]
\label{p:capline}
As $R \to \infty$,
\begin{equation}
\label{e:capline}
\capa_K([0, R])  \sim  \begin{cases}  R /(2 \int_0^\infty K(x) dx )    & \text{if } \alpha > 1,   \int_0^\infty K(x) dx > 0 , \\  R  /( 2 \int_0^R K(x) dx )  \sim  \frac{1-\gamma}{2 K(R) (\log R) }  & \text{if } \alpha = 1,  \\   c_{\alpha} / K(R) &   \text{if } \alpha \in [0, 1)  , \end{cases} 
\end{equation}
where $c_\alpha$ is defined as in \eqref{e:calpha2} (or equivalently \eqref{e:calpha}, see Section \ref{sss:cap}). Moreover, if $\alpha \in [0, 1]$ then for any $r = r(R)  \ge 0$ such that $r = R^{o(1)}$, as $R \to \infty$,
\begin{equation}
\label{e:captube}
\capa_K([0, R] \times [0, r]^{d-1} )  \sim \capa_K([0, R]) .  
\end{equation}
\end{proposition}

\begin{remark}
Proposition \ref{p:capline} extends results in \cite[Theorems 7.1, 7.2]{ams14}, which proved \eqref{e:capline} in the cases $\alpha \in (0,1)$ and $\alpha > 1$ under the additional assumption that $K \ge 0$. 
\end{remark}

\begin{proof}
Let us first consider \eqref{e:capline}. For $\alpha \in [0, 1)$ the result is a special case of Proposition \ref{p:capdom} below (see also Remark \ref{r:reduc}). To treat the case $\alpha \ge 1$ we observe that, for any stationary continuous kernel $K$ and any $R, s > 0$ such that $ \min_{y \ge s R} \int_0^y K(x) dx  > 0$,
\begin{equation}
\label{e:capline1}
  \frac{R^2}{\int_0^R \int_0^R K(x-y) dx dy } \le \capa_K([0, R]) \le \frac{ \int_{- s R}^{R + s R} \int_{-s R}^{R +  s R} K(x-y) dx dy}{ (2 \min_{y \ge s R} \int_0^y K(x) dx )^2 } .   
  \end{equation}
Indeed the left-hand side of \eqref{e:capline1} is deduced from \eqref{e:capbounds} with $\mu$ the uniform measure on the line-segment $[0,R]$. To establish the right-hand side we consider the function
\begin{equation}
\label{e:capline2}
 h(\cdot) = h_\mu(\cdot) = \int  K(x-\cdot) d \mu(x) = \frac{\int_{-s R}^{R + s R} K(x -  \cdot) dx }{2 \min_{y \ge s R} \int_0^y K(x) dx } ,  
 \end{equation}
where $\mu$ is the uniform measure on the line-segment $[-sR, R+sR]$ scaled by 
\[  \frac{R+2sR}{2\min_{y \ge s R} \int_0^y K(x) dx}. \]
 Then, recalling \eqref{e:energyh}, $\|h\|_H^2 =  E[\mu]$ is equal to the right-hand side of \eqref{e:capline1}, and moreover, for any $z \in [0, R]$,
\[ h(z) = \frac{\int_0^{z + s R} K(x) dx + \int_0^{R + s R - z} K(x) dx }{2 \min_{y \ge s R} \int_0^y K(x) dx } \ge 1 .  \]
Hence $h$ is a valid candidate in \eqref{e:capbounds}, which gives the right-hand side of \eqref{e:capline1}.

By stationary and since $K$ is eventually positive, for every $s > 0$ and sufficiently large $R$ we deduce that
\begin{equation}
\label{e:capline3}
   \frac{R^2}{\int_0^R \int_0^R K(x-y) dx dy } \le \capa_K([0, R]) \le \frac{ \int_{0}^{R + 2s R} \int_{0}^{R +  2s R} K(x-y) dx dy}{ (2 \int_0^{s R} K(x) dx )^2 } .  
   \end{equation}

In the case $\alpha > 1$ recall that by assumption $\int_0^\infty |K(x)| dx < \infty$ and  $\int_0^\infty K(x) dx > 0$. Then, by a change of variables and dominated convergence, the left-hand side in \eqref{e:capline3} satisfies, as $R \to \infty$,
\[    \frac{R^2}{\int_0^R \int_0^R K(x-y) dx dy }  =   \frac{R}{  \int_0^1  \int_{-yR}^{(1-y)R} K(x) dx  dy } \sim  \frac{R}{2 \int_0^\infty K(x) dx} .  \]
Moreover for any $s > 0$ the right-hand side of \eqref{e:capline3} satisfies, as $R \to \infty$
\[ \frac{ \int_{0}^{R + 2s R} \int_{0}^{R +  2s R} K(x-y) dx dy}{ (2  \int_0^{s R} K(x) dx )^2 }  \sim  \frac{   (R + 2s R)  \big( 2  \int_0^\infty K(x) \big) }{ \big( 2  \int_0^\infty K(x) dx \big) ^2 } =   \frac{R(1 + 2s)}{2 \int_0^\infty K(x) dx} . \]
Taking $s \to 0$ gives the result. 

In the case $\alpha = 1$ we use the following additional facts (see \eqref{e:regvar1}): for any $\rho > 0$, 
\[  \frac{1}{2R} \int_0^R  \int_0^R K(x-y) dx  dy \sim  \int_0^{s R}   K(x) dx   \sim   \int_0^{R} K(x) dx \sim \frac{1}{1-\gamma} R K(R) (\log R) .    \]
Hence the left-hand side of \eqref{e:capline3} satisfies,  as $R \to \infty$,
\[  \frac{R^2}{\int_0^R \int_0^R K(x-y) dx dy }   \sim  \frac{R}{2 \int_0^R K(x) dx}   \sim  \frac{1-\gamma}{2 K(R) (\log R) }   . \]
Moreover, for any $s > 0$ the right-hand side of \eqref{e:capline3} satisfies, as $R \to \infty$,
\begin{equation}
\label{e:capline4}
   \frac{  \int_{0}^{R + 2s R} \int_{0}^{R + 2 s R} K(x-y) dx dy }{  (2 \int_0^{s R} K(x) dx)^2 } \sim  \frac{R(1 + 2s)}{2 \int_0^R K(x) dx}   \sim  \frac{(1-\gamma)(1+2s)}{2 K(R) (\log R) }   . 
   \end{equation}
Taking $s \to 0$ gives the result also in this case.

\smallskip
We now turn to \eqref{e:captube}. In the case $\alpha = 1$ we use the fact that (see \eqref{e:regvar2}), since $r = R^{o(1)}$, for any $s > 0$,
\[ \min_{h \in [0, \sqrt{d-1} r] }  \int_0^{s R} K( \sqrt{x^2 + h^2} ) dx  \sim \int_0^{s R} K(x) dx   \sim \int_0^R K(x) dx  . \]
Let $h$ be the function in \eqref{e:capline2}, which satisfies, for $z = (z_1, z_2) \in [0, R] \times [0, r]^{d-1}$,
\begin{align*}
h(z) & = \frac{\int_0^{z_1 + s R} K(\sqrt{x^2 + |z_2|^2 } ) dx + \int_0^{R + s R - z_1} K(\sqrt{x^2 + |z_2|^2}) dx }{2    \min_{y \ge s R} \int_0^{s R} K(x) dx }  \\
& \ge \frac{  \min_{h \in [0, \sqrt{d-1} r] }  \int_0^{s R} K( \sqrt{x^2 + h^2} ) dx  }{     \int_0^{s R} K(x) dx }  \\
& \ge 1 + o(1) ,
\end{align*}
where we used the fact that $K$ is eventually positive. Hence applying \eqref{e:capbounds} with $h(\cdot)$ scaled by $(\min_{z \in  [0, R] \times [0, r]^{d-1}} h(z))^{-1}$ shows that 
\begin{align*}
  \capa_K([0, R] \times [0, r]^{d-1} ) &  \le (1+o(1))  \|h\|_H^2  \sim \frac{ \int_{0}^{R + 2s R} \int_{0}^{R +  2s R} K(x-y) dx dy}{ (2 \int_0^{s R} K(x) dx )^2 } \\
  & \sim \frac{R(1 + 2s)}{2 \int_0^R K(x) dx}   \sim  \capa_K([0, R])(1 + 2s) , 
  \end{align*}
where the third step used \eqref{e:capline4}. Since $ \capa_K([0, R] \times [0, r]^{d-1} )  \ge   \capa_K([0, R])$ by monotonicity of the capacity \eqref{e:mono}, taking $s \to 0$  gives the result.

\smallskip
Finally, in the case $\alpha \in [0, 1)$, Proposition \ref{p:capdom} below states that, for any $\rho > 0$,
\begin{align*}
&   \capa_K([0, R] \times [0, \rho R]^{d-1} ) K(R)  \\
&   \qquad \sim \Big(  \min_{ \mu \in \mathcal{P}([0,1] \times [0,\rho]^{d-1}) }  \int_{[0, 1] \times [0, \rho]^{d-1}}  \int_{[0, 1] \times [0, \rho]^{d-1}}   |x-y|^{-\alpha} d\mu(x)  d\mu(y) \Big)^{-1}  .
\end{align*}
By monotonicity, the latter quantity is at least $c_\alpha$, so it remains to rule out
\[ \liminf_{\rho \to 0}    \min_{ \mu \in \mathcal{P}([0,1] \times [0,\rho]^{d-1}) }  \int_{[0, 1] \times [0, \rho]^{d-1}}  \int_{[0, 1] \times [0, \rho]^{d-1}}   |x-y|^{-\alpha} d\mu(x)  d\mu(y)    <  c^{-1}_\alpha . \]
For the sake of contradiction suppose this were true, and extract a subsequence $\rho_n \to 0$ such that the minimisers $\mu_n$ converge weakly to a $\hat{\mu} \in \mathcal{P}([0,1])$, and 
\[    \lim_{n \to \infty}      \int_{[0, 1] \times [0, \rho_n]^{d-1}}  \int_{[0, 1] \times [0, \rho_n]^{d-1}}   |x-y|^{-\alpha} d\mu_n(x)  d\mu_n(y)    <  c^{-1}_\alpha .  \]
 Then by weak convergence we would have
\[   \int_{[0, 1]}  \int_{[0, 1]}   |x-y|^{-\alpha} d\hat{\mu}(x)  d\hat{\mu}(y) <  c^{-1}_\alpha  =     \min_{ \mu \in \mathcal{P}([0,1]) }  \int_{[0, 1]}  \int_{[0, 1]}   |x-y|^{-\alpha} d\mu(x)  d\mu(y)  , \]
which is a contradiction
\end{proof}

 \begin{remark}
 \label{r:tube}
The proof of \eqref{e:capline} in the case $\alpha > 1$ applies to any stationary kernel $K$ satisfying $\int_0^\infty |K(x)| dx < \infty$ and $\int_0^\infty K(x) dx > 0$. Notice also that the proof of \eqref{e:captube} in the case $\alpha \in [0, 1)$ requires only that $r = o(R)$ rather than $r = R^{o(1)}$. In principle this would allow us to weaken the definition of $\ell^{\ast \ast}_c$ in the case $\alpha \in (0, 1)$.
 \end{remark}
 
\begin{remark}
In the case $\alpha > 1$ of Proposition \ref{p:capline} we assumed that $\int_0^\infty K(x) dx > 0$, and in general this condition cannot be dispensed with. However for isotropic fields in dimension $d \ge 2$ this is without loss of generality by the following claim:
\end{remark}
 
 \begin{claim}
\label{c:int}
If $K : \R^d \to \R$, $d \ge 2$, is a non-zero isotropic kernel such that $K|_\R \in L^1(\R)$, then $\int_0^\infty K(x)  dx > 0$.
\end{claim}
\begin{proof}
Let $\mu$ be the spectral measure of $K$; the spectral measure of the restriction $K|_\R$ is the projection $P_1(\mu)$ of $\mu$ onto $\R$. Recall the classical fact that if the uniform measure on the unit sphere $\mathbb{S}^{d-1}$ is orthogonally projected onto a linear subspace of dimension $k \le d-1$, the push-forward measure on the unit disk has density proportional to $(1-\|x\|_2^2)^{(d-k-2)/2}$; in particular the density is continuous and strictly positive near the origin. Since $P_1(\mu)$ can be represented as an average of such projections, $P_1(\mu)$ has a density $\rho$ satisfying $\rho(0) > 0$. Hence $ \int_0^\infty K(x) dx = \rho(0)/2$ is strictly positive.
\end{proof}

We next generalise Proposition \ref{p:capline} to rescaled domains in the `non-integrable' case (i.e.\ in which $\alpha$ is less than the dimension of the domain). For simplicity we consider only the case of balls or boxes, although the proof works for bounded domains with sufficiently regular boundary (e.g.\ piece-wise smooth and Lipschitz).

\begin{proposition}[Capacity of rescaled domains; non-integrable case]
\label{p:capdom}
Let $D \subset \R^d$ be a domain which is either a ball $B(r)$, $r > 0$, or a box $[0, a_1] \times \cdots \times [0, a_d]$, $a_i > 0$. Then if $\alpha \in [0, d)$, as $R \to \infty$, 
\begin{equation}
\label{e:capbox2}
\capa_K(RD)  \sim   c_{D;\alpha} / K(R) 
\end{equation}
where
\[   c_{D;\alpha} =  \Big(  \min_{ \mu \in \mathcal{P}(D) }  \int_{D} \int_{D}  |x-y|^{-\alpha} d\mu(x)  d\mu(y) \Big)^{-1}  > 0 . \]
Note that $c_{D;\alpha} = 1$ if $\alpha = 0$.
\end{proposition}
\begin{proof}
We use a similar strategy to the proof of \cite[Theorem 7.2]{ams14}.

\smallskip
 Let us first consider the upper bound. We claim that without loss of generality we may assume that $K \ge 0$. Indeed, by Proposition \ref{p:rvdecomp} there exist continuous isotropic positive definite kernels $K_1$ and $K_2$ such that $K = K_1 + K_2$, $K_1 \ge 0$, and $K_1(R) \sim K(R)$ as $R \to \infty$; in particular $K_1$ is regularly varying with index $\alpha$. Then by monotonicity of the capacity~\eqref{e:mono}, 
  \[ \capa_K(RD) / K(R)  \le  \capa_{K_1}(RD)  / K(R) = (1+o(1))  \capa_{K_1}(RD)  / K_1(R) .\] 
So henceforth we assume that $K \ge 0$. Via \eqref{e:ca2} and a change of variables we have
\begin{equation}
\label{e:capdom0}
 \frac{1}{ \capa_K(RD) K(R) } = \min_{\mu \in \mathcal{P}(D)} \int_D \int_D \frac{K( R (x-y) )}{K(R)} d\mu(x) d\mu(y) . 
 \end{equation} 
Let $\mu_\alpha \in \mathcal{P}(D)$ denote a probability measure which minimises energy with respect to the $\alpha$-Riesz kernel on $D$, i.e.\ such that
\[   c_{D;\alpha}^{-1} =    \min_{\mu \in \mathcal{P}(D)} \int_D \int_D |x-y|^{-\alpha} d\mu(x) d\mu(y)  =  \int_{D} \int_{D}  |x-y|^{-\alpha} d\mu_\alpha(x)  d\mu_\alpha(y)  . \]
Suppose, for the sake of contradiction, that
\[   \liminf_{R \to \infty}  \min_{\mu \in \mathcal{P}(D)} \int_D \int_D \frac{K( R (x-y) )}{K(R)} d\mu(x) d\mu(y)  <  c_{D;\alpha}^{-1} , \]
Extract subsequences $R_n \to \infty$ and $\mu_n \in \mathcal{P}(D)$, $n \ge 1$, such that $\mu_n$ converges weakly to a $\mu_\infty \in \mathcal{P}(D)$, and also 
\[   \liminf_{n \to \infty} \int_D \int_D \frac{K( R_n (x-y) )}{K(R_n)} d\mu_n(x) d\mu_n(y)  <  c_{D;\alpha}^{-1} . \]
Then by Fatou's lemma (recall that we assume $K \ge 0$) and regular variation,
\[   c_{D;\alpha}^{-1}  > \liminf_{n \to \infty} \int_D \int_D \frac{K( R_n (x-y) )}{K(R_n)} d\mu_n(x) d\mu_n(y)  \ge    \int_D \int_D |x-y|^{-\alpha} d\mu_\infty(x) d\mu_\infty(y)   , \]
which contradicts the definition of $c_{D;\alpha}$.

\smallskip
We turn to proving the lower bound. Let $\theta > 0$ and define $\nu_\theta \in \mathcal{P}(D)$ to be the probability measure with bounded density obtained as follows. Let $X$ be a random variable whose law is $\mu_\alpha$ and let $U$ be an independent uniform random variable on $B(1)$. Then let $\nu_\theta$ be the law of the random variable $ c_\theta (X + \theta U)$, where $c_\theta \in (0,1)$ is chosen to be the largest value that ensures $Y$ is supported on~$D$ (i.e.\ $c_\theta =  r / (r+\theta)$ if $D= B(r)$ and $c_\theta = \min_i a_i /(a_i+\theta)$ if $D = [0,a_1] \times \cdots \times [0, a_d]$). Fixing $\varepsilon \in (0, d-\alpha)$, by Potter's bounds \cite[Theorem 1.5.6]{bgt87} we may take $t > 0$ sufficient large so that, for all $R \ge t$ and $x,y \in D$,
\[ \frac{K( R (x-y) )}{K(R)} \id_{|x-y| \ge t / R}  \le   \begin{cases} 2 |x-y|^{-\alpha - \varepsilon}, & |x-y| \le 1, \\ 2 |x-y|^{-\alpha+\eps}, & |x-y| \ge 1. \end{cases} \]
Hence by dominated convergence we have
\[  \lim_{R \to \infty}  \int_D \int_D \frac{K( R (x-y) )}{K(R)} \id_{|x-y| \ge t / R} \, d\nu_\theta(x) d\nu_\theta(y)  =  \int_D \int_D |x-y|^{-\alpha} d\nu_\theta(x) d\nu_\theta(y) . \]
On the other hand,
\[  \lim_{R \to \infty}  \int_D \int_D \frac{K( R (x-y) )}{K(R)} \id_{|x-y| \le t / R} \, d\nu_\theta(x) d\nu_\theta(y)   \le O(R^{-d} / K(R)) = o(1) \]
by the bounded density of $\mu_\theta$. This establishes that
\[   \limsup_{R \to \infty}  \min_{\mu \in \mathcal{P}(D)} \int_D \int_D \frac{K( R (x-y) )}{K(R)} d\mu(x) d\mu(y)  \le \int_D \int_D |x-y|^{-\alpha} d\nu_\theta(x) d\nu_\theta(y)  . \]
By \eqref{e:capdom0}, it remains to argue that
\[ \lim_{\theta \to 0}  \int_D \int_D |x-y|^{-\alpha} d\nu_\theta(x) d\nu_\theta(y)  =  \int_D \int_D |x-y|^{-\alpha} d\mu_\alpha(x) d\mu_\alpha(y) , \]
or equivalently that
\begin{equation}
\label{e:capdom1}
  \lim_{\theta \to 0} c_\theta \E \big[ | X_1 - X_2 + \theta ( U_1 - U_2) |^{-\alpha}  \big]  = \E \big[ |X_1-X_2|^{-\alpha} \big], 
  \end{equation}
where $X_i$ and $U_i$ are independent copies of $X$ and $U$ respectively. First, for $\theta \in [0,1]$ and $x_1, x_2 \in \R^d$ define
\[ h_\theta(x_1,x_2) = c_\theta \E[  \big[ | x_1 - x_2 + \theta ( U_1 - U_2) |^{-\alpha}  \big]   .\]
Since $c_\theta \to 1$ as $\theta \to 0$, and by dominated convergence, for every $x_1 \neq x_2$, $h_\theta(x_1,x_2) \to h_0(x_1,x_2)$ as $\theta \to 0$. It may also be checked by elementary calculation that 
\[ h_\theta(x_1,x_2) \le c_\alpha h_0(x_1,x_2) = c_\alpha |x_1-x_2|^{-\alpha} \]
for a constant $c_\alpha > 0$ that depends only on $\alpha$. Applying the dominated convergence theorem again establishes \eqref{e:capdom1}.
\end{proof}

\begin{remark}
\label{r:reduc}
Since the capacity of a set $D \subset \R^d$ depends only on the restriction of $K$ to $D$, one can apply Proposition \ref{p:capdom} to sets of dimension lower than the the ambient space. In particular, applying it to the line segment $D = [0, 1]$ recovers the $\alpha \in [0, 1)$ case of Proposition~\ref{p:capline}. 
\end{remark}

We next study `condensation' and `projection' properties of the capacity. For the next two propositions we assume that $K$ is asymptotic to a decreasing function. If $\alpha > 0$ this is automatic since in fact $K(x) \sim \sup_{t \ge x} K(t)$ (see Lemma \ref{l:evd}), but it is not true in general if $\alpha = 0$. In particular we will use the fact that
\begin{equation}
\label{e:evdec}
  \xi_r = \sup_{y \ge x \ge r}    \frac{ K(y) }{ K(x)} - 1  ,
 \end{equation}
tends to zero as $r \to \infty$, whereas in general this is only true for $\alpha > 0$. We also recall the uniform convergence property of regular varying functions \cite[Theorem 1.2.1]{bgt87} which implies that, for any $\delta > 0$ there exists $\rho \in (0,1)$ such that, as $R \to \infty$ eventually
\begin{equation}
\label{e:rv2}
\sup_{ z   \in R (1- \rho, 1 + \rho)} \Big|  \frac{K(z)}{K(R)} - 1  \Big| \le \delta  .
\end{equation}

\smallskip
For $0 \le s \le r \le R$, define 
\[    \mathcal{S}_{s, r, R} =    \cup_{1 \le i < \lfloor R/r \rfloor } \{ i r +  [0, s] \}  \]
to be a coarse grained approximation of the line segment $[0, R]$, using `grains' of length $s$ translated along $r \mathbb{Z}$ (see Figure \ref{f:cond}). In the `non-integrable' case $\alpha \in [0, 1]$, we show that, provided $r$ does not grow too quickly with $R$, and $s$ does not grow too slowly, the capacity of $\mathcal{S}_{s, r, R}$ approximates the capacity of $[0, R]$:

\begin{proposition}[Condensation]
\label{p:con}
Let $\alpha \in [0, 1]$ and let $r = r(R) \to \infty$ and $s = s(R) \to \infty$ be such that $s \le r \le R$ and, as $R \to \infty$,
\begin{equation}
\label{e:concon}
 r  = \begin{cases}   R^{o(1)} & \text{if } \alpha = 1 ,   \\  o(R)  & \text{if } \alpha \in [0, 1),  \end{cases} \qquad \text{and} \qquad  \begin{cases}    r K(s) (\log s) \ll R K(R) (\log R)  & \text{if } \alpha = 1 ,   \\  r K(s)  \ll R K(R)    & \text{if } \alpha \in [0, 1).  \end{cases}    
 \end{equation}
Then, as $R \to \infty$,
\[ \capa_K([0, R]) \sim \capa_K( \mathcal{S}_{s, r, R} )  . \]
\end{proposition}

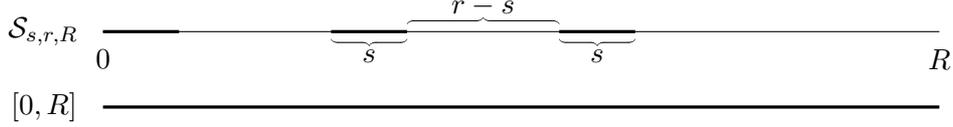
\begin{figure}
\begin{tikzpicture}
\draw (0,0) -- (11,0);
\node[below] at (0,-0.1) {$0$};
\node[below] at (11,-0.1) {$R$};
\draw[very thick] (0,0) -- (1,0);
\draw [decorate,decoration={calligraphic brace,mirror}] (3,-0.1) -- (4,-0.1);
\node[below] at (3.5,-0.1) {$s$};
\draw [decorate,decoration={calligraphic brace}] (4,0.1) -- (6,0.1);
\node[below] at (5,0.6) {$r-s$};
\draw [decorate,decoration={calligraphic brace,mirror}] (6,-0.1) -- (7,-0.1);
\node[below] at (6.5,-0.1) {$s$};
\draw[very thick] (3,0) -- (4,0);
\draw[very thick] (6,0) -- (7,0);
\draw[very thick] (0,-1) -- (11,-1);
\node[left] at (-0.2,0) {$ \mathcal{S}_{s, r, R} $};
\node[left] at (-0.2,-1) {$ [0,R]$};
\end{tikzpicture}
\caption{Condensation of line segments in Proposition \ref{p:con}: provided $r = r(R)$ does not grow too quickly and $s = s(R)$ does not grow too slowly, the capacities of $ \mathcal{S}_{s, r, R} $ and $[0,R]$ are asymptotically equivalent.}
\label{f:cond}
\end{figure}

\begin{proof}
 By monotonicity of the capacity \eqref{e:mono}, it is enough to prove that
\[ \liminf_{R \to \infty} \frac{\capa_K( \mathcal{S}_{s, r, R} )  }{ \capa_K([0, R]) } \ge 1 .\]
 By the same monotonicity we can and will also assume that $s < r/2$. Since $R \mapsto \capa_K([0,R])$ is regularly varying and $r = O(R)$, we may also assume that $R/r$ is an integer.
 
 Let us first consider the case $\alpha \in [0, 1)$. For $1 \le i < R/r $ we define $S_i = ir + [0, s]$ and $T_i = ir + [0, r]$, so that $\cup_i S_i =  \mathcal{S}_{s, r, R}$ and $\cup_i T_i = [0,R]$. Let $\mu_\alpha \in \mathcal{P}([0,1])$ denote the probability measure that minimises energy with respect to the $\alpha$-Riesz kernel on $[0, 1]$, i.e.\ such that
\[ c_\alpha^{-1} =  \int_{[0, 1]} \int_{[0,1]}  |x-y|^{-\alpha} d\mu_\alpha(x)  d\mu_\alpha(y) . \]
Further, for $\theta > 0$, let $\nu_\theta \in \mathcal{P}([0,1])$ denote the measure with continuous bounded density obtained by convolving $\nu_\alpha$ with the uniform distribution on $[0,\theta]$ and rescaling the resulting measure by $x \mapsto x/(1+\theta)$, and then let $\nu'_\theta$ be the probability measure on $\cup_i S_i$ that has constant density $s^{-1} \nu_\theta(T_i/ R) $ on each $S_i$. 

Recall from the proof of Proposition \ref{p:capdom} that, 
 \[ \lim_{\theta \to 0}   \int_{[0,1]} \int_{[0,1]} |x-y|^{-\alpha} d\nu_\theta(x) d\nu_\theta(y) = c^{-1}_\alpha .\]
Hence by applying the left-hand side of \eqref{e:capbounds} to the measure $\nu'_\theta$, taking $\theta \to 0$, and recalling Proposition \ref{p:capline}, it is sufficient to show that, for every $\theta > 0$,
\begin{equation}
\label{e:cond1}
 \limsup_{R \to \infty} \frac{ \sum_{i,j}   s^{-2} \nu_\theta(T_i/R) \nu_\theta(T_j/R)   \int_{S_i} \int_{S_j} K(x-y) dx dy }{  K(R)  \int_{[0,1]} \int_{[0,1]} |x-y|^{-\alpha} d\nu_\theta(x) d\nu_\theta(y)  } \le 1 .   
 \end{equation}
Since $\nu_\theta$ has bounded density, for every $c > 0$ we have, as $R \to \infty$,
\begin{align}
\label{e:cond2}
  &  \sum_{i,j : |i-j| \le c }  s^{-2} \nu_\theta(T_i/R) \nu_\theta(T_j/R)  \int_{S_i} \int_{S_j} K(x-y) dx dy \\
 \nonumber & \qquad = O(1) \times \frac{R}{r} \times \frac{r^2}{R^2 s^2} \Big(  \int_{[0, s]} \int_{[0, s]}  K(x-y) dx dy  +    s^2 K(s/2)(1 + \xi_{s/2} )  \Big)  \\
 \nonumber & \qquad = O(r K(s) / R )  = o(K(R))  ,
   \end{align}
where in the second step we used that  $s^{-2} \int_{[0, s]} \int_{[0, s]}  K(x-y) dx dy  = O(K(s))$ by Lemma \ref{l:regvar1}, that $K(s/2) = O(K(s))$ by regular variation, and that $\xi_s \to 0$ (recall \eqref{e:evdec}), and in the final step we used \eqref{e:concon}. On the other hand, recalling \eqref{e:rv2}, since $r \to \infty$, for every $\delta > 0$ one can fix $c > 0$ and $R > 0$ sufficiently large so that, for all $|i-j| \ge c$ and $x \in T_i, y \in T_j$,
\[   \Big| \frac{K(x-y)}{K( |i-j| r)}  \Big| \in \big (( 1+\delta)^{-1}, 1+\delta \big) .\]
Hence
\begin{align}
\label{e:cond3} &  \sum_{i,j : |i-j| \ge c }   s^{-2} \nu_\theta(T_i/R) \nu_\theta(T_j/R)   \int_{S_i} \int_{S_j} K(x-y) dx dy \\
\nonumber & \qquad  \le (1 + \delta)  \sum_{i,j : |i-j| \ge c }    \nu_\theta(T_i/R) \nu_\theta(T_j/R)   K(|i-j| r)  \\
\nonumber & \qquad \le (1+\delta)^2   \sum_{i,j : |i-j| \ge c }     \int_{T_i} \int_{T_j} K(x-y) d \nu_\theta(x) d\nu_\theta(y) . 
\end{align}
Finally, by a dominated convergence argument as in the proof of Proposition \ref{p:capdom}, for every $c \ge 2$, as $R \to \infty$
\begin{equation}
\label{e:cond4}
\lim_{R \to \infty} \frac{   \sum_{i,j : |i-j| \ge c }     \int_{T_i} \int_{T_j} K(x-y) d \nu_\theta(x) d\nu_\theta(y)  }{K(R)  \int_{[0,1]} \int_{[0,1]}  |x-y|^{-\alpha} d\nu_\theta(x) d\nu_\theta(y) } =1 .
 \end{equation}
Taking first $R \to \infty$ and then $c \to \infty$, and combining \eqref{e:cond2}--\eqref{e:cond4}, gives \eqref{e:cond1}.
 
The case $\alpha = 1$ is similar except we apply the left-hand side of \eqref{e:capbounds} to the normalised Lebesgue measure on $\cup_i T_i$ instead of $\nu_\theta$, and then instead of \eqref{e:cond1} we need to show that
\begin{equation}
\label{e:cond5}
  \limsup_{R \to \infty} \frac{ \frac{r^2} { s^2 R^2  }  \sum_{i,j}  \int_{S_i} \int_{S_j} K(x-y) dx dy }{  2(1-\gamma)^{-1} K(R) \log R    } \le 1  .  
  \end{equation}
 First, similarly to \eqref{e:cond2}, for every $c > 0$ we have, as $R \to \infty$,
\begin{align}
\label{e:cond6}
  &  \frac{r^2}{s^2 R^2}  \sum_{i,j : |i-j|  \le c }   \int_{S_i} \int_{S_j} K(x-y) dx dy \\
 \nonumber & \qquad = O(1) \times \frac{R}{r} \times \frac{r^2}{s^2 R^2} \Big(  \int_{[0, s]} \int_{[0, s]}  K(x-y) dx dy  +    s^2 K(s/2)(1 + \xi_{s/2} )  \Big)  \\
 \nonumber & \qquad = O(r K(s) \log s/ R )  = o(K(R) \log R)  ,
   \end{align}
where we used that $s^{-2} \int_{[0, s]} \int_{[0, s]}  K(x-y) dx dy  = O(K(s) \log s)$ by Lemma \ref{l:regvar2}, and also \eqref{e:concon}. On the other hand, exactly as in \eqref{e:cond3}, for every $\delta > 0$ one can fix $c > 0$ and $R > 0$ sufficiently large so that
 \begin{align}
\label{e:cond7} &  \sum_{i,j : |i-j| \ge c }   \frac{r^2}{s^2 R^2}  \int_{S_i} \int_{S_j} K(x-y) dx dy \\
\nonumber & \qquad \le (1+\delta)^2   \sum_{i,j : |i-j| \ge c }     \int_{T_i} \int_{T_j} K(x-y) dx dy
\end{align}
  Finally since $r = R^{o(1)}$, applying \eqref{e:regvar2}, for all $c > 0$ we have that
\begin{equation}
\label{e:cond8}
  \lim_{R \to \infty}  \frac{ \frac{1}{R^2}  \int_{[0,R]} \int_{[0,R]}K(x-y) \id_{|x-y| \ge c r} d x dy }{ 2(1-\gamma)^{-1} K(R) \log R   }  = 1 . \end{equation}
Combining \eqref{e:cond6}--\eqref{e:cond8}, we have established \eqref{e:cond5}.
 \end{proof}

The next proposition shows that the capacity of a collection of disjoint balls exceeds the capacity of their `projection' onto a line segment.

\begin{proposition}[Projection]
\label{p:proj}
There is a positive function $\xi'_r \to 0$ such that the following holds: For every $r \ge 1$, $s \in (0, r/2)$, finite collection $\{D_i\}_{1 \le i \le n}$ of translations of the ball $B(s)$ such that their centres $\{c_i\}_{1 \le i \le n}$ satisfy $|c_i - c_j| \ge |i-j| r$ for every $1 \le i,j \le n$, 
\[ \capa_K( \cup_i D_i ) \ge  \capa_K( \mathcal{S}_{s, r, (n-1)r} ) (1 - \xi'_r) .\]
\end{proposition}
\begin{proof}
Let $\mu$ be a measure of minimal energy for the set $ \cup_{1 \le i \le n} \{ i r + B(s) \} $, and recall the function $\xi_r$ defined in \eqref{e:evdec}. Since the centre of the balls satisfy the pairwise bounds $|c_i - c_j| \ge |i-j| r$, we have that 
\[   \int_{\cup_i D_i} \int_{ \cup_i D_i} K(x-y) d\mu(x) d\mu(y)  \le  (1 + \xi_r)   \int_{ \cup_i \{ir + B(s) \}}   \int_{ \cup_i \{ir + B(s) \}}   K(x-y) d\mu(x) d\mu(y) . \]
Given \eqref{e:dual} and \eqref{e:capbounds}, this shows that
\[ \capa_K( \cup_i D_i ) \ge \frac{1}{1+\xi_r} \capa_K(  \cup_{1 \le i \le n} \{ i r + B(s) \} ) , \]
which implies the result by the monotonicity of the capacity \eqref{e:mono}.
\end{proof}

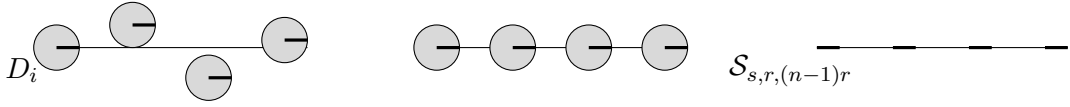
\begin{figure}
\begin{tikzpicture}
\draw (0,0) -- (3.3,0);
\node[left] at (-0.1,-0.3) {$D_i$};
\draw[fill=gray!30] (0,0) circle (0.3);
\draw[fill=gray!30] (1,0.3) circle (0.3);
\draw[fill=gray!30] (2,-0.4) circle (0.3);
\draw[fill=gray!30] (3,0.1) circle (0.3);
\draw[very thick] (0,0) -- (0.3,0);
\draw[very thick] (1,0.3) -- (1.3,0.3);
\draw[very thick] (2,-0.4) -- (2.3,-0.4);
\draw[very thick] (3,0.1) -- (3.3,0.1);
\draw (5,0) -- (8.3,0);
\draw[fill=gray!30] (5,0) circle (0.3);
\draw[fill=gray!30] (6,0) circle (0.3);
\draw[fill=gray!30] (7,0) circle (0.3);
\draw[fill=gray!30] (8,0) circle (0.3);
\draw[very thick] (5,0) -- (5.3,0);
\draw[very thick] (6,0) -- (6.3,0);
\draw[very thick] (7,0) -- (7.3,0);
\draw[very thick] (8,0) -- (8.3,0);
\draw (10,0) -- (13.3,0);
\node[left] at (10.6,-0.3) {$ \mathcal{S}_{s, r, (n-1)r} $};
\draw[very thick] (10,0) -- (10.3,0);
\draw[very thick] (11,0) -- (11.3,0);
\draw[very thick] (12,0) -- (12.3,0);
\draw[very thick] (13,0) -- (13.3,0);
\end{tikzpicture}
\caption{Projection of balls onto line segments in Proposition \ref{p:proj}: the capacity of the collection $\{D_i\}_{i \le n}$ (left) is larger, up to negligible error, than the capacity of its projection onto a line (centre), which in turn is larger than the capacity of the line segments $ \mathcal{S}_{s, r, (n-1)r}$ (right).}
\label{f:proj}
\end{figure}

Combining the previous two propositions we obtain the following:

\begin{corollary}
\label{c:conpro}
Let $\alpha \in [0, 1]$ and let $r = r(R) \to \infty$ and $s = s(R) \to \infty$ satisfy \eqref{e:concon}. Then
\[ \liminf_{\rho \to 1} \liminf_{R \to \infty} \inf_{D_i \in \mathfrak{D}_{R;\rho} }  \frac{\capa_K( \cup_i D_i ) }{ \capa_K( [0, R] ) } \ge 1  ,\]
where $\mathfrak{D}_{R;\rho}$ is the set of finite collections $\{D_i\}_{1 \le i \le n}$ of translations of the ball $B(s)$, with $n \ge \lfloor \rho R/r \rfloor$, such that their centres $\{c_i\}_{1 \le i \le n}$ satisfy $|c_i - c_j| \ge |i-j| r$.
\end{corollary}

\subsection{Gaussian large deviation bounds}
\label{s:ld}

In this section we collect large deviation bounds for continuous Gaussian fields.

\subsubsection{Exceedences}
The classical large deviation bound for exceedences is the following: 

\begin{proposition}[BTIS inequality; see {\cite[Theorem 2.9]{aw09}}]
\label{p:btis}
Let $f$ be a continuous Gaussian field on $\R^d$ and let $D \subset \R^d$ be a compact set. Then for  $u \ge 0$,
\[  - \log \P \Big[  \sup_{x \in D}  f(x) \ge   \E \big[ \sup_{x \in D} f(x) \big]  + u \Big]    \ge  \frac{u^2}{ 2   \sup_{x \in D}\textup{Var}[f(x)]  }   . \]
\end{proposition}

To apply the BTIS inequality we need control on the expected supremum of $f$ on compact sets. We next state simple bounds in the case of $C^1$-smooth stationary Gaussian fields:

\begin{proposition}
\label{p:ld1}
Let $f$ be a $C^1$-smooth stationary Gaussian field on $\R^d$ with covariance $K$. Then there exists a $c_d > 0$, depending only on the dimension $d$, such that for all $R \ge 3$,
\begin{equation}
\label{e:esup1}
 \E \Big[ \sup_{x \in B(R)} f(x) \Big]  \le c_d  \, \sqrt{ \max \Big\{ (\log R) \, K(0), R^{-2} \sup_{|k| = 1} \E[ (\partial^k f(0))^2]   \Big\}   } 
 \end{equation}
 and also
 \begin{equation}
\label{e:esup2}
 \E \Big[ \sup_{x \in B(R)} f(x) \Big]  \le c_d  \sqrt{ \max\Big\{ K(0),  R^2 \sup_{|\nu| = 1} \E[ (\partial^k f(0))^2]   \Big\} } .
 \end{equation}
\end{proposition}
\begin{proof}
In the proof $c_i > 0$ denote constants which depend only on the dimension. To prove \eqref{e:esup1}, consider the rescaled field $g(x) = f(x/R)$ for which
\[  \E \Big[ \sup_{x \in B(R)} f(x) \Big]   =  \E \Big[ \sup_{x \in B(R^2)} g(x) \Big]  .  \]
Let $X = \sup_{x \in B(1)} g(x)  $ and $\bar X = (X - \E[X])/\sqrt{K(0)}$. By Kolmogorov's theorem (see, e.g., \cite[Appendix A.9]{ns16}) and the BTIS  inequality (Proposition \ref{p:btis}) we have
\begin{equation}
\label{e:ld1}
    \E [ X ]  \le  c_1  \sqrt{ \max \Big\{ K(0), \sup_{|k| = 1} \E[ (\partial^k g(0))^2]   \Big\} }   \quad \text{and} \quad \P[ \bar X \ge u ] \le e^{- u^2 / 2}   \ , \ u \ge 0. 
    \end{equation}
Consider a collection $\{B_i\}$  of $c_2 R^{2d}$ translated copies of $B(1)$ which cover $B(R^2)$, and define $X_i =   \sup_{x \in B_i} g(x)$ and $\bar X_i = (X_i  - \E[X_i])/\sqrt{K(0)} $. By stationarity we have
\begin{equation}
\label{e:id2}
   \E \Big[  \sup_{x \in B(R^2)} g(x)  \Big] \le   \E \Big[  \max_i X_i \Big]  =  \E[X] + \sqrt{K(0)} \E \big[ \max_i \bar{X}_i \big] .  
   \end{equation}
   We use the following standard claim about sub-Gaussian random variables:
   
   \begin{claim}
   \label{c:sub}
  Let $(Y_i)_{1 \le i \le n}$ be identically distributed with $\P[Y_i \ge u] \le e^{-u^2/2}$ for all $u \ge 0$. Then there exists a universal $c > 0$ such that, for all $n \ge 3$,  $\E[ \max_{i \le n} Y_i ] \le  c  \sqrt{\log n}$.
  \end{claim}
  \begin{proof}
Observe that $\exp ( \lambda \E [ \max_i Y_i ] ) \le n  \E[ e^{\lambda Y_1 } ]  \le 4 n \lambda  e^{\lambda^2 /2}   $ for all $\lambda \ge 1$ by Jensen's inequality, then set $\lambda = \sqrt{\log n}$.
\end{proof}

Combining with \eqref{e:ld1} and \eqref{e:id2}, we have
\[  \E \Big[ \sup_{x \in B(R^2)} g(x) \Big]  \le c_3   \, \sqrt{ \max \Big\{ (\log R) \, K(0),  \sup_{|k| = 1} \E[ (\partial^k g(0))^2]   \Big\}   }  \]
which is equivalent to \eqref{e:esup1} after rescaling. Eq.\ \eqref{e:esup2} can be proven similarly, except this time considering the rescaled field $g(x) = f(Rx)$, and we omit the details.
\end{proof}

\subsubsection{Entropic bound}
The following bound is a simple consequence of \eqref{e:cmt}:

\begin{proposition}
\label{p:eb}
Let $f$ be a continuous Gaussian field on $\R^d$. Then for every $h \in H$ and event $A$,
\[ \P[ f  \in A ] \ge   \P[f + h \in A] \exp \Big(   -  \frac{\|h\|_H^2}{2  \P[f + h \in A]}  - 1  \Big)  .\]
\end{proposition}
\begin{proof}
Let $P$ and $Q$ denote the laws of $f$ and $f+h$ respectively. By \eqref{e:cmt} and Jensen's inequality
\begin{align*}
- \frac{\|h\|_H^2}{2}  & = \int  \log \frac{dP}{dQ} dQ   = \int_A \log \frac{dP}{dQ} dQ  + \int_{A^c} \log \frac{dP}{dQ} dQ   \\
&\le Q[A] \Big( \log \frac{P[A]}{Q[A]}  + \frac{Q[A^c]}{Q[A]}   \log \frac{P[A^c]}{Q[A^c]}  \Big) \\
& \le  Q[A] \Big( \log \frac{P[A]}{Q[A]}  + 1\Big) ,
\end{align*}
where the last inequality is since $P[A^c] \le 1$ and $\sup_{x \in [0, 1]} \frac{1-x}{x} \log \frac{1}{1-x} = 1$, which is easily verified. Rearranging gives the result.
\end{proof}

\begin{corollary}
\label{c:eb}
Let $f$ be a continuous Gaussian field on $\R^d$. Then for every $u \ge 0$, compact set $D \subset \R^d$, and event $A$ that is increasing and measurable with respect to $f|_D$,
\[  \P[ f  \in A ] \ge     \P[f + u  \in A] \exp \Big( -  \frac{u^2 \capa_K(D) }{2  \P[f + u  \in A]}  - 1 \Big)   .\]
\end{corollary}
\begin{proof}
Suppose $ \capa_K(D) < \infty$ since otherwise the statement is trivial. By convex duality (Proposition \ref{p:duality}) there exists $h \in H$ with $h|_D \ge 1$ and $\|h\|_H = \capa_K(D)$. Applying Proposition \ref{p:eb}, and since $A$ is increasing and measurable with respect to $f|_D$, gives the result.
\end{proof}

\smallskip
\section{Local-global decompositions}
\label{s:lgd}
In this section we introduce `local-global decompositions' of smooth Gaussian fields, show their existence under our assumptions, and deduce bounds on exceedences of the global field.

\subsection{Definition and existence}
\label{s:locglo}
 
Let us begin by defining a local-global decomposition:

\begin{definition}
\label{d:locglo}
Let $f$ be a $C^1$-smooth isotropic Gaussian field on $\R^d$ with covariance $K$ that is eventually positive. We say that $f$ \textit{has a local-global decomposition} if there exists a collection of pairs of isotropic Gaussian fields $(f_L, g_L)_{L \ge 1}$ such that, for every $L \ge 1$:
\begin{itemize}
\item $ f \stackrel{d}{=} f_L + g_L $.
\item The field $f_L$ is $L$-range dependent, i.e.\
\[  \mathbb{E}[f_L(0) f_L(x) ]  = 0  \qquad \text{for all } |x| \ge L .\]
\item The field $g_L$ is $C^1$-smooth and satisfies
\begin{equation}
\label{e:locglo1}
  \lim_{M \to \infty} \limsup_{L \to \infty} \sup_{|x| \ge ML}  \Big| \frac{ \mathbb{E}[g_L(0) g_L(x) ] }{ K(x) } -  1 \Big|  = 0  ,
\end{equation}
and, as $L \to \infty$,
\begin{equation}
\label{e:locglo2}
\mathbb{E}[ g_L(0)^2 ] = O(K(L))  \quad \text{and} \quad  \sup_{|k| = 1} \mathbb{E}[ (\partial^k g_L(0))^2 ]  =   o(1)  . 
\end{equation}
If $|1/K(L)| = L^{o(1)}$ as $L \to \infty$, we ask in addition that
\begin{equation}
\label{e:locglo3}
 \sup_{|k| = 1} \mathbb{E}[ (\partial^k g_L(0))^2 ]  =   O(  L^{-2} K(L)  ) .
 \end{equation}
 \end{itemize}
\end{definition}

We next give two sufficient conditions for the existence of such a decomposition.  The reader may wish to skip this and move to Section~\ref{s:egf} at a first reading. The first condition is more restrictive but easier to verify:

\begin{assumption}
\label{a:gen1}
Suppose Assumption \ref{a:rv} holds. In addition, there exists a non-negative function $v \in C^0(\R^+)$ such that $K(r) = \mathcal{L}[v](r^2)$, where $\mathcal{L}[\cdot]$ denotes the Laplace transform. Moreover $v(s) s^{d/2+3}$ is bounded, and there is a $c > 0$ such that, as $s \to 0$,
\begin{equation}
\label{a:v}
 v(s) \sim \begin{cases}
c s^{\alpha/2 - 1} L(1/\sqrt{s}) & \text{if } \alpha \in (0, d), \\ 
c  s^{-1}  (\log (1/s))^{-\gamma-1} & \text{if } \alpha = 0 .    \end{cases}
\end{equation}
\end{assumption}

\begin{remark}[Examples]
\label{r:examples1}
The function $r \mapsto (1 + r)^{-\alpha/2}$, $\alpha > 0$, is the Laplace transform of
\[ v(s) = \Gamma(\alpha/2)^{-1} s^{\alpha/2-1} e^{-s} , \quad s > 0 ,\]
where $\Gamma(\cdot)$ is the gamma function. Hence the smooth isotropic field on $\R^d$, $d \ge 2$, with Cauchy covariance \eqref{e:cauchy} satisfies Assumption \ref{a:gen1} if $\alpha \in (0,d)$. More generally, one can take any non-negative function $v \in C^0(\R^+)$ with $v(s) s^{d/2+3}$ bounded and $v(s)$ satisfying \eqref{a:v}. Then we show in Appendix~\ref{a:wnd} that the isotropic Gaussian field on~$\R^d$, $d \ge 2$, whose covariance is $K(r) \propto \mathcal{L}[v](r^2)$, suitably normalised, satisfies Assumption \ref{a:gen1}.
\end{remark}

The second is broader but may be harder to verify:

\begin{assumption}
\label{a:gen2}
Suppose Assumption \ref{a:rv} holds. In addition, the spectral measure has a density $\rho$, the function $q = \mathcal{F}[\sqrt{\rho}]$ is in $C^3(\R^d)$, for each multi-index $k$ with $|k| \le 3$, $\partial^k q \in L^2(\R^d)$, and there is a $c > 0$ such that, as $x \to \infty$,
\begin{equation}
\label{a:q}
 q(x) \sim \begin{cases}
c x^{-(d+\alpha)/2} \sqrt{L(x)} & \text{if } \alpha \in (0, d), \\ c  x^{-d/2} (\log x)^{-(\gamma+1)/2} & \text{if } \alpha = 0 .    \end{cases}
\end{equation}
If $\alpha = 0$ we additionally assume that, as $x \to \infty$,
\begin{equation}
\label{a:extra}
 \sup_{|k| = 1}   |\partial^k q(x)| = O \big(x^{-1} q(x) \big)  . 
 \end{equation}
\end{assumption}

\begin{remark}[Examples]
\label{r:examples2}
A first class of examples is obtained by taking a smooth isotropic unimodal function $q$ such that $\partial^k q \in L^2(\R^d)$ for all $|k| \le 3$, which satisfies \eqref{a:q}, and such that $x^{(d-1)/2} q(x)$ is eventually decreasing (and also satisfying~\eqref{a:extra} if $\alpha = 0$). Then if $d \in \{2,3\}$ and $\alpha \in ((d-3)/2, d)$, we show in Appendix \ref{a:tau} that the isotropic Gaussian field on $\R^d$, $d \ge 2$, whose covariance kernel is $K \propto q \star q$, suitably normalised, satisfies Assumption \ref{a:gen2}. The restriction $d \in \{2,3\}$ and $\alpha \in ((d-3)/2, d)$ is for technical reasons, and it is possible these examples also satisfy Assumption~\ref{a:gen2} without this restriction.

\smallskip
Second, one can obtain a class of examples via \textit{spectral blow-up}, namely by assuming that the spectral density $\rho$ exists and satisfies, as $\lambda \to 0$,
\begin{equation}
\label{e:blowup}
\rho(\lambda) \sim \begin{cases} 
 \lambda^{\alpha - d }  & \text{if } \alpha \in (0,d) , \\ 
 \lambda^{- d } (\log (1/\lambda) )^{-(\gamma +1)}  & \text{if } \alpha = 0 .
  \end{cases}
\end{equation} 
As we show in Appendix \ref{a:tau}, if $\alpha \in ((d-3)/2, d)$ and $K$ (resp.\ $q$) is continuous with $K(x) x^{(d-3)/2}$ (resp.\ $q(x) x^{(d-3)/2}$) eventually decreasing, then \eqref{a:k} (resp.\ \eqref{a:q}), suitably normalised, is implied by~\eqref{e:blowup}.
\end{remark}

\begin{remark}
Although we do not show this formally, a consequence of the classical Tauberian theorem for the Laplace transform is that Assumption \ref{a:gen1} implies the spectral density satisfies~\eqref{e:blowup} (suitably normalised). Hence if $\alpha \in ((d-3)/2, d)$ and $q$ is continuous with $q(x) x^{(d-3)/2}$ eventually decreasing, Assumption \ref{a:gen1} implies Assumption \ref{a:gen2}. It is plausible that this implication holds in general, but we lack a proof.
\end{remark}

We now verify the sufficiently of these conditions:

\begin{proposition}
\label{p:locglo}
If $f$ satisfies Assumption \ref{a:gen1} or \ref{a:gen2} it has a local-global decomposition.
\end{proposition}

Before embarking on the proof let us explain the main idea, which is to represent the field as a \textit{spatial moving average}. More precisely, under Assumption \ref{a:gen2} we use the representation $f = q \star W$, where $q$ is the defined in Assumption~\ref{a:gen2} and $W$ is the white noise on $\R^d$. On the other hand, we show in Appendix \ref{a:wnd} that, under Assumption \ref{a:gen1}, there exists a $q(x,t) \in L^2(\R^d \times \R^+)$ such that $f = q \star_1 W$, where $W$ is the white noise on $\R^d \times \R^+$, and $\star_1$ denotes the convolution over $\R^d \times \R^+$ restricted to $\R^d \times \{0\}$. In both cases we define $f_L$ and $g_L$ via a truncation of $q$ at scale $L$.

\smallskip
Similar decompositions were used in \cite{mv20, s21,m22} to study Gaussian field percolation. An important difference in our application is that relevant properties are carried by the global field, whereas in \cite{mv20,s21,m22} the global field was discarded.

\begin{proof}[Proof of Proposition \ref{p:locglo} in the case of Assumption \ref{a:gen2}]
Recall that $f = q \star W$ (see Appendix~\ref{a:wnd}). Let $\varphi : \mathbb{R}^+ \to \R$ be a smooth function with bounded derivatives of all order such that $\varphi(x) = 1$ for $x \le 1/4$ and $\varphi(x) = 0$ for $x \ge 1/2$. For $L \ge 1$, define $q_L(\cdot) = q(\cdot) \varphi(\cdot / L)$, and the isotropic Gaussian fields
\[ f_L = q_L \star W \quad \text{and} \quad g_L = ( q - q_L)  \star W  = f - f_L . \] 
Let us verify that $f_L$ and $g_L$ fulfil the requirements of a local-global decomposition. By definition it is clear that $f_L$ and $g_L$ are isotropic Gaussian fields and that $f_L$ is $L$-range dependent; hence it remains to show that $g_L$ is $C^1$-smooth and satisfies \eqref{e:locglo1}--\eqref{e:locglo3}.

As a preliminary we observe, since $q \in C^3(\R^d)$ and $\varphi$ has uniformly bounded derivatives, $q_L \in C^3(\R^d)$. Hence $g_L$ is $C^2$-smooth (Lemma \ref{l:a1}), and for each $|k| \le 2$,
\[   \partial^k f = (\partial^k q) \star W  \quad \text{and} \quad \partial^k g_L = (\partial^k (q-q_L)) \star W .\]   
 We also observe that, by assumption, as $L \to \infty$, for some $c > 0$,
 \begin{equation}
 \label{e:locglop1}
  q(L)^2 L^d \sim \begin{cases} c K(L)  & \text{if } \alpha \in (0, d) , \\ c K(L) / \log L  & \text{if }  \alpha = 0 . \end{cases} 
 \end{equation}
In particular, $q$ is regularly varying with index $(d+\alpha)/2 \in (\alpha,d)$.

 \textit{Proof of \eqref{e:locglo1}.} Since $f = f_L + g_L$ and $f_L$ is $L$-range dependent, for $x \ge 3 L$ we have
\begin{equation}
\label{e:decomp}
 \mathbb{E}[g_L(0) g_L(x) ]  =  \E[f(0) f(x)] - 2 \E[f_L(0) f(x)]  = K(x) - 2  \langle q_L(\cdot), q (\cdot - x) \rangle_{L^2}   .   
 \end{equation}
Moreover, since $x \ge 3L$ and $q$ is eventually positive, as $L \to \infty$,
\[  \langle q_L(\cdot) , q (\cdot - x) \rangle_{L^2}   \le \| q_L \|_{L^1} \sup_{y \in B(x + L)} q(y)    = O \Big(  \int_0^L  x^{d-1} q(x) \, dx    \times  q(x-L) \Big) , \]
where we used that $q(x) \sim \sup_{t \ge x} q(t)$ (Lemma \ref{l:evd}). Since $\alpha < d$ we have by Proposition \ref{p:kit} (with $\beta = (d-\alpha)/2 - 1 > -1$) and \eqref{e:locglop1} that, as $L \to \infty$,
 \[  \int_0^L x^{d-1} q(x) \, dx     =   O( L^d q(L) ) = \frac{O(K(L)) }{q(L) }  . \]
Moreover, by Lemma \ref{l:evd}, for sufficiently large $M > 1$,
\[ \sup_{|x| \ge ML} \frac{q(x-L)}{K(x)}  \sim \frac{ q((M-1)L)}{K(ML) }   . \]
Hence, by the regular variation of $K$ and $q$, for sufficiently large $M > 1$,
\[ \limsup_{L \to \infty} \sup_{|x| \ge ML}  \frac{K(L)}{K(x)}  \frac{ q(x-L)}{q(L) }  =   \limsup_{L \to \infty}   \frac{K(L)}{K(ML)}  \frac{q((M-1)L) }{q(L) }  = M^\alpha (M-1)^{-(d+\alpha)/2}     . \]
Recalling that $(d+\alpha)/2 > \alpha$, we have shown that
\[   \lim_{M \to \infty} \limsup_{L \to \infty} \sup_{|x| \ge ML}  \frac{ \langle q_L(\cdot), q(\cdot - x) \rangle_{L^2} }{K(x) }   =  0 ,   \]
which, combining with \eqref{e:decomp}, completes the proof.
 
\textit{Proof of \eqref{e:locglo2}.} We note that, since $q$ is eventually positive,
\[  \E[g_L(0)^2 ]   =   \|q -  q_L \|_2  = O\Big(  \int_{L/4}^\infty x^{d-1} q(x)^2  dx \Big)  , \]
In the case $\alpha \in (0, d)$ we have by Proposition \ref{p:kit} (with $\beta = -\alpha - 1 < -1$) and \eqref{e:locglop1}
\[  \int_{L/4}^\infty x^{d-1} q(x)^2  dx   = O(K(L)) , \]
and in the case $\alpha = 0$ instead by Proposition \ref{p:kitlog} (with $\beta = -\gamma - 1 < -1$)
\[  \int_{L/4}^\infty x^{d-1} q(x)^2  dx   \sim  \int_{L/4}^\infty x^{-1} (\log x)^{-(\gamma + 1)}   dx  =  O( (\log L)^{-\gamma}) = O(K(L)) . \]
Moreover, since $\partial^k q \in L^2$, as $L \to \infty$,
\[ \E[ ( \partial^k g_L(0))^2 ]  = \| \partial^k (q - q_L) \|_{L_2} \to 0 .  \]

\textit{Proof of \eqref{e:locglo3}.} Eq.\ \eqref{e:locglo3} only applies in the case $q(x) \sim x^{-d/2} (\log x)^{-(\gamma+1)/2}$. Recalling the assumption \eqref{a:extra}, for each $|k| = 1$,
\[ \E[ (\partial^k g_L(0))^2 ]  = \| \partial^k (q - q_L) \|_{L^2} = O \Big(  \int_{L/4}^\infty  x^{d-1} x^{-(d+2)} (\log x)^{-\gamma + 1} dx \Big) = O(L^{-2} (\log L)^{-\gamma} )  , \]
as required, where we used Proposition \ref{p:kit} (with $\beta = -3$).
 \end{proof}
 
 \begin{proof}[Proof of Proposition \ref{p:locglo} in the case of Assumption \ref{a:gen1}]
In Appendix \ref{a:wnd} we show that, under Assumption \ref{a:gen1},
\[ f = q \star_1 W \ , \quad q(x,t) = w(t) Q(x/t) ,\]
where $w(t) \propto   t^{-d-3} v(1/(4t^2))$, $Q(x) = e^{-|x|^2/2}$, and $W$ is the white noise on $\R^d \times \R^+$. Let $\varphi$ be as in the proof in the case of Assumption \ref{a:gen2}, and for $L \ge 1$, define $q_L(\cdot) = w(t) Q(\cdot/t) \varphi(\cdot/L)$ and the isotropic Gaussian fields
\[ f_L = q_L \star_1 W \quad \text{and} \quad g_L = ( q - q_L)  \star_1 W  = f - f_L . \] 
As in the previous proof, it is clear that $g_L$ and $f_L$ are isotropic Gaussian fields, $f_L$ is $L$-range dependent, and $g_L$ is $C^2$-smooth (see Lemma \ref{l:a1}) with, for each $|k| \le 2$,
\[   \partial^k f = (\partial^{k,0} q) \star_1 W  \quad \text{and} \quad \partial^k g_L = ( \partial^{k,0} (q-q_L) ) \star W , \]
where $\partial^{k,0} = \partial^k_x$. Moreover, the bounds \eqref{e:locglo2} and \eqref{e:locglo3} are a special case of \cite[Proposition 2.6]{m22}. Hence it remains to show that $g_L$ satisfies \eqref{e:locglo1}.

Similarly to \eqref{e:decomp}, for $x \ge 3 L$ we have
\begin{equation}
\label{e:decomp2}
\mathbb{E}[g_L(0) g_L(x) ]  =  \E[f(0) f(x)] - 2 \E[f_L(0) f(x)]  = K(x) - 2  \langle q_L(\cdot), q (\cdot - x) \rangle_{L^2(\R^d \times \R^+)}   .   
\end{equation}
Using that $x \mapsto e^{-x^2}$ is decreasing on $\R^+$ and $q_L$ is supported on $B(L)$, we have
\begin{align*}
 \langle q_L(\cdot), q (\cdot - x) \rangle_{L^2(\R^d \times \R^+)} &\le \int_0^\infty \| q_L(\cdot,t)  \|_{L^1} \sup_{y \in B(x + L)} q(y,t) \, dt    \\
 &  \le  c_1 L^d \int_0^\infty w(t)  e^{-|x-L|^2/(2t^2)} \, dt ,
 \end{align*}
 for a constant $c_1 > 0$. By the change of variables $t \mapsto 1/(2 \sqrt{s})$ and applying Proposition \ref{p:taulaplace}, as $x \to \infty$,
 \begin{align*}
\int_0^\infty w(t)  e^{-|x|^2/(2t^2)} \, dt & = c_2  \int_0^\infty s^{d/2}   v(s) e^{-2|x|^2 s} \, ds \\
 & = c_2 \mathcal{L}[ v(s) s^{d/2} ]( 2x^2 ) \\
   & \le c_3 (1+o(1) )
 K \big(\sqrt{2} x \big) x^{-d} 
  \end{align*}
 for constants $c_2,c_3 > 0$. Since, by Lemma \ref{l:evd}, for sufficiently large $M$,
 \[   \sup_{|x| \ge ML}  \frac{ K(\sqrt{2}(x-L)) (x-L)^{-d}  }{K(x) }   \sim  L^{-d} (\sqrt{2}(M-1))^{-\alpha} (M-1)^{-d}  M^{\alpha}   , \] we conclude that
 \[  \limsup_{L \to \infty} \sup_{|x| \ge ML}  \frac{ \langle q_L(\cdot), q(\cdot - x) \rangle_{L^2(\R^d \times \R^+)} }{K(x) }   \le   c_4 (M-1)^{-\alpha-d}  M^{\alpha} \le c_5 M^{-d} . \]
 for $c_4, c_5> 0$. Combining with \eqref{e:decomp2} and sending $M \to \infty$, we have the result.
 \end{proof}

\subsection{Exceedences of the global field}\label{s:egf}
We next control the exceedences of the global field. In the next two results $f$ is a $C^1$-smooth isotropic Gaussian field on $\R^d$ with covariance $K$ that is regularly varying with index $\alpha \in [0, \infty)$, and we assume that $f$ has a local-global decomposition $(f_L, g_L)_{L \ge 1}$. 

\smallskip
The first result controls exceedences in a ball of scale $L$:
\begin{proposition}
\label{p:ld2}
For any $u > 0$, as $L \to \infty$, 
\begin{equation}
\label{e:gsup}
\E \Big[\sup_{x \in B(L)} g_L \Big] = o(1)  \quad \text{and} \quad  - \log \P \Big[ \sup_{x \in B(L)} g_L \ge u \Big] \ge \frac{1}{O(K(L))}  .
\end{equation}
Moreover, if $\alpha > 0$ then for every $\psi \in (0, \min\{\alpha/2,1\})$ there is a $c_1 > 0$ such that, for any $L \ge 1$ and $u \ge L^{-\psi}$,
\begin{equation}
\label{e:gsup2}
 - \log \P \Big[ \sup_{x \in B(L)} g_L \ge u \Big] \ge \frac{c_1 u^2}{K(L)}  ,
\end{equation}
and if $\alpha = 0$ there is a $c_2 < 0$ such that, for any $L \ge 1$ and $u \ge c_2^{-1} \sqrt{ K(L)}$,  
\begin{equation}
\label{e:gsup3}
 - \log \P \Big[ \sup_{x \in B(L)} g_L \ge u \Big] \ge \frac{c_2 u^2}{K(L)}  .
\end{equation}
\end{proposition}
 \begin{proof}
We have by \eqref{e:esup1} that, as $L \to \infty$,
 \[ \E \Big[\sup_{x \in B(L)} g_L \Big]  \le  \sqrt{  O \Big(  \max \Big\{ (\log L) \E[ (g_L(0)^2 ]   , L^{-2} \sup_{|\nu| = 1}  \E[ ( \partial^\nu g_L(0) )^2 ]   \Big) \Big\} }   \]
 and
 \[ \E \Big[\sup_{x \in B(L)} g_L \Big]  \le  \sqrt{  O \Big(  \max \Big\{  \E[ (g_L(0)^2 ]   , L^2 \sup_{|\nu| = 1}  \E[ ( \partial^\nu g_L(0) )^2 ]   \Big) \Big\} }   .\]
By the assumed properties of $g_L$ in \eqref{e:locglo2} and \eqref{e:locglo3}, in the case $\alpha > 0$, for any $\psi \in (0, \min\{\alpha/2,1\})$ the first bound is $O(L^{-\psi})$, where as in the case $\alpha = 0 $ the second bound is $O(\sqrt{K(L)})$; in particular, in both cases  $\E [\sup_{x \in B(L)} g_L ] \to 0$. The result then follows from the BTIS inequality (Proposition~\ref{p:btis}) and the fact that $\E[g_L(0)^2] = O(K(L))$ by \eqref{e:locglo2}.
 \end{proof}
 
The next result controls exceedences in a collection of disjoint balls simultaneously; it is inspired by \cite[Theorem 3.2]{Szn15}, which proved a similar result for the GFF. For $s, r > 0$, and a ball $D$ which is a translated copy of $B(r)$, let $\Gamma_s(D)$ denote the ball that is concentric to $D$ with radius $s$. 

\begin{theorem}
\label{t:ld3}
Let $u > 0$ and let $s = s(L) \to \infty$ be such that, as $L \to \infty$, $K(L) = o(K(s))$. Then there exists a function $\gamma_{L, M}$ satisfying
\[ \lim_{M \to \infty} \limsup_{L \to \infty} \gamma_{L, M} = 0  \]
such that, for every $L, M \ge 1$ and finite collection $D = (D_i)_{1 \le i \le n}$ of disjoint translations of the ball $B(L)$ such that their pairwise distance is larger than $ML$,
\[ - \log \mathbb{P}[ A]  \ge  \frac{u^2   \capa_K(  \cup_i \Gamma_s(D_i)   )( 1 - \gamma_{L,M})   }{2}    ,\]
where $A$ is the event that $\{g_L(x) \ge u\}$ has non-empty intersection with each $D_i$.
\end{theorem}
\begin{proof}
Let $u > 0$, $L, M \ge 1$, and the finite collection $D = (D_i)$ be given. Let $\mu$ be a measure of minimum energy for the set $\cup_i \Gamma_s(D_i)$ with respect to $K$, and define the positive weights $\lambda_i = \mu(\Gamma_s(D_i))$ which satisfy $\sum_i \lambda_i = 1$. Observe that the event $A$ implies that 
\[     \sup_{x_i \in D_i}  \sum_i \lambda_i  g_L(x_i) \ge u  , \]
where the supremum is taken over the choice of one $x_i$ in each $D_i$. By stationarity, as $L \to \infty$,
\[ \E \Big[ \sup_{x_i \in D_i}  \sum_i \lambda_i  g_L(x_i)   \Big]  =    \mathbb{E} \Big[ \sum_i \lambda_i   \sup_{x_i \in D_i}  g_L(x_i)  \Big] \le  \E \Big[\sup_{x \in B(L)}  g_L(x)  \Big]   \to  0 , \]
where in the last step we used~\eqref{e:gsup}. Then by Proposition \ref{p:btis} applied to the Gaussian field $F : D_1 \times \ldots \times D_n \to \R$, $(x_1, \ldots, x_n) \mapsto \sum_i \lambda_i  g_L(x_i)$,
  \begin{align*}
       - \log \P \Big[  \sup_{x_i \in D_i}  \sum_i \lambda_i  g_L(x_i) \ge u   \Big]   &  \ge  \frac{ \Big(u -   \E \Big[ \sup_{x_i \in D_i}  \sum_i \lambda_i  g_L(x_i)   \Big]   \Big)^2  }{ 2   \sup_{x_i \in D_i}   \textup{Var} [ \sum_i \lambda_i  f(x_i)   ]   }  \\
       &  \ge  \frac{ u^2 +  o(1)   }{ 2   \sup_{x_i \in D_i}   \textup{Var} [ \sum_i \lambda_i  f(x_i)   ]   }     
       \end{align*}
as $L  \to \infty$. Hence it remains to prove that
    \[ \sup_{x_i \in D_i}  \textup{Var} \big[ \sum_i \lambda_i  g_L(x_i)   \big]  \le (\capa_K(\cup_i \Gamma_s(D_i) ))^{-1} (1 + \gamma_{L, M}) \]
    for some $ \lim_{M \to \infty} \limsup_{L \to \infty} \gamma_{L, M} = 0$, or equivalently (by \eqref{e:fluct} and \eqref{e:dual})
    \begin{align*}
 &    \sup_{x_i \in D_i}  \sum_{i,j} \mu(\Gamma_s(D_i)) \mu(\Gamma_s(D_j))  \mathbb{E}[g_L(0) g_L(x_i - x_j)]    \\
     & \qquad \qquad  \le (1 + \gamma_{L, M}) \sum_{i,j} \int_{\Gamma_s(D_i)} \int_{\Gamma_s(D_j)} K(x-y)  d\mu(x) d\mu(y)  .
     \end{align*}
    We first deal with the diagonal terms $i = j$. For brevity we abbreviate $\tilde{D}_i = \Gamma_s(D_i)$. Note that, by the definition of the global field $g_L$
    \[    \mu(\tilde{D}_i)^2 \mathbb{E}[g_L(0)^2]  \le  c_1 \mu(\tilde{D}_i)^2 K(L)   .      \]    
for some $c_1 > 0$ independent of $i$. On the other hand, by the left-hand side of \eqref{e:capbounds} applied to the measure $\mu(\cdot) / \mu(\tilde{D}_i)$ (which is a probability measure on $\tilde{D}_i$),
    \[ \int_{\tilde{D}_i} \int_{\tilde{D}_i} K(x-y)  d\mu(x) d\mu(y)  \ge \frac{  \mu(\tilde{D}_i)^2 }{\capa_K(\tilde{D}_i)}   = \frac{  \mu(\tilde{D}_i)^2 }{\capa_K(B(s))} . \]
 By Proposition \ref{p:capdom} since $K(L) = o(K(s))$, 
 \[  \capa_K(B(s)) \sim c_2 / K(s) = o(1/K(L))  ,\]
 and we conclude that
 \[   \mu(\tilde{D}_i)^2 \mathbb{E}[g_L(0)^2]  \le  \int_{\tilde{D}_i} \int_{\tilde{D}_i} K(x-y)  d\mu(x) d\mu(y)   )   \quad \text{for all } i \]
eventually as $L \to \infty$. To handle the off-diagonal terms we define
\begin{align*}
 \gamma_{L, M}  &=   \sup_{ \substack{|x'-y'|\ge ML, \\ \max\{ |x-x'|, |y-y'| \}\le 2s}}   \Big| \frac{ \mathbb{E}[g_L(0) g_L(x'-y') ] }{ K(x-y) } -  1 \Big|    \\
 & =   \sup_{ \substack{|x'-y'|\ge ML, \\ \max\{ |x-x'|, |y-y'| \}\le 2s}}   \Big| \frac{ \mathbb{E}[g_L(0) g_L(x'-y') ] }{ K(x'-y')} \times   \frac{ K(x'-y')  }{ K(x-y) } -  1 \Big|   ,
 \end{align*} 
 which satisfies $\lim_{M \to \infty} \limsup_{L \to \infty} \gamma_{L, M} = 0$ by the assumption on $g_L$ and the regular variation of $K$. Then since distinct $D_i$ and $D_j$ are separated by distance at least $ML$ we have that
\[  \sum_{i \neq j} \mu(\tilde{D}_i) \mu(\tilde{D}_j)  \sup_{x_i \in D_i}  \mathbb{E}[g_L(0) g_L(x_i - x_j)]     \le (1 + \gamma_{L, M}) \sum_{i\neq j} \int_{\tilde{D}_i} \int_{\tilde{D}_j} K(x-y)  d\mu(x) d\mu(y)  ,\]
as required.
  \end{proof}

\smallskip
\section{Sprinkled bootstrapping: A priori bounds and exponential decay}
\label{s:eup}

In this section we use a sprinkled bootstrapping argument to obtain various a priori bounds; this is similar to the approach pioneered in \cite{rs13,pt15,pr15}. First we give conditions under which $\ell^\ast_c > -\infty$. As a byproduct we also obtain that, at levels $\ell < \ell_c^\ast$, the connectivity of the excursion set $\{f \le \ell\}$ has, in the case $\alpha \in (0, 1]$, at least stretched exponential decay with arbitrary exponent $\psi < \alpha$, and exponential decay in the case $\alpha > 1$. In the latter case this completes the proof of the upper bounds in Theorem~\ref{t:gen}.

\smallskip
We work in a more general setting than in the rest of the paper, namely we consider continuous stationary random fields that satisfy a weak notion of `local-global decomposition'; in particular, we do not impose Gaussianity, isotropy, or regularly varying assumptions.

\smallskip
The main result of this section is the following: 

\begin{proposition}
\label{p:sb}
Let $f$ be a continuous stationary random field on $\R^d$ such that, for every $L \ge 1$, $f \stackrel{d}{=} f_L + g_L$, where $f_L$ is a stationary $L$-range dependent field on $\R^d$, and $g_L$ is a continuous field on $\R^d$ satisfying, for constants $0 < \delta' < \delta < 1 $ and $c > 0$,
\begin{equation}
\label{e:tail}
 \P\big[ \sup_{x \in B(L)} |g_L(x)| \ge  (\log L)^{-\delta} \big] \le   e^{- c (\log L)^{1-\delta'}} \ , \quad L \ge 2 . 
 \end{equation}
 Then $\ell^\ast_c > - \infty$.

Suppose that in addition there exists constants $\psi, \delta, c_1 > 0$ such that
\begin{equation}
\label{e:tail2}
 \P\big[ \sup_{x \in B(L)} |g_L(x)| \ge  (\log L)^{-(1+\delta)} \big] \le  e^{- c_1 L^\psi } \ , \quad L \ge 2 . 
 \end{equation}
Then for every $\ell < \ell^\ast_c$ and $\psi' \in (0, \psi) \cap (0, 1]$, there exists a $c_2> 0$ such that
\begin{equation}
\label{e:armbound}
 \P[ \ann_\ell( R) ] \le e^{- c_2 R^{\psi'}} \ , \quad R \ge 1.
 \end{equation}
\end{proposition}

Let us show that Proposition \ref{p:nontriv} and the upper bound in the case $\alpha > 1$ of Theorem \ref{t:gen} follow from Proposition \ref{p:sb}:

\begin{proof}[Proof of Proposition \ref{p:nontriv}]
The fact that $\ell_c^\ast \le \min\{ \ell_c,  \ell_c^{\ast \ast}\}$ is clear by definition. Moreover, it is shown in \cite{mrv20} that, $\ell_c^\ast = \ell_c = 0$ for all fields we consider in dimension $d=2$. By restricting to a plane, this also implies that $ \ell_c  \le 0$ in all dimension $d \ge 2$. It remains to observe that if either $\alpha > 0$ or $\alpha = 0 $ and $\gamma > 1$, then by \eqref{e:gsup2}--\eqref{e:gsup3} the assumptions in the first statement of Proposition \ref{p:sb} are satisfied (in the $\alpha = 0$ case one must choose $\delta < \gamma-1$ in order for \eqref{e:tail} to hold, which leads to the restriction that $\gamma > 1$), and hence $\ell^\ast_c > -\infty$. 
\end{proof}

\begin{proof}[Proof of upper bound in the case {$\alpha > 1$ of Theorem \ref{t:gen}}]
By Proposition \ref{p:locglo} and \eqref{e:gsup2}, both assumptions in Proposition \ref{p:sb} are satisfied, and in \eqref{e:tail2} one may take any $\psi \in (0, \alpha)$. Then the conclusion follows by setting $\psi \in (\alpha,1)$ and then $\psi' = 1$ in \eqref{e:armbound}.
\end{proof} 

\begin{remark}
For later use we note that both assumptions in Proposition~\ref{p:sb} are satisfied by the fields considered in Theorem \ref{t:gen} for any $\alpha > 0$, and one may take any $\psi \in (0,\alpha)$ in \eqref{e:tail2}. In particular, if $\alpha \le 1$ then \eqref{e:armbound} holds for any $\psi' \in (0,\alpha)$.
\end{remark}

Before proving Proposition \ref{p:sb} we isolate the (deterministic) construction that underpins the proof (see Figure \ref{f:boot1}). It will be convenient to work with annuli defined with respect to the sup-norm. As such, for $R > 0$, we define $\Lambda_R = [-R, R]^d$ and $A^\infty(R) =  \Lambda_{2R} \setminus \Lambda_{R/\sqrt{d}}$. For a path in $\R^d$, we say that the path \textit{crosses} $A^\infty(R)$, or a translated copy, if it intersects both $\partial \Lambda_{R/\sqrt{d}}$ and $\partial \Lambda_{2 R}$. 

\begin{lemma}
\label{l:sb}
There exists a constant $c_d \ge 1$, depending only on the dimension, such that for every $R_1 \ge 1$ and $R_2 \ge 2R_1$ there exist collections $(A_i^1)_{i \le n}$ and $(A_j^2)_{i \le n}$, $n = c_d (R_2/R_1)^d$, of translated copies of $A^\infty(R_1)$ which satisfy:
\begin{enumerate}
\item $\text{dist}(A^1_i, A^2_j) \ge (2 - 1/\sqrt{d})(R_2-2R_1)$ for all $i,j$;
\item If $\gamma$ is a path that crosses $A^\infty(R_2)$ then there exist $i,j$ such that $\gamma$ crosses $A^1_i$ and $A^2_j$.
\end{enumerate}
\end{lemma}
\begin{proof}
Choose the collection $(A_i^1)_{i \le n}$ so that the inner square of each $A_i^1$ (which is a translated copy of $\Lambda_{R_1/\sqrt{d}}$) is contained in $\Lambda_{R_2 / \sqrt{d}}$ (i.e.\ the inner square of $A^\infty(R_2)$) and has a boundary face that is contained in a boundary face of $\Lambda_{R_2 / \sqrt{d}}$, and so that the union of the inner squares covers $\partial \Lambda_{R_2 / \sqrt{d}}$. Similarly choose the collection $(A_j^2)_{j \le n}$ so that the inner square of each $A_j^2$ is contained in $\Lambda^c_{2 R_2}$ and has a boundary face that is contained in a boundary face of $\Lambda_{2 R_2}$, and so that the union of the inner squares covers $\partial \Lambda_{2 R_2}$.  It is simple to check that these collections satisfy the requirements of the lemma.
\end{proof}

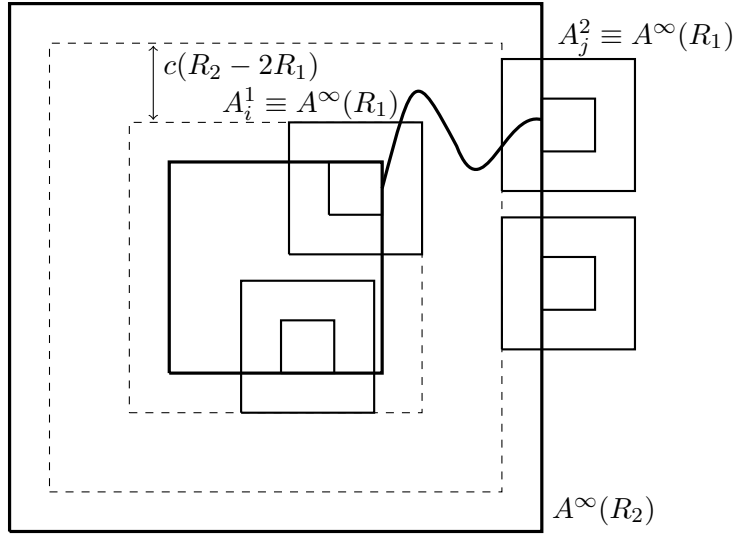
\begin{figure}
\begin{tikzpicture}[scale=0.7]
\path (-8,0);
\draw[very thick] (-2,-2) -- (2,-2) -- (2,2) -- (-2,2) -- (-2,-2);
\draw[dashed] (-2.75,-2.75) -- (2.75,-2.75) -- (2.75,2.75) -- (-2.75,2.75) -- (-2.75,-2.75);
\draw[dashed] (-4.25,-4.25) -- (4.25,-4.25) -- (4.25,4.25) -- (-4.25,4.25) -- (-4.25,-4.25);
\draw[very thick] (-5,-5) -- (5,-5) -- (5,5) -- (-5,5) -- (-5,-5);
\draw[thick] [shift={(5.5,-0.3)}] (-0.5,-0.5) -- (0.5,-0.5) -- (0.5,0.5) -- (-0.5,0.5) -- (-0.5,-0.5);
\draw[thick] [shift={(5.5,-0.3)}] (-1.25,-1.25) -- (1.25,-1.25) -- (1.25,1.25) -- (-1.25,1.25) -- (-1.25,-1.25);
\draw[thick] [shift={(5.5,2.7)}] (-0.5,-0.5) -- (0.5,-0.5) -- (0.5,0.5) -- (-0.5,0.5) -- (-0.5,-0.5);
\draw[thick] [shift={(5.5,2.7)}] (-1.25,-1.25) -- (1.25,-1.25) -- (1.25,1.25) -- (-1.25,1.25) -- (-1.25,-1.25);
\draw[thick] [shift={(1.5,1.5)}] (-0.5,-0.5) -- (0.5,-0.5) -- (0.5,0.5) -- (-0.5,0.5) -- (-0.5,-0.5);
\draw[thick] [shift={(1.5,1.5)}] (-1.25,-1.25) -- (1.25,-1.25) -- (1.25,1.25) -- (-1.25,1.25) -- (-1.25,-1.25);
\draw[thick] [shift={(0.6,-1.5)}] (-0.5,-0.5) -- (0.5,-0.5) -- (0.5,0.5) -- (-0.5,0.5) -- (-0.5,-0.5);
\draw[thick] [shift={(0.6,-1.5)}] (-1.25,-1.25) -- (1.25,-1.25) -- (1.25,1.25) -- (-1.25,1.25) -- (-1.25,-1.25);
\draw[very thick] (2,1.5) .. controls (2.6,3.8) .. (3.4,2.3) .. controls (3.9,1) and (4.3,3) .. (5,2.8);
\node[right] at (5,-4.6) {$A^\infty(R_2)$};
\node[right] at (-1.2,3.1) {$A_i^1 \equiv A^\infty(R_1)$};
\node[right] at (5.1,4.35) {$A_j^2 \equiv A^\infty(R_1)$};
\draw[<->] (-2.3,2.8) -- (-2.3,4.2);
\node[right] at (-2.3,3.8) {$c (R_2 - 2R_1)$};
\end{tikzpicture}
\caption{Illustration of Lemma \ref{l:sb}. A path that crosses the annulus $A^\infty(R_2)$ must cross at least one annulus in each of the collections $(A_i^1)_{i \le n}$ and $(A_j^2)_{i \le n}$ of translated copies of $A^\infty(R_1)$ (only partially depicted).}
\label{f:boot1}
\end{figure}

\begin{proof}[Proof of Proposition \ref{p:sb}]
In the proof $c_i > 0$ are constants which may depend on the field and may change from line to line. Let $\ann^\infty_\ell(R)$ denote the event that $\{f \le \ell \}$ contains a path that crosses $A^\infty(R)$. It is clear that
\begin{equation}
\label{e:comparearms}
  \arm_\ell(1, R) \subset  \ann^\infty_\ell(R/(2\sqrt{d})) \subset \ann_\ell(R/(2\sqrt{d})) 
  \end{equation}
for all sufficiently large $R$. Since $\{f \in  \ann^\infty_{\ell}(R)  \}$ is increasing in $\ell$, for any $L \ge 1$, $R_1,R_2> 0$, $\ell \in \R$, and $\eps > 0$ we have
\begin{equation}
\label{e:sb1}
 \big\{f_L   \in \ann^\infty_{\ell - \eps}(R_1)  \big\} \cap \big\{ \sup_{x \in \Lambda_{2R_1} }  g_L(x)    \le  \eps  \big\}   \Longrightarrow \big\{ f \in \ann^\infty_\ell(R_1)  \big\} , 
 \end{equation}
 and
 \begin{equation}
\label{e:sb2}
 \big\{ f  \in \ann^\infty_{\ell - 2\eps}(R_2)   \big\}  \cap \big\{ \inf_{x \in \Lambda_{2R_2} }  g_L(x)    \ge - \eps \big\}   \Longrightarrow \big\{ f_L \in \ann^\infty_{\ell - \eps}(R_2)  \big\}.
 \end{equation}
 Then for every $R_1 \ge 1$, $R_2 \ge 2R_1$, $\ell \in \R$, and $\eps > 0$, by applying \eqref{e:sb1}--\eqref{e:sb2} with $L = (2-1/\sqrt{d})(R_2 - 2R_1)$, and using the construction in Lemma \ref{l:sb}, the fact that $f_L$ is $L$-range dependent, and the union bound, we deduce that
\begin{align}
\label{e:sb3}
&  \P [ f  \in \ann^\infty_{\ell - 2\eps}(R_2) ]   \le  \P[   f_L \in \ann^\infty_{\ell - \eps}(R_2)  ] + \P \big[\sup_{x \in \Lambda_{2R_2} }  |g_L(x)|   \ge \eps \big]   \\
\nonumber & \qquad \le  c_d^2 (R_2/R_1)^{2d} \P[   f_L \in \ann^\infty_{\ell - \eps}(R_1) ]^2 + \P \big[\sup_{x \in \Lambda_{2R_2} }  |g_L(x)|   \ge \eps \big]   \\
  \nonumber & \qquad \le  c_d^2 (R_2/R_1)^{2d} \Big(  \P[   f  \in \ann^\infty_{\ell}(R_1) ] +  \P \big[\sup_{x \in \Lambda_{2R_1} }  |g_L(x)|   \ge \eps \big]    \Big)^2 + \P \big[\sup_{x \in \Lambda_{2R_2} }  |g_L(x)|   \ge \eps \big]   \\
\nonumber & \qquad \le c_d^2   (R_2/R_1)^{2d} \P [ f  \in \ann^\infty_{\ell}(R_1) ]^2 +  2\big(c_d^2  (R_2/R_1)^{2d} + 1 \big)  \P \big[\sup_{x \in \Lambda_{2R_2} }  |g_L(x)|   \ge \eps \big]  .
 \end{align}

We turn to the first statement. Fix $\eta$ so that $\delta'<\eta < \delta$, where $\delta', \delta$ are as in the statement of the proposition. Fix $n_0 \ge 1$ and $\ell_{n_0} \in \R$ to be chosen later (none of the subsequent constants depends on $n_0$ or $\ell_{n_0}$), and define sequences $(R_n)_{n \ge n_0}$ and $(\ell_n)_{n \ge n_0}$ as
\[ R_n = e^{n^{1/\eta}} \quad \text{and} \quad \ell_{n+1} =   \ell_n - 2  (\log L_n)^{-\delta} , \]
where $L_n = (2-1/\sqrt{d})(R_{n+1} - 2R_n)$ (we may take $n_0$ sufficiently large so that $L_n \ge 1$). Since $(\log R_{n})^{-\delta}$ is summable, so is  $(\log L_n)^{-\delta}$, and in particular we have  $\ell_\infty := \inf_{n \ge n_0} \ell_n > - \infty$.  Also by choosing $n_0$ large enough we have
 \begin{equation}
 \label{e:rn}
  R_{n+1} \le   e^{c_1 n^{1/\eta-1}} R_n \le R_n^2 , \quad n \ge n_0 .
  \end{equation}
 
 Now consider $a_n = \P[ \ann^\infty_{\ell_n}( R_n) ] $. Applying \eqref{e:sb3} (with the settings $R_1 = R_n$, $R_2 = R_{n+1}$, $\ell = \ell_n$ and $\eps = \eps_n =  (\log L_n)^{-\delta}$) we have
\begin{equation}
\label{e:boot2}
 a_{n+1} \le c_d^2 (R_{n+1}/R_n)^{2d} a_n^2 + 2\big(c_d^2 (R_{n+1}/R_n)^{2d}  + 1\big)  \P \big[\sup_{x \in \Lambda_{2R_{n+1}} }  |g_{L_n}(x)|   \ge \eps_n \big]   .
 \end{equation}
  By the union bound and \eqref{e:tail}, for sufficiently large $n_0$,
\begin{align}
\label{e:boot3}
 \P \big[\sup_{x \in \Lambda_{2R_{n+1}} }  |g_{L_n}(x)|   \ge \eps_n \big]   &  \le c_1  \P \big[\sup_{x \in B(L_n) }  |g_{L_n}(x)|   \ge \eps_n \big]    \\
 \nonumber & \le c_1 e^{- c_2  (\log L_n)^{1-\delta'}}  \le  c_3 e^{- c_4 n^{ ( 1 - \delta')/\eta }   }   .
 \end{align}
Combining \eqref{e:rn}--\eqref{e:boot3}, recalling that $\delta' < \eta$, and adjusting the constants, we deduce that
\begin{equation}
\label{e:sbrec}
 a_{n+1} \le  e^{c_1 n^{\rho} }   a_n^2  +    e^{- c_2 n^{ \rho'} }   , \quad n \ge n_0,
\end{equation}
where $\rho = 1/\eta-1 > 0$ and $\rho' =( 1 - \delta')/\eta$ satisfy $0 < \rho < \rho'$.

We next prove by induction that, if $n_0$ and $\ell_{n_0}$ are sufficiently large and small respectively, \eqref{e:sbrec} implies that 
\[ a_n \le   e^{- (3c_2/4) n^{\rho'} }   \ , \quad n \ge n_0, \]
For this, first take $n_0$ large enough so that
\begin{equation}
\label{e:boot4}
 e^{c_1 n^\rho}   e^{-  (3c_2/2) n^{ \rho'}  }   + e^{-   c_2 n^{ \rho'}  }  \le      e^{- (3c_2/4) (n+1)^{\rho'} }  , \quad n \ge n_0 .  
 \end{equation}
Then the base case $n=n_0$ holds by taking $\ell_{n_0}$ sufficiently small. For the induction step, combining \eqref{e:sbrec} and \eqref{e:boot4} we have
\begin{align*}
 a_{n+1} & \le e^{c_1 n^\rho}   a_n^2 +  e^{- c_2 n^{ \rho'} } \le   e^{c_1 n^\rho}   e^{-  (3c_2/2) n^{ \rho'}  }   + e^{-   c_2 n^{ \rho'}  }  \le       e^{- (3c_2/4) (n+1)^{\rho'} }   .  
 \end{align*}
Hence by monotonicity we conclude that 
\[ \P[ \ann^\infty_{\ell_\infty}(R_n)  ]   \le  e^{- c_3 n^{\rho'} } = e^{ - c_3 (\log R_n)^{1-\delta'} }    \]
for all $n \ge n_0$, from which we deduce (recall that $R_{n+1} \le R_n^2$) that
\[ \P[ \ann^\infty_{\ell_\infty}(R)  ]   \le   e^{ - c_4 (\log R)^{1-\delta'} }    \]
for all sufficiently large $R$ by monotonicity. Recalling \eqref{e:comparearms}, this shows that $\ell^\ast_c \ge \ell_\infty > -\infty$.  

\smallskip
We turn to the second statement of Proposition \ref{p:sb}, which is proven similarly. Let $\ell < \ell^\ast_c$ and $\psi' \in (0, \psi) \cap (0,1]$ be given. Fix a constant $R_0 \ge 1$ to be chosen later, and for each $n \ge 1$ define $R_n$ recursively as
\[ R_{n+1} = 2R_n + R_n^{\psi'/\psi}  \in [2 R_n, 3 R_n]  . \]
In particular there exists a $c_0 > 0$, depending on $R_0$, such that
\begin{equation}
\label{e:rbounds}
R_0 2^n \le R_n \le c_0 2^n . 
\end{equation}
Indeed the lower bound is immediate, and the upper bound follows from the fact that
\[  \log ( R_{n+1} /  2^{n+1} ) - \log ( R_n / 2^n ) = \log (  1 + R_n^{\psi'/\psi - 1}  / 2  )   \le R_n^{\psi'/\psi - 1} / 2    \]
is summable. We also define $\ell_0 = \ell + (\ell_c^\ast - \ell)/2$ and  $\ell_{n+1} =   \ell_n - 2 (\log L_n)^{-(1+\delta)}$, where  $L_n = (2-1/\sqrt{d})  R_n^{\psi'/\psi}$ (again choose $R_0$ large enough so that $L_n \ge 1$). Since $(\log L_n)^{-(1+\delta)}$ is summable, we may choose $R_0$ sufficiently large so that $\inf_{n \ge 1} \ell_n \ge \ell$.

Again consider $a_n = \P[ \ann^\infty_{\ell_n}( R_n) ]$, and observe that since $\ell_0 < \ell_c^\ast$ and by \eqref{e:comparearms}, one can make $a_0$ arbitrarily small by choosing $R_0$ sufficient large. Using a similar argument as in the proof of the first statement, we obtain
\[ a_{n+1} \le c_d^2 a_n^2 - (2c_d + 1)  \P \big[\sup_{x \in \Lambda_{6R_n} }  |g_{L_n}(x)|   \ge \eps_n \big] \ , \quad n \ge 0,  \]
where $\eps_n = (\log L_n)^{-(1+\delta)}$.  By \eqref{e:tail2} and the union bound, for sufficiently large $R_0$,
\[  \P \big[\sup_{x \in \Lambda_{6R_n} }  |g_{L_n}(x)|   \ge \eps_n \big]    \le c_1  \P \big[\sup_{x \in B(R_n) }  |g_{L_n}(x)|   \ge \eps_n \big]   \le c_1 e^{- c_2  L_n^\psi }  = c_1 e^{- c_2 R_n^{\psi' } }  . \]
Combining these, and adjusting the constants, we deduce that
\begin{equation}
\label{e:sbrec2}
a_{n+1} \le c_1 a_n^2 +  e^{- c_2 R_n^{\psi' } }    \ , \quad n \ge 0 .
\end{equation}
To analyse \eqref{e:sbrec2} define $a_n' = c_1(a_n + e^{-(c_2/2) R_n^{\psi'} } )$. Then we have
\begin{align*}
a_{n+1}' & = c_1(a_{n+1} + e^{-(c_2/2) R_{n+1}^{\psi'} }  )    \le  c_1^2 a_n^2 + c_1 e^{-c_2 R_n^{\psi'} } + c_1 e^{-(c_2/2) 2^{\psi'} R_n^{\psi'} }  \\
& \le c_1^2 a_n^2 + 2 c_1 e^{-(c_2/2) 2^{\psi'} R_n^{\psi'} } \le (c_1 a_n)^{2^{\psi'}} +  \big(c_1  e^{-(c_2  / 2)   R_n^{\psi'} } \big)^{2^{\psi'}} \le (a_n')^{2^{\psi'}}  ,
\end{align*}
where the first inequality used \eqref{e:sbrec2} and that $R_{n+1} \ge 2R_n$, the second that $\psi' \le 1$, and the third assumed that $c_1 \ge 1$ is fixed sufficient large so that $2 c_1 \le c_1^{2^{\psi'}}$, and then $R_0$ is taken sufficient large so that  $c_1 a_0 < 1$. We conclude that there exists a $c_3 > 0$ such that
\[ a_n \le e^{- c_3 2^{\psi' n} } \ , \quad n \ge 0. \]
Recalling \eqref{e:rbounds}, this implies that $a_n \le  e^{-c_1 R_n^{\psi'} }$, from which we deduce that
\begin{equation}
\label{e:apriori2}
  \P[ \ann^\infty_{\ell}( R) ] \le  e^{-c_2 R^{\psi'}}  \, \quad R \ge 1, 
  \end{equation}
by monotonicity. Considering \eqref{e:comparearms}, this gives the result.
\end{proof}

\smallskip
\section{Renormalisation and sharp upper bounds}
\label{s:seup}

In this section we prove the upper bounds in Theorem \ref{t:gen} in the `sub-exponential' cases. Throughout this section we assume that the conditions of Theorem \ref{t:gen} hold for $\alpha \in [0, 1]$, and in particular a local-global decomposition is available. It will sometimes be convenient to use the notation $g(x) \ll h(x)$ and $g(x) \gg h(x)$ to mean $g(x) = o(h(x))$ and $h(x) =  o(g(x))$ respectively.

\smallskip
As in the previous section we work with annuli defined with respect to the sup-norm. In particular, we recall the annulus $A^\infty(R)$ and the crossing event $\ann^\infty_\ell(R)$ from the previous section. The main ingredient in the proof is a bootstrap for the probability of $\arm_\ell(R)$ given a priori bounds on $\ann^\infty_\ell(L)$ on a mesoscopic scale $1 \ll L \ll R$:

\begin{proposition}[Bootstrap for annular crossings]
\label{p:boot}
Let $L = L(R) \to \infty$ satisfy, as $R \to \infty$,
\begin{equation}
\label{e:subcon}
 \begin{cases}
LK(L) \ll R K(R)  &  \alpha \in [0,1),  \\   \log L \ll \log R   & \alpha = 1.
\end{cases} 
\end{equation}
Then for every $\ell<\ell_\ast$ and $\delta > 0$ there exists a $c > 0$ such that, as $R \to \infty$, eventually
\begin{align*}
&   - \log  \P[\arm_\ell(R)]      \\
\nonumber & \qquad \qquad  \ge \min \Big\{   \frac{ ( (\ell^\ast_c-\ell)^2 - \delta) \capa_K([0, R]) }{2} ,    \frac{ P R}{c L}  , \frac{R}{c L K(L)}   \Big\} -   \frac{ cR}{L} \log (R/L)   ,
\end{align*}
where $P = P(R) =  - \log \P[\ann^\infty_\ell(L)]$. In particular, if 
\[   \frac{R}{L} \log(R/L) \ll \capa_K([0, R]) \ll \frac{R}{L} \min\Big\{ P , \frac{1}{ K(L) } \Big\}  \]
as $R \to \infty$, then
\[  - \log  \P[\arm_\ell(R)]   \geq \frac{(\ell^\ast_c-\ell)^2  \capa_K([0, R]) }{2} (1+o(1))  .\]
\end{proposition}

Once again we isolate the construction that underpins the bootstrap (see Figure \ref{f:boot2}):

\begin{lemma}
\label{l:sb2}
There exists a $c \ge 1$, depending only on the dimension, such that for every $R, M, L \ge 6 \sqrt{d}$ there exists a set of $n = \lfloor R/(6ML) \rfloor$ collections $(A_i^k)_{i \le c  R^{d-1}/ L^{d-1}}$, $k = 1 , \ldots, n$, each consisting of translated copies of $A^\infty(L)$, which satisfy:
\begin{enumerate}
\item $\text{dist}(A^{k_1}_i, A^{k_2}_j) \ge (M-4\sqrt{d})L |k_1 - k_2|$ for all $k_1 \neq k_2$ and all $i,j$;
\item If $\gamma$ is a path that intersects $\partial B(1)$ and $\partial B(R)$ then for each $k$ there exists an $i$ such that $\gamma$ crosses $A^k_i$.
\end{enumerate}
\end{lemma}
\begin{proof}
For $k = 1,  \ldots , \lfloor R/(6ML) \rfloor$ consider the sphere $S_k = \{ x \in \R^d : |x| = kML\}$ and define the collection $(A_i^k)_{i \le c  R^{d-1}/ L^{d-1}}$ such that each $A_i^k$ is centred on $S_k$, and the union over $i$ of the inner squares of $A_i^k$ (which are translated copies of $\Lambda_{L / \sqrt{d} })$ covers $S_k$. These collections satisfy the required properties.
\end{proof}

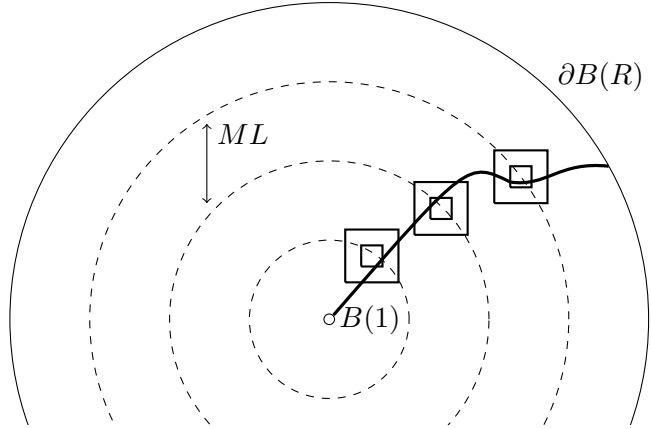
\begin{figure}
\begin{tikzpicture}[scale=0.7]
\clip (-10,-2) rectangle + (20,10);
\draw (0,0) circle (0.1);
\draw (0,0) circle (6);
\draw[dashed] (0,0) circle (1.5);
\draw[dashed] (0,0) circle (3);
\draw[dashed] (0,0) circle (4.5);
\draw[thick] [shift={(0.8,1.2)}] (-0.2,-0.2) -- (0.2,-0.2) -- (0.2,0.2) -- (-0.2,0.2) -- (-0.2,-0.2);
\draw[thick] [shift={(0.8,1.2)}] (-0.5,-0.5) -- (0.5,-0.5) -- (0.5,0.5) -- (-0.5,0.5) -- (-0.5,-0.5);
\draw[thick] [shift={(2.1,2.1)}] (-0.2,-0.2) -- (0.2,-0.2) -- (0.2,0.2) -- (-0.2,0.2) -- (-0.2,-0.2);
\draw[thick] [shift={(2.1,2.1)}] (-0.5,-0.5) -- (0.5,-0.5) -- (0.5,0.5) -- (-0.5,0.5) -- (-0.5,-0.5);
\draw[thick] [shift={(3.6,2.7)}] (-0.2,-0.2) -- (0.2,-0.2) -- (0.2,0.2) -- (-0.2,0.2) -- (-0.2,-0.2);
\draw[thick] [shift={(3.6,2.7)}] (-0.5,-0.5) -- (0.5,-0.5) -- (0.5,0.5) -- (-0.5,0.5) -- (-0.5,-0.5);
\draw[very thick] (0.08,0.08) .. controls (2.6,3) .. (3.4,2.6) .. controls (4.1,2.5) and (4.3,3) .. (5.25,2.9);
\node[right] at (0,0) {$B(1)$};
\node[right] at (4.1,4.6) {$\partial B(R)$};
\draw[<->] (-2.3,2.2) -- (-2.3,3.7);
\node[right] at (-2.3,3.5) {$ML$};
\end{tikzpicture}
\caption{Illustration of Lemma \ref{l:sb2}. A path that intersects $\partial B(1)$ and $\partial B(R)$ must cross at least one annulus in each of the collections $(A_i^k)_{i \le n}$ of translated copies of $A^\infty(L)$ centred along concentric rings (only partially depicted).}
\label{f:boot2}
\end{figure}

\begin{proof}[Proof of Proposition \ref{p:boot}]
Fix $s = s(R) \to \infty$ that satisfies, in the case $\alpha \in [0,1)$,
\begin{equation}
\label{e:condboot1}
  K(L) \ll K(s) \ll \frac{R K(R)}{L } , 
  \end{equation}
and in the case $\alpha = 1$,
\begin{equation}
\label{e:condboot2}
 K(L) \ll K(s) \ll \frac{R K(R) (\log R) }{ L (\log L)} \le \frac{R K(R) (\log R) }{ (\log s)} ,
 \end{equation}
which is possible by \eqref{e:subcon}. We may also assume that $s \ll L \ll R$, since otherwise \eqref{e:condboot1}--\eqref{e:condboot2} cannot hold.

Fix $\delta > 0$ such that $2\delta < \ell^\ast_c - \ell$, and fix also $M \ge 6 \sqrt{d}$ and $\rho \in (0, 1)$.  Consider the set of collections $(A_i^k)_{i \le c  R^{d-1}/ L^{d-1}}$, $k = 1,  \ldots , \lfloor R/(6ML) \rfloor $, of translated copies of the annulus $A^\infty(L)$ guaranteed by Lemma \ref{l:sb2}, and for $\gamma \in (0, 1)$ let $\mathcal{C}_\gamma$ denote the collection of all sets $C = (A^k_{i_k})_k$ containing at most one annulus $A_i^k$ from each collection $(A_i^k)_i$ and at least $  \lfloor \gamma R/(4ML) \rfloor $ annuli in total; this has cardinality at most $|\mathcal{C}_\gamma| \le (R/L)^{c_d R/L}$ for some constant $c_d > 0$ depending only on the dimension. Recall the local-global decomposition $f = f_L + g_L$, and observe that, by the second property of Lemma \ref{l:sb2}, for sufficiently large $R$ the event $\arm_\ell(R)$ implies that one of the following occur: 
 \begin{itemize}
 \item $E_1 = \cup_{C \in \mathcal{C}_\rho} \{ \{g_L \ge \ell^\ast_c - \ell - 2\delta\} \text{ has non-empty intersection with each annulus in } C \}$;
 \item $E_2 =  \cup_{C \in \mathcal{C}_{1-\rho}}\{ \{ f_L \le \ell^\ast_c - 2\delta \} \text{ crosses each annulus in } C  \}$. 
 \end{itemize} 
We bound the probability of these events separately. For the first event, recall that $\Gamma_s(A_i)$ denotes the ball that is concentric to $A_i$ with radius $s$. Then by the union bound, Theorem~\ref{t:ld3} (which applies since $K(L) \ll K(s)$), and Corollary \ref{c:conpro} (with the setting $r = (M-4)L$, which applies by \eqref{e:condboot1}--\eqref{e:condboot2}), for $R$ large enough (depending on $M$ and $\rho$),
 \begin{align*}
  -\log \P[E_1] & \ge \sup_{C \in \mathcal{C}_\rho}   \frac{ (\ell^\ast_c-\ell - 2\delta)^2 \capa_K( \cup_{A_i \in C} \Gamma_s(A_i)) (1- \gamma'_M) }{2} - \frac{c_d R}{L}\log(R/L)      \\
   & \ge   \frac{ (\ell^\ast_c-\ell - 2\delta)^2 \capa_K([0, R]) (1- \gamma'_M -  \gamma_{\rho,M})  }{2} - \frac{c_d R}{L}\log(R/L)       
\end{align*}
for some $\gamma'_M \to 0$ as $M \to \infty$, and $\gamma_{\rho,M}$ such that, for every $M > 0$, $\gamma_{\rho,M} \to 0$ as $\rho \to 1$. In particular we may choose $M$ sufficiently large, and then $\rho$ sufficiently close to $1$, so that
\[    -\log \P[E_1]  \ge   \frac{ (\ell^\ast_c-\ell - 3\delta)^2 \capa_K([0, R]) }{2} - \frac{c_d R}{L}\log(R/L)         \]
For the second event, by the union bound, stationarity, and the fact that $f_L$ is $L$-range dependent,  we have
\[ - \log \P[E_2] \ge    \frac{(1-\rho) R}{8ML}  \Big( - \log  \P[ f_L \in \ann^\infty_{\ell^\ast_c - 2\delta}(R)  ] \Big) -   \frac{c_d R}{L}\log(R/L)   . \]
Observe that $ \{f_L \in \ann^\infty_{\ell^\ast_c - 2\delta}(R) \} $ implies that
\[  \{ f \in \ann^\infty_{\ell^\ast_c - \delta}(R) \}  \quad \text{or} \quad  \big\{ \sup_{x \in B(2L)} g_L \ge \delta \big\}   .  \]
Hence by the tail control for $g_L$ in \eqref{e:gsup} we deduce that, as $R \to \infty$,
\[  -\log \P[E_2]  \ge    \min \Big\{   \frac{ (1-\rho) P R}{8ML}  , \frac{R}{O(L K(L))}     \Big\}    - \frac{c_d R}{L} \log (R/L)  . \]
  Combining the bounds on $\P[E_1]$ and $\P[E_2]$ gives the result.
 \end{proof}

To apply Proposition \ref{p:boot} we require an priori bound on $\P[ \ann^\infty_\ell(R) ] $. In the case $\alpha \in (0, 1]$ we use the bound established in Proposition \ref{p:sb} (more precisely we use \eqref{e:apriori2}), whereas in the case $\alpha = 0$ we use the following result:

\begin{theorem}[{\cite{mrv20}}]
\label{t:apriori}
Suppose $\alpha = 0$, so that $K(x) \sim (\log x)^{-\gamma}$, $\gamma > 0$. Then for every $\ell < \ell^\ast_c$ there exists a $c > 0$ such that, for all $R \ge 2$,
\[   \P[ \ann^\infty_\ell(R) ]  \le  e^{-c (\log R)^{\gamma'/2}}       \]
where $ \gamma' =  \min\{\gamma, 1\}$.
\end{theorem}
\begin{proof}
Recalling \eqref{e:comparearms} it suffices to prove the statement for $\ann_\ell(R)$, which has the advantage of being rotationally symmetric. Applying \cite[Propositions 2.14 and 2.15]{mrv20}, there exists $c_1,c_2 > 0$ such that, for every $\ell \in \R$ and $R \ge 2$,
\begin{equation}
\label{e:sharp}
  \P[  \ann_\ell(R )  ]  \begin{cases}  \le c_1 e^{-c_2 (\E[T_R] - \ell) \sqrt{ \min\{ \log R, 1/ \bar{K}(\sqrt{R} ) \} } } & \text{if } \ell \le \E[T_R]  , \\ \ge 1 - c_1 e^{-c_2 (\E[T_R] - \ell) \sqrt{ \min\{ \log R, 1/ \bar{K}(\sqrt{R} ) \} } } & \text{if } \ell \ge \E[T_R]   ,
\end{cases} 
\end{equation}
 where $T_R$ is a random variable with c.d.f.\ $F_{T_R}(\ell) = \P[ T_R \le \ell ] = \P[ \ann_\ell(R) ]$, and $\bar K(R) = \sup_{|x| \ge R}  |K(x)|$. Note that in our setting we have
 \[  \min\{ \log R, 1/ \bar{K}(\sqrt{R} ) \}   \asymp (\log R)^{\gamma'}   \to \infty .\]
 Now fix $\ell < \ell' < \ell^\ast_c$, so that $ \P[  \ann_{\ell'}(R )  ]   \to 0$. First note that $\liminf_{R \to \infty} \E[T_{R}] \ge \ell'$ eventually, since otherwise \eqref{e:sharp} implies that $ \P[  \ann_{\ell'}(R )  ]   \to 1$. Hence applying \eqref{e:sharp} we deduce that
 \[  \P[ \ann_\ell(R) ]  \le  c_3 e^{-c_4 (\log R)^{\gamma'/2} }  ,\]
 for some $c_3, c_4 > 0$. Adjusting constants, we have the result.
 \end{proof}
\begin{remark}
Although \cite[Propositions 2.14 and 2.15]{mrv20} were only stated in $d=2$, the proof goes through unchanged in all dimensions. The non-degeneracy assumption in \cite{mrv20} is satisfied since we assume that the spectral measure of $f$ contains an open set in its support. We also emphasise that the fact that $\ann_\ell(R)$ is rotationally symmetric and $f$ is isotropic was crucial in applying \cite[Propositions 2.14 and 2.15]{mrv20}. 
\end{remark}

We can now give the proof of the sub-exponential upper bounds in Theorem \ref{t:gen}, which consists of applying the bootstrap in Proposition \ref{p:boot} starting from our a priori bounds:

\begin{proof}[Proof of upper bound in Theorem \ref{t:gen}]
By \eqref{e:comparearms}, it suffices to prove the bounds for $\ann^\infty_\ell(R)$. We treat separately the cases (i) $\alpha = 0$, (ii) $\alpha \in (0, 1)$, and (iii) $\alpha = 1$. All the asymptotic statements in the proof will be as $R \to \infty$.

\smallskip \noindent \textbf{Case $\alpha = 0$.} Recall that in this case $K(x) \sim (\log x)^{-\gamma}$ for $\gamma > 0$. Define the positive constant
\[ \mu \in \big( \gamma -   \min\{ \gamma,1 \} / 2 , \gamma \big)  ,\]
and the scale
\[ L = R (\log R)^{-\mu} .\]
By Theorem \ref{t:apriori} we have the a priori bound 
\[ P = - \log \P[ \ann^\infty_\ell(L) ] \ge c_1  (\log L)^{\min\{\gamma,1\} /2}  \ge c_2 (\log R)^{\min \{\gamma,1\}/2 }  ,  \]
 for some $c_1,c_2 > 0$. One can check that $L K(L)  \ll  R K(R)$, and also that, by our choice of $\mu$,
\begin{equation}
\label{e:capdom}
   \frac{R}{L} \log(R/L) \ll (\log  R)^\gamma \ll \frac{R}{L} \min\Big\{P , \frac{1}{K(L) } \Big\} .  
   \end{equation}
Indeed, the left-hand side of \eqref{e:capdom} is asymptotic to
\[ (\log R)^\mu (\log \log R)  \ll (\log R)^\gamma ,\]
and the right-hand side of \eqref{e:capdom} is eventually greater than 
\[ c_2 (\log R)^\mu  (\log R)^{\min \{\gamma,1\}/2 }  \gg (\log R)^\gamma  .\]
Recall that $\capa_K([0,R]) \sim 1/K(R) \sim  (\log R)^\gamma $ by Proposition \ref{p:capline}). Then applying Proposition~\ref{p:boot} we have
\[ - \log  \P[\ann^\infty_\ell(R)]  \ge  \frac{ c_\alpha(\ell^\ast_c-\ell)^2 (1 + o(1) )}{2 K(R)}   .\]

\smallskip \noindent  \textbf{Case $\alpha \in (0,1)$.} In this case we have the a prior bound (see \eqref{e:apriori2})
\begin{equation}
\label{e:apriori01}
P =  - \log  \P[\ann^\infty_\ell(L)]   \gg L^{\psi'} 
\end{equation}
 for any $\psi' < \alpha$. Define $L = R K(R) (\log R)^2$. Again one can check that $L K(L) \ll  R K(R)$, and also that
\[   \frac{R}{L} \log(R/L) \ll \frac{1}{K(R)} \ll \frac{R}{L} \min\Big\{ P , \frac{1}{ K(L) } \Big\}  . \]
Recalling that $\capa_K([0,R]) \sim  c_\alpha / K(R)$ (Proposition \ref{p:capline}), in this case Proposition~\ref{p:boot} yields
\[ - \log  \P[\ann^\infty_\ell(R)]  \ge  \frac{ c_\alpha(\ell^\ast_c-\ell)^2 (1 + o(1) )}{2 K(R)}   .\]
 
 \smallskip \noindent  \textbf{Case $\alpha =1$.} 
Recall that in this case $K(x) \sim x^{-1} (\log x)^{-\gamma}$ for $\gamma < 1$, and that we also have the a prior bound \eqref{e:apriori01} for any $\psi' < 1$ (again see \eqref{e:apriori2}). Define $L = (\log R)^{2-\gamma}$. In this case one can check that $L  = R^{o(1) }$, and also that
\[   \frac{R}{L} \log(R/L) \ll R (\log R)^{\gamma-1} \ll \frac{R}{L} \min\Big\{ P , \frac{1}{ K(L) } \Big\}  . \]
Recalling that  (Proposition \ref{p:capline})
\[ \capa_K([0, R])  \sim    \frac{1-\gamma}{2}  R  (\log R)^{\gamma-1}    \sim  \frac{1-\gamma}{2 K(R) (\log R) }   ,\]
in this case Proposition \ref{p:boot} yields 
\begin{equation*}
 - \log  \P[\ann^\infty_\ell(R)]  \ge    \frac{ (\ell^\ast_c-\ell)^2  (1 + o(1) ) }{4(1-\gamma)^{-1} K(R) (\log R) }  . \qedhere
 \end{equation*}
\end{proof}

\smallskip
\section{Proof of the remaining results}
\label{s:lb}

In this section we prove the remaining results, namely Proposition \ref{p:cl}, the lower bounds in Theorem \ref{t:gen}, and Theorem~\ref{t:diamgen}.

\begin{proof}[Proof of Proposition \ref{p:cl}]
By isotropy and the equality in law of $f$ and $-f$ we have
\[ \P[ \cross_0(R) ]   = 1/2 . \]
Also, by Proposition \ref{p:capdom} there is a $c_1 > 0$ such that $\capa_K([0, R]^2)  \le c_1 R^{-\alpha}$ for every $R \ge 1$. Hence by Corollary \ref{c:eb}, for every $R \ge 1$ and $\ell < 0$,
\[  \P[ \cross_\ell(R) ]  \ge  \frac{1}{2 e} e^{ - \frac{\ell^2 R^\alpha}{c_1} } . \]
On the other hand, by continuity one can find $\ell < 0$ sufficiently close to $0$ so that  $\P[ \cross_\ell(R) ]  \ge 1/4$ for all $R \le 1$. Recalling that the correlation length $\xi(\ell)$ is defined with $\eps \in (0, 1/(2e) ) \in (0, 1/4)$, this implies that for $\ell < 0$ sufficiently close to $0$ the correlation length $\xi(\ell)$ satisfies
  \[ \frac{1}{2 e} e^{ - \frac{\ell^2 \xi(\ell)^\alpha}{c_1} } \le \eps . \]
  Rearranging gives that
\[ \xi(\ell)   \ge  c_2 \ell^{-2/\alpha} \]
where $c_2 = c_1 \log(1/(2 e \eps)) )^{1/\alpha} > 0$, as required.
\end{proof}

\begin{proof}[Proof of lower bounds in Theorem \ref{t:gen}]
For the case $\alpha > 1$ see Remark \ref{r:pa}. Let us consider the case $\alpha \in (0, 1]$. Fix $\delta > 0$ and recall the tube $T(R; \rho) = [0, R] \times [0, R^{\rho}]^{d-1}$ and the tube crossing event $\tube_\ell(R; \rho)$. By the definition of $\ell^{\ast \ast}_c$, there exists a $\rho = \rho(R) \ge 0$ (depending on $\delta$) such that $\rho \to 0$ and $\P[\tube_{\ell^{\ast \ast}_c + \delta}(R; \rho) ] \to 1$ as $R \to \infty$. Applying Corollary \ref{c:eb} we have, as $R \to \infty$,
\begin{align*}
\P[  \tube_\ell(R; \rho) ]  & \ge    \P[\tube_{\ell^{\ast \ast}_c + \delta}(R; \rho) ]  \exp \Big( -  \frac{(\ell_c^{\ast \ast} + \delta - \ell)^2 \capa_K(T(R;\rho)) }{2  \P[\tube_{\ell^{\ast \ast}_c + \delta}(R; \rho) ] }  - 1 \Big)  \\
& =   (1 + o(1) ) \exp \Big( -  \frac{(\ell_c^{\ast \ast} + \delta - \ell)^2 \capa_K(T(R;\rho)) (1 + o(1)) }{2 }  - 1 \Big)  .
\end{align*}
Since $\capa_K(T(R;\rho)) \sim \capa_K([0, R]) \to \infty$ as $R \to \infty$ by Proposition \ref{p:capline}, by taking $R \to \infty$ and then $\delta \to 0$ we deduce that, as $R \to \infty$
\[ - \log \P[  \tube_\ell(R; \rho) ]  \le  \frac{(\ell_c^{\ast \ast} - \ell)^2 \capa_K([0, R]) (1+o(1)) }{2} . \]
Finally, by the union bound it is evident that
\[ \P[ \arm_\ell(1, R) ]   \ge  c R^{-d} \P[ f \in \tube_\ell(R; \rho) ]   \]
for some $c > 0$, and the result follows by the asymptotics of the capacity in Proposition \ref{p:capline}.

\smallskip 
The argument is similar in the case $\alpha = 0$, except it is sufficient to work with $\ell_c$ instead of $\ell^{\ast \ast}_c$ since the capacity of $B(R)$ is comparable to the capacity of $[0, R]$ (see Proposition \ref{p:capdom}). By Corollary \ref{c:eb} we have, for any $\delta > 0$,
\[  \P[  \arm_\ell(R) ] \ge    \P[\arm_{\ell_c + \delta}(R) ]  \exp \Big( -  \frac{(\ell_c + \delta - \ell)^2 \capa_K(B(R))}{2  \P[\arm_{\ell_c + \delta}(R) ] }  - 1 \Big)   .\]
By the definition of $\ell_c$ we have, as $R \to \infty$,
\[ \P[\arm_{\ell_c + \delta}(R; \rho) ] \to 1 , \]
and by Proposition \ref{p:capdom} we have $\capa_K(B(R)) \sim 1/K(R)$ as $R \to \infty$. Hence taking $R \to \infty$ and then $\delta \to 0$ we have
\begin{equation*}
 - \log \P[ \arm_\ell(R) ] \le \frac{(\ell_c  - \ell)^2 (1 + o(1)) }{2 K(R)} . \qedhere
\end{equation*}
\end{proof}

\begin{proof}[Proof of Theorem \ref{t:diamgen}]
We begin with the upper bounds, which are immediate consequences of the upper bounds in Theorem \ref{t:gen}. 

\smallskip
For $R \ge 1$, consider covering $B(R)$ with a collection $(B_i)_i = (x_i + B(1))_i$ of at most $c_d R^d$ translated copies of $B(1)$, where $x_i \in B(R)$, and $c_d > 0$ is a constant depending only on the dimension. Suppose that $\{D_{R, \ell} \ge r + 2 \}$ occurs for some $r > 1$. Then by definition there is a point $y \in B(R)$ and a path $\gamma \subset \{f \le \ell\} \cap B(R)$ joining $y$ and $y + \partial B(r+2)$. Let $1 \le i \le c_d R^d$ be such that $y \in x_i + B(1)$. Then $\gamma$ also intersects both $x_i + B(1)$ and $x_i + \partial B(r) \subseteq y + B(r+2)$. Hence by the union bound and stationarity,
\begin{equation}
\label{e:diam1}
 \P[D_{R, \ell} \ge r + 2 ] \le  c_d R^d \P[ \arm_\ell(1, r) ] .
 \end{equation}

We now split the analysis into cases. In the case $\alpha \in (1, d)$, fix $c_1 > 0$ such that $\P[\arm_\ell(1, R)] \le e^{-c_1 R}$ for $R \ge 1$, which is possible by \eqref{e:gen1}, and set $r = c_2 \log R  $ for some constant $c_2 > d / c_1$. Then by \eqref{e:diam1} we have, as $R \to \infty$,
\[ \P[D_{R, \ell} \ge r + 2 ] \le  c_d R^d \P[ \arm_\ell(1, r) ] \le c_d R^d e^{-c_1 r } = c_d R^d e^{- c_1 c_2 \log R  } \to 0 , \]
from which we deduce the result. In the case $\alpha = 1$ we instead fix $\eps > 0$ and set
\[ r = \frac{4d}{(1-\gamma) (\ell^\ast_c - \ell)^2}( 1 + \eps ) (\log R) (\log \log R)^{1-\gamma}  . \] 
Then by \eqref{e:diam1} and \eqref{e:gen2} we have, as $R \to \infty$,
\begin{align*}
    \P[D_{R, \ell} \ge r  + 2] & \le  c_d R^d \P[ \arm_\ell(1, r) ]  \le c_d R^d \exp \Big(  \frac{ - (\ell^\ast_c - \ell)^2 ( 1 + o(1) )}{ (4/(1-\gamma)) r^{-1} (\log r)^{1-\gamma} } \Big) \\
    &  = c_d R^d e^{- d(1+\eps)(1+o(1)) ( \log R ) } \to 0 .
    \end{align*}
Since $\eps > 0$ was arbitrary, we deduce the result. In the case $\alpha \in (0, 1)$ we set instead
\[ r =  \Big( \frac{2d}{c_\alpha (\ell^\ast_c - \ell)^2} \Big)^{1/\alpha} ( 1 + \eps ) (\log R)^{1/\alpha} (\log \log R)^{-\gamma/\alpha}  , \] 
and by \eqref{e:diam1} and \eqref{e:gen3} we again have that $  \P[D_{R, \ell} \ge r + 2] \to 0$ as $R \to \infty$, as required.

We turn to the lower bounds, which we deduce from the lower bounds in Theorem \ref{t:gen} using the local-global decomposition $f = f_L + g_L$. More precisely, we will use that, for any $R, r, L \ge 1$ and $\eps > 0$
\begin{equation}
\label{e:diam2}
 \big\{f_L  \in \{ D_{R, \ell-\eps} \ge r - 1 \} \big\} \cap \big\{ \sup_{x \in B(R)}  g_L(x)    \le \eps \big\}   \Longrightarrow \big\{ f \in \{ D_{R,\ell} \ge r - 1 \} \big\} , 
 \end{equation}
 and
 \begin{equation}
\label{e:diam3}
 \big\{ f  \in \arm_{\ell - 2 \eps}(1,r) \big\}  \cap \big\{ \inf_{x \in B(r)}  g_L(x)    \ge - \eps \big\}   \Longrightarrow \big\{ f_L \in \arm_{\ell-\eps}(1,r)  \big\} , 
 \end{equation}
which follow since $\{f \in \{ D_{R, \ell} \ge r-1\} \}$ and $\{f \in \arm_{\ell}(R)  \}$ are increasing in $f$ and $\ell$.

For $R \ge 8$ and $1 \le r \le L \le R/8$, consider positioning a collection $(B_i)_i = (x_i + B(r))_i$ of at least $c_d(R/L)^d$  translated copies of $B(r)$ such that $B_i \subset B(R)$ for each $i$ and $\text{dist}(B_i, B_j) \ge L$ for each $i \neq j$, where $c_d > 0$ is a constant depending only on the dimension. The relevance of this construction is that the event $\{D_{R ,\ell} \ge r - 1\}$ is implied by the existence of a path in $\{f \le \ell\}$ from $x_i + B(1)$ to $x_i + B(r)$ for at least one $1 \le i \le c_d (R/L)^d$.  Then by \eqref{e:diam2} and the stationarity and $L$-range dependence of $f_L$, we have for any $\eps > 0$ 
\begin{align}
\nonumber
   & \P[ f \in \{D_{R, \ell} \ge r - 1\}  ] \ge  1 - \P[ f_L \notin \{ D_{R, \ell - \eps} \ge r-1 \} ]    - \P \big[   \sup_{x \in B(R)}  g_L(x)    \ge \eps  \big]   \\
  \nonumber  & \quad \ge 1  -  \big(  1 -  \P[ f_L \in \arm_{\ell - \eps}(1,r) ] \big)^{c_d (R/L)^d}     - \P \big[   \sup_{x \in B(R)}  g_L(x)    \ge \eps  \big]   \\
 \label{e:diam4}     & \quad \ge 1  - \exp   \big(  - c_d (R/L)^d  \P[ f \in \arm_{\ell -  2\eps}(1,r) ] - c_d'  (R/L)^d r^d P \big)  -   c'_d R^d P 
    \end{align}
    where $P =  \P \big[   \sup_{x \in B(1)} | g_L(x) |    \ge \eps  \big]   $ and $c'_d > 0$ is a constant, where in the last step we used the bound $(1-x)^t \le e^{-xt}$ for $x, t > 0$, \eqref{e:diam3}, and the union bound. We now set $L = (\log R)^{\max\{2,2/\alpha\}}$ so that, by Proposition \ref{p:ld2} we have, for every $\eps, k > 0$, as $R \to \infty$,
\[ R^k P = R^k \P \big[   \sup_{x \in B(1)} | g_L(x) |    \ge \eps  \big]  \le R^k \exp( -  c_3 (\log R)^{c_4} )  \to 0 \]
for some constants $c_3 > 0$ and $c_4 > 1$. In view of \eqref{e:diam4}, we conclude that as $R \to \infty$ eventually 
\begin{equation}
\label{e:diam5}
\P[ f \in \{D_{R, \ell} \ge r - 1\}  ]  \ge 1 - \exp \big( -  c_d (R/L)^d  \P[ f \in \arm_{\ell -  2\eps}(1,r) ] \big) (1 - o(1))  + o(1)  .
\end{equation}
where the $o(1)$ terms are independent of the choice of $r$.

We now split the analysis into cases. In the case $\alpha \in (1, d)$, we fix $\eps > 0$ and $c_5 > 0$ such that $\P[\arm_{\ell-2\eps}(1, R)] \ge e^{-c_5 R}$ for sufficiently large~$R$, which is possible by \eqref{e:gen1}, and set $r  = c_6 \log R$ for some constant $c_6 \in (0, d/c_5)$.  Then we have, as $R \to \infty$,
\[ (R/L)^d  \P[ f \in \arm_{\ell -  2\eps}(1,r) ]  \ge R^d (\log R)^{-2d}   e^{- c_5 c_6 \log R}  \to \infty . \]
 Hence by \eqref{e:diam5} we conclude that $\P[ f \in \{D_{R, \ell} \ge r - 1\}  ]  \to 1$ as $R \to \infty$, which gives the result. In the case $\alpha = 1$ we fix instead $\eps'  \in (0, 1)$, set $\eps > 0$ to be such that
 \begin{equation}
 \label{e:diam6}
  \frac{ (\ell^{\ast \ast}_c - \ell + 2\eps)^2 }{(\ell^{\ast \ast}_c - \ell)^2} ( 1 - \eps') < 1  , 
  \end{equation}
 and set 
 \[ r = \frac{4d}{(1-\gamma) (\ell^{\ast \ast}_c - \ell)^2}( 1 - \eps' ) (\log R) (\log \log R)^{1-\gamma}   .\]
 Then by \eqref{e:gen2} we have, as $R \to \infty$,
\begin{align*}
 (R/L)^d  \P[ f \in \arm_{\ell -  2\eps}(1,r) ]  & \ge  R^d (\log R)^{-2d}   \exp \Big(  \frac{ - (\ell^{\ast \ast}_c - \ell + 2\eps)^2 ( 1 + o(1) )}{ (4/(1-\gamma)) r^{-1} (\log r)^{1-\gamma} } \Big)   \\
 &  = R^d  (\log R)^{-2d}  e^{ - c_6 (\log R)} \to \infty 
 \end{align*}
 for some $c_6 \in (0, 1)$. Again by \eqref{e:diam5} we conclude that $\P[ f \in \{D_{R, \ell} \ge r - 1\}  ]  \to 1$ as $R \to \infty$, which gives the result. Finally in the case $\alpha \in (0, 1)$ we again fix $\eps'  \in (0, 1)$ and $\eps > 0$ satisfying \eqref{e:diam6}, and instead set 
\[ r =  \Big( \frac{2d}{c_\alpha (\ell^{\ast \ast}_c - \ell)^2} \Big)^{1/\alpha} ( 1 - \eps' ) (\log R)^{1/\alpha} (\log \log R)^{-\gamma/\alpha} . \] 
Using \eqref{e:gen3} we again have $ (R/L)^d  \P[ f \in \arm_{\ell -  2\eps}(1,r) ]   \to 1$ as $R \to \infty$, and so $\P[ f \in \{D_{R, \ell} \ge r - 1\}  ]  \to 1$ by \eqref{e:diam5}.
\end{proof}

\smallskip
\appendix

\section{White noise decompositions}
\label{a:wnd}

In this section we consider the white noise decompositions introduced in Section \ref{s:locglo}. We begin by stating conditions that guarantee smoothness and regularity. The second lemma is \cite[Lemma A.1]{m22}, and the proof of the first is analogous.

\begin{lemma}
\label{l:a1}
Let $q \in L^2(\R^d) \cap  C^3(\R^d)$ and let $W$ be the white noise on $\R^d$. Then $f = q \star W$ is an a.s.\ $C^2$-smooth stationary Gaussian fields, and for every multi-index $k$ with $|k| \le 2$,
\[ \partial^k f = \partial^k q \star W .\]
If moreover $q$ is non-zero, the spectral measure of $f$ has an open set in its support.
\end{lemma}

\begin{lemma}
\label{l:a2}
Let $q(x,t) \in L^2(\R^d \times \R^+)$ be such that, for every $t > 0$, $q(\cdot,t) \in C^3(\R^d)$, and for every multi-index $\nu$ with $|\nu| \le 3$, $\partial^{\nu,0} q \in L^2(\R^d \times \R^+)$ and $\partial^{\nu,0} q$ is locally bounded on $\R^d \times \R^+$. Let $W$ be the white noise on $\R^d \times \R^+$. Then $f = q \star W$ is an a.s.\ $C^2$-smooth stationary Gaussian fields, and for every multi-index $k$ with $|k| \le 2$,
\[ \partial^k f = \partial^{k,0} q \star_1 W .\]
If moreover there exists a $t \ge 0$ such at $q(\cdot,t)$ is non-zero and $q(\cdot,t') \to q(\cdot,t) \in L^1(\R^d)$ as $t' \to t$, then the spectral measure of $f$ has an open set in its support.
\end{lemma}

Next we consider multi-scale decompositions of isotropic fields whose covariance is given by $K(r) = \mathcal{L}[v](r^2)$, where $v$ is a non-negative function and $\mathcal{L}[\cdot]$ is the Laplace transform. In particular we justify the examples in Remark \ref{r:examples1}. First recall the classical Tauberian theorem for the Laplace transform:

\begin{proposition}[Tauberian theorem for the Laplace transform]
\label{p:taulaplace}
Let $\alpha \ge 0$ and let $L$ be a slowly varying function. Then
\begin{equation}
\label{e:tl1}
 \mathcal{L}[v](r) \sim  r^{-\alpha/2} L(r)  \quad \text{as } r \to \infty 
 \end{equation}
 if and only if
\begin{equation}
\label{e:tl2}
 \int_0^s v(s') ds' \sim \Gamma(1 + \alpha/2)^{-1} s^{\alpha/2}  L(1/s) \quad \text{as } s \to 0 .
 \end{equation}
Moreover if (i) $\alpha > 0$, or (ii) $\alpha = 0$ and $L(r) \sim (\log r)^{-\gamma}$, $\gamma > 0$, then \eqref{e:tl1}--\eqref{e:tl2} are implied by
\[ v(s) \sim  \begin{cases}
  \Gamma(\alpha/2)^{-1}  s^{\alpha/2 - 1} L(1/s)& \alpha > 0 , \\
 \gamma s^{-1} (\log (1/s))^{-\gamma - 1}  & \alpha = 0 .
 \end{cases}  \]
\end{proposition}
\begin{proof}
The first statement is \cite[Theorem 3, p.\ 445]{fel71}, and the second follows by combining with Propositions \ref{p:kit} and \ref{p:kitlog}.
\end{proof}

\begin{proposition}
\label{p:smrep}
Let $f$ be a continuous isotropic Gaussian field on $\R^d$ with covariance $K$, and suppose there exists a non-negative function $v \in C^0(\R^+)$, with $v(s) s^{d/2+3}$ bounded, such that $K(r) = \mathcal{L}[v](r^2)$. Define $ q(x,t) =  \sqrt{w(t)} Q(x/t)$, where
\begin{equation}
\label{e:qlaplace}
 w(t) = c_{d,\alpha} t^{-d-3}  v(1/(4t^2) ) \ , \quad Q(x) = e^{-|x|^2/2}  \ ,  \quad \text{and} \quad  c_{d,\alpha}=   2^{-1} \pi^{-d/2} ,
 \end{equation}
 and let $W$ be the white noise on $\R^d \times \R^+$. Then $q$ satisfies the conditions in Lemma \ref{l:a2}, and
\[ f \stackrel{d}{=}  q \star_1 W . \]
 Moreover if $v$ satisfies \eqref{a:v} then $K$, suitably renormalised, satisfies \eqref{a:k}.
\end{proposition}

\begin{proof}
For the first statement it suffices to check that 
\begin{align*}
(q \star q)(r) & = c_{d,\alpha} \int_0^\infty  t^{-d-3}  v(1/(4t^2))   (e^{-|\cdot|^2/(2t^2)} \star e^{-|\cdot|^2/(2t^2)} )(r) \, dt \\
& =  \int_0^\infty \frac{1}{2 t^3}  v(1/(4t^2)) e^{-r^2/(4t^2)} \,dt \\
& = \int_0^\infty e^{-sr^2} v(s) \, ds   = \mathcal{L}[v](r^2) = K(r) ,
\end{align*}
where the second equality used that 
\[ (e^{-|\cdot|^2/(2t^2)} \star e^{-|\cdot|^2/(2t^2)} )(x) =  \pi^{d/2} t^d  e^{-x^2/(4t^2)} , \]
and the third equality was the change of variable $t \mapsto 1/(2\sqrt{s})$. The second statement is a consequence of Proposition \ref{p:taulaplace}.
\end{proof}

\smallskip
\section{Tauberian theorems for the isotropic Fourier transform}
\label{a:tau}

Consider a continuous isotropic Gaussian field $f$ on $\R^d$ with covariance~$K$ and whose spectral measure has a density $\rho$. Recall the function $q = \mathcal{F}[\sqrt{\rho}]$ and assume that it is continuous. The aim of this section is to prove the following:

\begin{proposition}
\label{p:tau}
Let $d\ge2$, $\alpha \in [0, d)$, $\gamma > 0$, and $L$ be a slowly varying function. Consider the following statements:
\begin{enumerate}
\item As $x \to \infty$,
  \begin{equation}
  \label{e:tauk}
  K(x) \sim  \begin{cases}  x^{-\alpha}L(x)  & \text{if } \alpha \in (0, d) , \\  (\log x)^{-\gamma} &  \text{if } \alpha = 0 . \end{cases}   
  \end{equation}
\item As $x \to \infty$,
\begin{equation}
\label{e:tauq}
  q(x) \sim  \begin{cases} c_{d,\alpha}'  x^{-(d+\alpha)/2} \sqrt{L(x)}  & \text{if } \alpha \in (0, d) , \\ c_{d,\alpha}' x^{-d/2}(\log x)^{-(\gamma+1)/2} &  \text{if } \alpha = 0 ,  \end{cases} 
  \end{equation}
  where
    \begin{equation}
  \label{e:c'}
   c_{d,\alpha}' =  \begin{cases}  \frac{ \pi^{d/4}  2^{-\alpha} \sqrt{ \Gamma( (d-\alpha)/2 )} \Gamma((d-\alpha)/4)     }{ \sqrt{\Gamma(\alpha/2)} \Gamma( (d+\alpha)/4) }   & \text{if } \alpha \in (0,d) , \\ 
\frac{ \sqrt{\gamma}}{\sqrt{2\pi}}    & \text{if }  \alpha = 0 ,
  \end{cases}  
  \end{equation}
    and $\Gamma$ denotes the Gamma function.
\item As $\lambda \to 0$, 
\begin{equation}
\label{e:taurho}
    \rho(\lambda) \sim \begin{cases} 
 c_{d,\alpha}''  \lambda^{\alpha - d } L(1/\lambda)  & \text{if } \alpha \in (0,d) , \\ 
   c_{d,\alpha}'' \lambda^{- d } (\log (1/\lambda) )^{-(\gamma+1)}  & \text{if } \alpha = 0 ,
  \end{cases}  
  \end{equation}
  where
   \begin{equation}
\label{e:c''}
 c_{d,\alpha}'' =  \begin{cases} 
  \frac{ \pi^{d/2} 2^{d-\alpha} \Gamma((d - \alpha)/2)}{\Gamma(\alpha/2)} & \text{if } \alpha \in (0,d) , \\ 
 2 \pi \gamma    & \text{if }  \alpha = 0 .
  \end{cases} 
  \end{equation}
   
\end{enumerate}
Then:
\begin{itemize}
\item If $\alpha > (d-3)/2$ and if $K(x) x^{(d-3)/2}$ and $q(x) x^{(d-3)/2}$ are eventually decreasing, then $(3)$ implies $(1)$ and $(2)$. 
\item If $\alpha > (d-3)/2$, and if $K(x) x^{(d-3)/2}$ and $q(x) x^{(d-1)/2}$ are eventually decreasing, then $(2)$ and $(3)$ are equivalent, and both imply $(1)$.
\item If $\alpha > (d-1)/2$, and if $K(x) x^{(d-1)/2}$ and $q(x) x^{(d-1)/2}$ are eventually decreasing, then $(1)$--$(3)$ are equivalent.
\end{itemize}
\end{proposition}

As a corollary we justify the validity of the examples in Remark \ref{r:examples2}:
\begin{corollary}
Suppose $d \in \{2,3\}$ and $q: \R^d \to \R$ is a continuous isotropic unimodal function satisfying \eqref{e:tauq} for $\alpha \in [0, d) \cap ((d-3)/2, d)$, $\gamma > 0$, and a slowly varying function $L$. Suppose in addition that $q(x) x^{(d-1)/2}$ is eventually decreasing. Then the continuous isotropic Gaussian field with covariance $K(x) = (q \star q)(x)$ satisfies each of the statements $(1)$--$(3)$ in Proposition \ref{p:tau}.  
\end{corollary}
\begin{proof}
The field $f$ has spectral density $\rho = (\mathcal{F}[q])^2$ (recall that $q$ is in $L^2(\R^d)$). By assumption the statement $(2)$ in Proposition \ref{p:tau} is satisfied. Moreover, since $K = q \star q$ is isotropic and unimodal (as a convolution of two isotropic unimodal functions), $K$ is eventually decreasing, and so Proposition \ref{p:tau} implies that the statements $(1)$ and $(3)$ are also satisfied.
\end{proof}

\begin{remark}
\label{r:agen}
We use the condition $d \in \{2,3\}$ only to guarantee that if $K$ is eventually decreasing then so is $K(x) x^{(d-3)/2}$. We believe that conclusion of this corollary is true without the restrictions $d \in \{2,3\}$ and $\alpha > (d-3)/2$, possibly under mild extra conditions on $q$, however we do not pursue this here.
\end{remark}

To prove Proposition \ref{p:tau} we rely on (i) a classical Tauberian theorem relating the decay of a function at infinity with the behaviour of its Hankel transform at the origin, and (ii) Hankel-type formulae for the isotropic Fourier transform.

\begin{theorem}[Tauberian theorem for the Hankel transform; see Theorem 1.2 in \cite{bi97}]
\label{t:han}
Let $\nu > -1/2$, and suppose $f : \R^+ \to \R$ is a function such that $r^{1+\nu} f(r) \in L^1_{\text{loc}}[0, \infty)$ and $r^{1/2}f(r)$ is eventually decreasing. Then the order-$\nu$ Hankel transform
\[ F(s) = \int_0^{\infty_-} f(r) J_\nu(sr) \, r \, dr , \ s > 0 \]
exists and is continuous, where $\int_0^{\infty-} \cdots$ denotes the improper integral $\lim_{M \to \infty} \int_0^M \cdots$. Moreover, for a slowly varying function $L$, and for every $1/2 < \beta < 2 + \nu$, the following are equivalent:
\begin{enumerate}
\item $f(r) \sim r^{-\beta} L(r)$ as $r \to \infty$;
\item $F(s)\sim \tilde{c}(\beta, \nu) s^{\beta-2} L(1/s)$ as $s \to 0$, where
\[ \tilde{c}(\beta, \nu) = 2^{1-\beta} \frac{\Gamma(1-\beta/2 + \nu/2)}{\Gamma(\beta/2 + \nu/2)} > 0 .\]
\end{enumerate}
\end{theorem}

\begin{proposition}[Hankel-type formulae for the isotropic Fourier transform]
\label{p:form}
Let $g$ be a continuous isotropic positive definite function on $\R^d$ and let $\mu$ be its spectral measure. Define the bounded non-decreasing function
\[ G(\lambda) = \int_{|s| < \lambda} d\mu(s) , \quad \lambda > 0 . \]
Then if $x^{(d-3)/2} g(x)$ is eventually decreasing, $G$ is continuous and has the representation 
\[   G(\lambda) = \frac{(2 \pi)^d }{\Gamma(d/2) 2^{(d-2)/2}}  \lambda^{d/2} \int_0^{\infty_-} (g(x) x^{(d-4)/2}) J_{d/2}(\lambda x) \, x \, dx . \]
Moreover, if $d \ge 2$ and $x^{(d-1)/2} g(x)$ is eventually decreasing, the spectral measure has a density $\rho$ which is continuous on $\R^d \setminus \{0\}$ and has the representation
\[    \rho(\lambda) =  (2 \pi)^{d/2}  \lambda^{(2-d)/2} \int_0^{\infty_-} (g(x) x^{(d-2)/2}) J_{(d-2)/2}(\lambda x) \, x \, dx .  \]
\end{proposition}
\begin{proof}
The Hankel transform formulae for $G$ and $\rho$ are standard (see \cite{lo13}, and note the convention used to define the Fourier transform differs from ours), and the fact that they are well-defined and continuous follows from Theorem \ref{t:han} (applied in the first case to $x \mapsto g(x) x^{-(d-4)/2}$ and $\nu = d/2$, and in the second case to $x \mapsto g(x) x^{(d-2)/2}$ and $\nu = (d-2)/2$).
\end{proof}

\begin{proof}[Proof of Proposition \ref{p:tau}]
We prove the three conclusions of the proposition in order:

\textit{First conclusion.} Suppose that $\alpha > (d-3)/2$, $K(x) x^{(d-3)/2}$ and $q(x) x^{(d-3)/2}$ are eventually decreasing, and assume that \eqref{e:taurho} holds; we will show that this implies \eqref{e:tauk} and \eqref{e:tauq}. Define the functions
\[ G_K(\lambda) = \int_{|s|<\lambda} \rho(|s|) ds    \quad \text{and} \quad G_q(\lambda) = \int_{|s|<\lambda} \sqrt{\rho}(|s|) ds   .\]
Then in light of \eqref{e:taurho} (and recalling that the volume of the $d$-sphere is $2 \pi^{d/2} / \Gamma(d/2)$) we have, as $\lambda \to 0$, 
\[    G_K(\lambda)  \sim \begin{cases}    \frac{2 \pi^{d/2} c'' }{\Gamma(d/2)}   \int_0^\lambda s^{\alpha - 1} L(1/s) ds  \sim    \frac{2 \pi^{d/2} c'' }{\Gamma(d/2) \alpha}  \lambda^\alpha   L(1/\lambda) & \alpha \in (0, d) , \\
  2 \pi c''   \int_0^\lambda s^{-1} (\log (1/s))^{-(\gamma+1)} ds \sim    \frac{2 \pi c'' }{ \gamma} (\log (1/\lambda))^{-\gamma}  &   d=2, \alpha = 0 ,
\end{cases} \]
where we used, after a change of variables, Propositions \ref{p:kit} (with $\beta = - \alpha - 1 < -1$) and \ref{p:kitlog} (with $\beta = -\gamma - 1 < - 1$). Similarly, as $\lambda \to 0$, 
\[   G_q(\lambda)  \sim \begin{cases}     \frac{2 \pi^{d/2} \sqrt{c''} }{\Gamma(d/2)}  \int_0^\lambda s^{(\alpha+d-2)/2} \sqrt{L(1/s)} ds  \sim  \frac{2 \pi^{d/2} \sqrt{c''} }{\Gamma(d/2) ((\alpha+d)/2)} \lambda^{(\alpha+d)/2}  \sqrt{L(1/\lambda)} & \alpha \in (0, d) , \\
 2 \pi \sqrt{c''}  \int_0^\lambda  (\log (1/s))^{-(\gamma+1)/2} ds \sim    2 \pi \sqrt{c''} \lambda (\log (1/\lambda))^{-(\gamma+1)/2}  &   d=2, \alpha = 0 ,
\end{cases} \]
where we used again Proposition \ref{p:kit} (with $\beta = -2-((\alpha+d-2)/2) < -1$).

Since $x^{(d-3)/2} K(x)$ and $x^{(d-3)/2} q(x)$ are continuous and eventually decreasing by assumption, by Proposition \ref{p:form} we have the representations
\[    G_K(\lambda) = \frac{(2 \pi)^d }{\Gamma(d/2) 2^{(d-2)/2}}  \lambda^{d/2} \int_0^{\infty_-} (K(x) x^{(d-4)/2}) J_{d/2}(\lambda x) \, x \, dx \] 
and
\[   G_q(\lambda) = \frac{(2 \pi)^d }{\Gamma(d/2) 2^{(d-2)/2}}  \lambda^{d/2} \int_0^{\infty_-} (q(x) x^{(d-4)/2}) J_{d/2}(\lambda x) \, x \, dx . \]
Then by Theorem \ref{t:han} (applied to the function $x \mapsto K(x) x^{(d-4)/2}$, $\nu = d/2$, and $\beta = \alpha - (d-4)/2$, and notice that $\beta \in (1/2, 2+\nu)$ since $\alpha \in ((d-3)/2, d)$) we have, as $x \to \infty$, 
\[    K(x)  \sim \begin{cases}     \frac{c''}{(2 \pi)^{d/2} \alpha \tilde{c}(\alpha - (d-4)/2,d/2)} x^{-\alpha} L(x)  & \alpha \in (0, d) , \\
  \frac{c''}{2 \pi \gamma  }  (\log x)^{-\gamma} & d=2, \alpha = 0 ,
\end{cases} \]
 and also by Theorem \ref{t:han} (applied to the function $x \mapsto q(x) x^{(d-4)/2}$, $\nu = d/2$, and $\beta = 2 + \alpha / 2$, and notice that $\beta \in (1/2, 2+\nu)$ since $\alpha \in (-3, d)$), as $x \to \infty$,
\[    q(x)  \sim \begin{cases}     \frac{ \sqrt{ c''} }{(2 \pi)^{d/2} ((\alpha+d)/2) \tilde{c}(2+\alpha/2,d/2) } x^{-(\alpha+d)}/2  L(x)  & \alpha \in (0, d) , \\
 \frac{ \sqrt{ c'' } }{2 \pi }  x^{-1} (\log x)^{-(\gamma+1)/2} &  d=2,\alpha = 0 ,
\end{cases} \]
where in the case $d=2$ and $\alpha = 0$ we used that $\tilde{c}(1,1) = \tilde{c}(2,1) = 1$ by standard properties of the Gamma function. Recalling the identity $z \Gamma(z) = \Gamma(z+1)$ one can check that
\[    \tilde{c}(\beta+1,\nu+1) (\beta + \nu) = \tilde{c}(\beta,\nu) \  , \quad \beta,\nu > 1/2 .   \]
Since we have defined $c'$ and $c''$ to satisfy 
\begin{equation}
\label{e:c''2}
c'' =  \begin{cases} 
  (2 \pi)^{d/2} \alpha \tilde{c}(\alpha - (d-4)/2,d/2)  =   \frac{ \pi^{d/2} 2^{d-\alpha} \Gamma((d - \alpha)/2)}{\Gamma(\alpha/2)} & \text{if } \alpha \in (0,d) , \\ 
 2 \pi \gamma    & \text{if } d=2, \alpha = 0 ,
  \end{cases} 
  \end{equation}
and
  \begin{equation}
  \label{e:c'2}
   c' =  \begin{cases}  \frac{ \sqrt{ c''}}{(2 \pi)^{d/2} ((\alpha+d)/2) \tilde{c}(2+\alpha/2,d/2) }  =  \frac{ \sqrt{ c''}}{(2 \pi)^{d/2} \tilde{c}(1+\alpha/2,d/2-1) }     & \text{if } \alpha \in (0,d) , \\ 
\frac{ \sqrt{ c''} }{2 \pi}    & \text{if }  d=2,\alpha = 0 ,
  \end{cases}  
  \end{equation}
 the first conclusion is established.
 
\textit{Second conclusion.} Suppose that $d \ge 2$, $\alpha > (d-3)/2$, $K(x) x^{(d-3)/2}$ and $q(x) x^{(d-1)/2}$ are eventually decreasing, and \eqref{e:tauq} holds; by the first conclusion it is then sufficient deduce \eqref{e:taurho}. By Proposition \ref{p:form}, $\sqrt{\rho}$ has the representation 
\[   \sqrt{\rho}(\lambda) =  (2 \pi)^{d/2}  \lambda^{(2-d)/2} \int_0^{\infty_-} (q(x) x^{(d-2)/2}) J_{(d-2)/2}(\lambda x) \, x \, dx .  \]
 Then by Theorem \ref{t:han} (applied to the function $x \mapsto q(x) x^{(d-2)/2}$, $\nu = (d-2)/2$, and $\beta = (\alpha +2)/2$, and notice that $\beta \in (1/2, 2+\nu)$ since $\alpha \in (-2, d)$) we have, as $\lambda \to 0$, 
 \[      \sqrt{\rho}(\lambda) \sim  \begin{cases} c'  (2 \pi)^{d/2} \tilde{c}(1+\alpha/2, d/2-1) \lambda^{(\alpha - d)/2 } \sqrt{L(1/\lambda)}  & \text{if } \alpha \in (0,d) , \\  c'  2 \pi \lambda^{- 1 } (\log (1/\lambda))^{-(\gamma+1)/2} & \text{if } d = 2, \alpha = 0 .
  \end{cases}   \]
By \eqref{e:c'2}, after squaring this is equivalent to \eqref{e:taurho}.

\textit{Third conclusion.}  Suppose that $d \ge 2$, $\alpha > (d-1)/2 > 0$, $K(x) x^{(d-1)/2}$ and $q(x) x^{(d-1)/2}$ are eventually decreasing, and \eqref{e:tauk} holds; by the second conclusion it is then sufficient deduce \eqref{e:taurho}. By Proposition \ref{p:form}, $\rho$ has the representation 
\[    \rho(\lambda) =  (2 \pi)^{d/2}  \lambda^{(2-d)/2} \int_0^{\infty_-} (K(x) x^{(d-2)/2}) J_{(d-2)/2}(\lambda x) \, x \, dx .  \]
 Then by Theorem \ref{t:han} (applied to the function $x \mapsto K(x) x^{(d-2)/2}$, $\nu = (d-2)/2$, and $\beta = \alpha - (d-2)/2$, and notice that $\beta \in (1/2, 2+\nu)$ since $\alpha \in ((d-1)/2, d)$) we have, as $\lambda \to 0$, 
 \[      \rho(\lambda) \sim     (2 \pi)^{d/2} \tilde{c}(\alpha  - (d-2)/2, d/2 -1 ) \lambda^{\alpha - d } L(1/\lambda) \sim   \frac{2^d \pi^{d/2} 2^{-\alpha} \Gamma((d - \alpha)/2)}{\Gamma(\alpha/2)}   \lambda^{\alpha - d } L(1/\lambda)    ,  \]
 which gives \eqref{e:taurho} by \eqref{e:c''2}.
 \end{proof}

\smallskip
\section{Basic properties of regularly varying functions}
\label{a:reg}

We collect basic properties of regularly varying functions; see \cite{bgt87} for a detailed treatment. Recall that a slowly varying function is a continuous strictly positive function on $\R^+$ satisfying \eqref{e:sv}. Our first set of results concern regularly varying functions on $\R^+$, i.e.\ continuous functions that satisfy $h(x) \sim x^{-\alpha} L(x)$ as $x \to \infty$ for an index $\alpha \in \R$ and slowly varying~$L$.

\smallskip We begin with the fact that regularly varying functions are asymptotically monotone:

\begin{lemma}
\label{l:evd}
Let $h$ be a regularly varying function on $\R^+$ with index $\alpha > 0$. Then, as $R \to \infty$,
\[ \sup_{x \ge R} h_1(x)  \sim h_1(R)  .  \]
Moreover, if $h_1$ and $h_2$ are regularly varying functions with respective indices $\alpha_1 > \alpha_2 \ge 0$, then for every $c \in \R$ there exists $m_0 > -c$ such that, for every $m \ge m_0$,
\[ \sup_{x \ge mR} \frac{h_1(x + cR)}{h_2(x)} \sim \frac{h_1((m+c)R)}{h_2(mR)}  .\]
\end{lemma}
\begin{proof}
The first statement is \cite[Theorem 1.5.3]{bgt87}. The second is proven similarly; more precisely, an identical proof shows that, for any $m > -c$,
\begin{equation}
\label{e:evd1}
 \sup_{ x \ge m R} \frac{h_1(x +cR)}{ h_1((m +c)R)}  \times \frac{h_2(m R)}{h_2(x)}  \to   \sup_{ \rho \in [m,\infty) } \frac{ (\rho+c)^{-\alpha_1} }{ (m+c)^{-\alpha_2}} \times \frac{ m^{-\alpha_2} } { \rho^{-\alpha_2} } . 
 \end{equation}
Since $\alpha_1 > \alpha_2 \ge 0$, one can fix $m_0 > -c$ so that $ (\rho+c)^{-\alpha_1}  \rho^{\alpha_2} $ is strictly decreasing on $\rho \ge m_0$. Then for any $m \ge m_0$, the right-hand side of \eqref{e:evd1} evaluates to $1$, as required.
\end{proof}

We next recall that slowly varying functions can be treated as constants when integrating:

\begin{proposition}[Karamata integral theorem, see {\cite[Theorem 1.5.11]{bgt87}}]
\label{p:kit}
Let $L$ be a continuous slowly varying function on $\R^+$. If $\beta > -1$, then as $R \to \infty$,
\[ \int_1^R x^{\beta} L(x) dx \sim \frac{1}{\beta+1} R^{\beta+1} L(R)  , \]
and if $\beta < -1$, then as $R \to \infty$,
\[ \int_R^\infty x^{\beta} L(x) dx \sim \frac{-1}{\beta+1} R^{\beta+1} L(R)  .\]
\end{proposition}

In the borderline case $\beta = -1$ one obtains similar behaviour by restricting to the subclass of slowly varying functions that are powers of the logarithm:

\begin{proposition}
\label{p:kitlog}
If $\beta > -1$, then as $R \to \infty$,
\[ \int_1^R x^{-1} (\log x)^\beta  dx \sim  \frac{1}{\beta + 1} (\log R)^{\beta + 1}  , \]
and if $\beta < -1$, then as $R \to \infty$, 
\[ \int_R^\infty x^{-1} (\log x)^\beta  dx \sim   \frac{-1}{\beta + 1} (\log R)^{\beta + 1}   . \]
\end{proposition}

We also note some simple variants of the integrals in the previous two propositions:
 
 \begin{lemma}
 \label{l:regvar1}
 Let $h$ be a regularly varying function on $\R^+$ with index $\alpha \in [0, 1)$. Then as $R \to \infty$,
 \[    \frac{1}{R^2} \int_0^R \int_0^R h(|x-y|) dx dy \sim \frac{2 h(R)}{(2-\alpha)(1-\alpha)} .   \]
 \end{lemma}
 \begin{proof}
 Let $x_0 > 0$ be such that $h(x) \ge 0$ for $x \ge x_0$. Then by Fatou's lemma and direct calculation one has
  \begin{align*}
      \liminf_{R \to \infty}  \int_0^1 \int_0^1 \frac{h(R|x-y|) }{h(R)} dx dy  & \ge  \liminf_{R \to \infty}  \int_0^1 \int_0^1 \id_{|x-y| \ge x_0/R} \frac{h(R|x-y|) }{h(R)} dx dy  -  \frac{x_0  \|h\|_{\infty}   }{R h(R) } \\
      &  \ge \int_0^1 \int_0^1 |x-y|^{-\alpha} dx dy   = \frac{2}{(2-\alpha)(1-\alpha)} . 
      \end{align*}
A matching upper bound can be obtained from Potter's bounds \cite[Theorem 1.5.6]{bgt87} and a dominated convergence argument (see the proof of Proposition \ref{p:capdom} for a similar argument), which completes the proof by a change of variables. \end{proof}
 
 \begin{lemma}
 \label{l:regvar2}
 Let $h$ satisfy $h(x) \sim x^{-1} (\log x)^{-\gamma}$ for $\gamma < 1$.  Then
 \begin{equation}
 \label{e:regvar1}
  \frac{1}{2R} \int_0^R \int_0^R h(|x-y|) dx dy  \sim   \int_0^{R} h(x) dx  \sim \frac{1}{1-\gamma}   R h(R) (\log R).   
 \end{equation}
 Moreover, if $r = r(R) \ge 0$ satisfies $r = R^{o(1)}$ as $R \to \infty$, then
\begin{equation}
\label{e:regvar2}
  \frac{1}{2R} \int_0^R \int_0^R h(|x-y|) \id_{|x-y| \ge r} dx dy  \sim  \min_{z \in [0, r] } \int_0^R h( \sqrt{x^2 + z^2} ) dx  \sim   \int_0^R h(x) dx   . 
 \end{equation}
\end{lemma}
\begin{proof}
By Proposition \ref{p:kitlog} we have that, as $R \to \infty$,
\[   \int_0^R h(x) dx \sim \frac{1}{1-\gamma} (\log R)^{-\gamma + 1} \sim  \frac{1}{1-\gamma} R h(R) (\log R)  . \]
We next establish that 
\[   \frac{1}{2R} \int_0^R \int_0^R h(|x-y|) \id_{|x-y| \ge r} dx dy  \sim \frac{1}{1-\gamma} (\log R)^{-\gamma + 1}  .\] 
Since $h$ is bounded we may assume that $r \to \infty$. Then by Propositions \ref{p:kitlog} and \ref{p:kit}
\begin{align*}
& \frac{1}{2R}  \int_0^R  \int_0^R   h(|x-y|) \id_{|x-y| \ge r} dx  dy \\
& \sim  \frac{1}{R}  \int_0^{R-r}  \int_r^{R-y}  \frac{(\log x)^{-\gamma}}{x}  dx  dy \\
& \sim  \frac{1}{R (1-\gamma) }  \int_1^{R}  (\log y )^{-\gamma+1}   dy   -   O((\log r)^{-\gamma+1})   \sim \frac{1}{1-\gamma} (\log R)^{-\gamma + 1} ,
\end{align*}
where in the last step we used that $r = R^{o(1)}$. Finally, since $h(x)$ is eventually positive and $h(x) \sim \sup_{t \ge x} h(x)$ (Lemma \ref{l:evd}), as $R \to \infty$,
\begin{align*}
    \int_0^R h(x) dx   (1+o(1))  & \ge  \min_{z \in [0, r] } \int_0^R h( \sqrt{x^2 + z^2} ) dx   \\
    &  =  \min_{z \in [0, r] }    \int_{r^2}^R h( x \sqrt{ 1 + z^2 / x^2 } ) dx  + \int_0^{r^2} h( x \sqrt{ 1 + z^2 / x^2 } ) dx  \\
    & \ge \int_{2r^2}^{R/2} h(x) dx (1 + o(1) ) - O(1)  \\
    & \sim \Big( \frac{1}{1-\gamma} (\log (R/2))^{-\gamma + 1} - \frac{1}{1-\gamma} (\log (2 r^2))^{-\gamma + 1} \Big) \sim \frac{1}{1-\gamma} (\log R)^{-\gamma + 1},
    \end{align*}
    as required, where we used $r = R^{o(1)}$ in the final step.
\end{proof}

Finally we establish a decomposition of isotropic regularly varying kernels on $\R^d$ that was used in the proof of Proposition \ref{p:capdom}:

\begin{proposition}
\label{p:rvdecomp}
Let $h$ be a function on $\R^d$ that is isotropic, positive definite, and regularly varying with index $\alpha \in [0, d)$. Then there exist continuous, isotropic, and positive definite functions $h_1$ and $h_2$ such that $h = h_1 + h_2$,  $h_1 \ge 0$, and $h_1 (x) \sim h(x)$ as $x \to \infty$.
\end{proposition}
\begin{proof}
For $r > 0$, define the functions $\varphi_r(x) = \id_{B(r)}(x) / \textrm{Vol}(B(r))$ and $\bar{h}_r = h \star \varphi_r$. We first show that there exists an $r_0 > 0$ such that $\bar{h}_{r_0}$ is non-negative. 

\smallskip
Since $h$ is eventually positive, by linear rescaling we may assume that $h(x) \ge 0$ for $|x| \ge 1$. Observe that, for sufficiently large $r$ and all $x = (t, t, \ldots , t) \in \R^d$, $t \in [0, (1 + r)/\sqrt{d}]$, we have
\[ x + B(r) \supset  [0, r / c_d]^d \setminus B(c_d) , \]
for a constant $c_d > 0$ depending only on the dimension. We deduce that, for sufficiently large~$r$,
\[  \inf_{x \in B(1 + r)}   \bar{h}_r(x)  \ge \frac{  \int_{y \in [0, r/c_d]^d \setminus B(c_d) }   h(y) \, dy  - \int_{B(1)} h(x) dx} {\textrm{Vol}(B(r) )}  . \]
On the other hand, by a change of variables and Fatou's lemma,
\begin{align*}
 \liminf_{r \to \infty}  (r^d  h(r))^{-1}   \int_{ [0, r/c_d]^d \setminus B(c_d) }   h(y) / h(r)  \, dy & =   \liminf_{r \to \infty}   \int_{ [0, 1/c_d]^d \setminus B(c_d/r) }   h(yr) / h(r)  \, dy  \\
 & =  \int_{y \in [0, 1/c_d]^d } |y|^{-\alpha} \, dy     > 0 .
 \end{align*}
Since $r^d  h(r) \to \infty$ as $r \to \infty$ (recall $\alpha < d$), this shows that $\int_{[0, r/c_d]^d \setminus B(c_d) }   h(y) \, dy \to \infty$, and we conclude that 
   \[ \inf_{x \in B(1 + r)}   \bar{h}_r(x) > 0 \]
for sufficiently large $r$. Since also $\bar{h}_r(x) \ge 0$ for every $r > 0$ and $x \in B(1 + r)^c$, we have established the existence of an $r_0$ such that  $\bar{h}_{r_0} \ge 0$.

\smallskip
Next define the continuous isotropic functions $h_1 =  h \star \varphi_{r_0} \star \varphi_{r_0} = \bar{h}_{r_0} \star \varphi_{r_0}$ and $h_2 = h - h_1$. By construction we have $h_1 \ge 0$. Moreover, since $\| \varphi_{r_0} \star \varphi_{r_0}\|_{L^1(R^d)} = \| \varphi_{r_0} \|_{L^1(R^d)}  =1$,
\begin{equation}
\label{e:rvdecomp2}
\inf_{ y \in x + B(2r_0) } h(y) / h(x)    \le  h_1(x) / h(x)  \le  \sup_{ y \in x + B(2r_0) } h(y) /h(x) ,
\end{equation}
and by the uniform convergence property \eqref{e:rv2}, both sides of \eqref{e:rvdecomp2} tend to $1$ as $x \to \infty$. 

\smallskip
It remains to verify that $h_1$ and $h_2$ are positive definite, which we do by exhibiting them as Fourier transforms of positive functions. Define the function $v = \mathcal{F}[ \varphi_{r_0} \star \varphi_{r_0}]  = \mathcal{F}[ \varphi_{r_0} ]^2$ and the finite measure $\mu = \mathcal{F}[h]$. Since $v$ is a non-negative, positive definite (as the Fourier transform of a positive function), and $v(0) = \| \varphi_{r_0} \star \varphi_{r_0}\|_{L^1(R^d)} = 1$, $v$ takes values in $[0,1]$. In particular $v \mu$ and $(1-v)\mu$ are finite measures. Since $h_1 = \mathcal{F}[v \mu]$ and $h_2 = \mathcal{F}[\mu] - \mathcal{F}[v \mu] = \mathcal{F}[ (1-v) \mu]$, the proof is complete.
 \end{proof}

\bigskip

\bibliographystyle{halpha-abbrv}
\bibliography{subcrit}

\begin{thebibliography}{DRRV23}
\expandafter\ifx\csname url\endcsname\relax
  \def\url#1{\texttt{#1}}\fi
\expandafter\ifx\csname doi\endcsname\relax
  \def\doi#1{\burlalt{doi:#1}{http://dx.doi.org/#1}}\fi
\expandafter\ifx\csname urlprefix\endcsname\relax\def\urlprefix{URL }\fi
\expandafter\ifx\csname href\endcsname\relax
  \def\href#1#2{#2}\fi
\expandafter\ifx\csname burlalt\endcsname\relax
  \def\burlalt#1#2{\href{#2}{#1}}\fi

\bibitem[AB87]{ab87}
M.~Aizenman and D.~Barsky.
\newblock Sharpness of the phase transition in percolation models.
\newblock {\em Comm. Math. Phys.}, 108(3):489--526, 1987.

\bibitem[AMS14]{ams14}
R.~Adler, E.~Moldavskaya, and G.~Samorodnitsky.
\newblock On the existence of paths between points in high level excursion sets
  of {G}aussian random fields.
\newblock {\em Ann. Probab.}, 43(3):1020--1053, 2014.

\bibitem[AW09]{aw09}
J.-M. Aza{\"{\i}}s and M.~Wschebor.
\newblock {\em Level sets and extrema of random processes and fields}.
\newblock John Wiley \& Sons, Inc., Hoboken, NJ, 2009.

\bibitem[BG17]{bg17}
V.~Beffara and D.~Gayet.
\newblock Percolation of random nodal lines.
\newblock {\em Publ. Math. IHES}, 126:131--176, 2017.

\bibitem[BGT87]{bgt87}
N.~H. Bingham, C.~M. Goldie, and J.~L. Teugels.
\newblock {\em Regular Variation}.
\newblock Cambridge University Press, 1987.

\bibitem[BI97]{bi97}
N.~Bingham and A.~Inoue.
\newblock An {A}bel--{T}auber theorem for {H}ankel transforms.
\newblock In N.~Kono and N.~Shieh, editors, {\em Trends in probability and
  related analysis}, pages 83--90. World Scientific, River Edge, 1997.

\bibitem[BS07]{bs07}
E.~Bogomolny and C.~Schmit.
\newblock Random wavefunctions and percolation.
\newblock {\em Journal of Physics A: Mathematical and Theoretical},
  40(47):14033--14044, 2007.

\bibitem[CN20]{CN20a}
A.~Chiarini and M.~Nitzschner.
\newblock Entropic repulsion for the {G}aussian free field conditioned on
  disconnection by level-sets.
\newblock {\em Probability Theory and Related Fields}, 177(1):525--575, 2020.

\bibitem[CS18]{cs18}
A.~Chakrabarty and G.~Samorodnitsky.
\newblock Asymptotic behaviour of high {G}aussian minima.
\newblock {\em Stoch. Proc. Appl.}, 128(7):2297--2324, 2018.

\bibitem[DGRS23]{dgrs20}
H.~Duminil-Copin, S.~Goswami, P.-F. Rodriguez, and F.~Severo.
\newblock Equality of critical parameter for percolation of {G}aussian free
  field level-sets.
\newblock {\em Duke Math. J.}, 172(5):839--913, 2023.

\bibitem[DPR23]{DPR21}
A.~Drewitz, A.~Pr\'{e}vost, and P.-F. Rodriguez.
\newblock Critical exponents for a percolation model on transient graphs.
\newblock {\em Invent. Math.}, 232:229--299, 2023.

\bibitem[DRRV23]{drrv21}
H.~Duminil-Copin, A.~Rivera, P.-F. Rodriguez, and H.~Vanneuville.
\newblock Existence of unbounded nodal hypersurface for smooth {G}aussian
  fields in dimension $d \ge 3$.
\newblock {\em Ann. Probab.}, 51(1):228--276, 2023.

\bibitem[DRT20]{dcrt20}
H.~Duminil-Copin, A.~Raoufi, and V.~Tassion.
\newblock Subcritical phase of $d$-dimensional {P}oisson-{B}oolean percolation
  and its vacant set.
\newblock {\em Ann. H. Lebesgue}, 3:677--700, 2020.

\bibitem[Fel71]{fel71}
W.~Feller.
\newblock {\em An Introduction to Probability Theory and its Applications, Vol
  {II}}.
\newblock J. Wiley $\&$ Sons, 1971.

\bibitem[Fug60]{fug60}
B.~Fuglede.
\newblock On the theory of potentials in locally compact spaces.
\newblock {\em Acta Math.}, 103:139--215, 1960.

\bibitem[GM90]{gm90}
G.~Grimmett and J.~Marstrand.
\newblock The supercritical phase of percolation is well behaved.
\newblock {\em Proc. R. Soc. Lond.}, Ser. A430:439--457, 1990.

\bibitem[Gri99]{gr99}
G.~Grimmett.
\newblock {\em Percolation}.
\newblock Springer: Berlin, Germany, 1999.

\bibitem[GRS22]{grs21}
S.~Goswami, P.-F. Rodriguez, and F.~Severo.
\newblock On the radius of {G}aussian free field excursion clusters.
\newblock {\em Ann. Probab.}, 50(5):1675--1724, 2022.

\bibitem[IK91]{ik91}
M.~Isichenko and J.~Kalda.
\newblock Statistical topography. {I}. {F}ractal dimension of coastlines and
  number-area rule for {I}slands.
\newblock {\em J. Nonlinear Sci.}, 1:255--277, 1991.

\bibitem[Jan97]{jan97}
S.~Janson.
\newblock {\em {Gaussian Hilbert spaces}}, volume 129.
\newblock Cambridge: Cambridge University Press, 1997.

\bibitem[JGRS20]{jgrs20}
N.~Javerzat, S.~Grijalva, A.~Rosso, and R.~Santachiara.
\newblock Topological effects and conformal invariance in long-range correlated
  random surfaces.
\newblock {\em SciPost Phys.}, 9(050), 2020.

\bibitem[Lan72]{lan72}
N.~Landkof.
\newblock {\em Foundations of Modern Potential Theory}.
\newblock Springer-Verlag, 1972.

\bibitem[LO13]{lo13}
N.~Leonenko and A.~Olenko.
\newblock Tauberian and {A}belian theorems for long-range dependent random
  fields.
\newblock {\em Methodol. Comput. Appl. Probab.}, 15:715--742, 2013.

\bibitem[Men86]{men86}
M.~Menshikov.
\newblock Coincidence of critical points in percolation problems.
\newblock {\em Sov. Math. Dokl.}, 33:856--859, 1986.

\bibitem[MRV23]{mrv20}
S.~Muirhead, A.~Rivera, and H.~Vanneuville.
\newblock The phase transition for planar {G}aussian percolation models without
  {FKG}.
\newblock {\em Ann. Probab.}, 51(5):1785--1829, 2023.
\newblock With an appendix by L. K\"{o}hler-Schindler.

\bibitem[MS83]{ms83b}
S.~Molchanov and A.~Stepanov.
\newblock Percolation in random fields. {II}.
\newblock {\em Theor. Math. Phys.}, 55(3):592--599, 1983.

\bibitem[Mui22]{m22}
S.~Muirhead.
\newblock Percolation of long-range correlated {G}aussian fields {II}.
  {S}harpness of the phase transition.
\newblock {\em to appear in Ann. Probab.}, available at arXiv:2206.10724, 2022.

\bibitem[MV20]{mv20}
S.~Muirhead and H.~Vanneuville.
\newblock The sharp phase transition for level set percolation of smooth planar
  {G}aussian fields.
\newblock {\em Ann. I. Henri Poincar\'e Probab. Stat.}, 56(2):1358--1390, 2020.

\bibitem[Nit18]{Nit18}
M.~Nitzschner.
\newblock Disconnection by level sets of the discrete {G}aussian free field and
  entropic repulsion.
\newblock {\em Electron. J. Probab.}, 23:1--21, 2018.

\bibitem[NS16]{ns16}
F.~Nazarov and M.~Sodin.
\newblock Asymptotic laws for the spatial distribution and the number of
  connected components of zero sets of {G}aussian random functions.
\newblock {\em J. Math. Phys. Anal. Geo.}, 12(3):205--278, 2016.

\bibitem[NS20]{NS20}
M.~Nitzschner and A.-S. Sznitman.
\newblock Solidification of porous interfaces and disconnection.
\newblock {\em J. Eur. Math. Soc.}, 22(8):2629--2672, 2020.

\bibitem[Pit82]{pit82}
L.~Pitt.
\newblock Positively correlated normal variables are associated.
\newblock {\em Ann. Probab.}, 10(2):496--499, 1982.

\bibitem[PR15]{pr15}
S.~Popov and B.~R\'{a}th.
\newblock On decoupling inequalities and percolation of the excursion sets of
  the {G}aussian free field.
\newblock {\em J. Stat. Phys.}, 159:312--320, 2015.

\bibitem[PS22]{PS21a}
C.~Panagiotis and F.~Severo.
\newblock Analyticity of {G}aussian free field percolation observables.
\newblock {\em Commun. Math. Phys.}, 396:187--223, 2022.

\bibitem[PT15]{pt15}
S.~Popov and A.~Teixeira.
\newblock Soft local times and decoupling of random interlacements.
\newblock {\em J. Eur. Math. Soc.}, 17(10):2545--2593, 2015.

\bibitem[RS13]{rs13}
P.-F. Rodriguez and A.-S. Sznitman.
\newblock Phase transition and level-set percolation for the {G}aussian free
  field.
\newblock {\em Comm. Math. Phys.}, 320(2):571--601, 2013.

\bibitem[RV20]{rv20}
A.~Rivera and H.~Vanneuville.
\newblock The critical threshold for {B}argmann-{F}ock percolation.
\newblock {\em Ann. Henri Lebesgue}, 3:169--215, 2020.

\bibitem[Sev22]{s21}
F.~Severo.
\newblock Sharp phase transition for {G}aussian percolation in all dimensions.
\newblock {\em Ann. Henri Lebesgue}, 5:987--1008, 2022.

\bibitem[SW01]{sw01}
S.~Smirnov and W.~Werner.
\newblock Critical exponents for two-dimensional percolation.
\newblock {\em Math. Res. Lett.}, 8(5):729--744, 2001.

\bibitem[Szn15]{Szn15}
A.-S. Sznitman.
\newblock Disconnection and level-set percolation for the {G}aussian free
  field.
\newblock {\em J. Math. Soc. Japan}, 67(4):1801--1843, 2015.

\bibitem[Szn17]{Szn17}
A.-S. Sznitman.
\newblock Disconnection, random walks, and random interlacements.
\newblock {\em Probab. Theory Related Fields}, 167(1-2):1--44, 2017.

\bibitem[Szn19]{Szn19b}
A.-S. Sznitman.
\newblock On macroscopic holes in some supercritical strongly dependent
  percolation models.
\newblock {\em Ann. Probab.}, 47(4):2459--2493, 2019.

\bibitem[Szn21]{Szn20}
A.-S. Sznitman.
\newblock Excess deviations for points disconnected by random interlacements.
\newblock {\em Probab. Math. Phys.}, 2(3):563--611, 2021.

\bibitem[Szn23a]{Szn19a}
A.-S. Sznitman.
\newblock On bulk deviations for the local behavior of random interlacements.
\newblock {\em Ann. Scient. Ec. Norm. Sup.}, 56(3):801--858, 2023.

\bibitem[Szn23b]{Szn21}
A.-S. Sznitman.
\newblock On the cost of the bubble set for random interlacements.
\newblock {\em Invent. Math.}, 233:903--950, 2023.

\bibitem[Wei84]{w84}
A.~Weinrib.
\newblock Long-range correlated percolation.
\newblock {\em Phys. Rev. B}, 29(1):387, 1984.

\end{thebibliography}

\end{document}